\documentclass[10pt]{article}
\usepackage{amsmath,amssymb,amsthm}
\usepackage{amsfonts}

\usepackage{mathtools}
\usepackage{color}
\usepackage[nohead,margin=1.0in]{geometry}
\usepackage{booktabs}
\usepackage{graphicx,caption,subcaption}
\usepackage[ruled,vlined]{algorithm2e}
\usepackage{verbatim}
\usepackage{url}
\usepackage{grffile}
\usepackage{float}

\theoremstyle{plain}
\newtheorem{theorem}{Theorem}[section]
\newtheorem{remark}{Remark}[section]
\newtheorem{definition}{Definition}[section]
\newtheorem{lemma}{Lemma}[section]
\newtheorem{assumption}{Assumption}[section]
\newtheorem{proposition}{Proposition}[section]
\newtheorem{corollary}{Corollary}[section]

\numberwithin{equation}{section}

\newcommand{\N}{\mathbb{N}}
\newcommand{\R}{\mathbb{R}}
\newcommand{\C}{\mathbb{C}}
\newcommand{\K}{\mathbb{K}}

\renewcommand{\d}{\mathrm{d}}

\newcommand{\commentout}[1]{}

\usepackage{csquotes} 
\usepackage{enumitem} 
\newlist{enumerateAssumption}{enumerate}{2} 
\setlist[enumerateAssumption]{label=(\roman*),parsep = 0pt, topsep = 0pt} 
\newcommand{\D}{\mathcal{D}}

\newcommand{\penc}{{\mathcal{R}_c}}
\newcommand{\pens}{{\mathcal{R}_s}}

\DeclareMathOperator{\sign}{sgn}

\newcommand{\J}{J_{\alpha,\beta,\mu}^{\delta,\epsilon}}

\title{Joint super-resolution image reconstruction and parameter identification in imaging operator: Analysis of bilinear operator equations, numerical solution, and application to magnetic particle imaging}
\author{Tobias Kluth\thanks{Center for Industrial Mathematics, University of Bremen, 28357 Bremen, Germany (\texttt{tkluth@math.uni-bremen.de, cbathke@math.uni-bremen.de, pmaass@math.uni-bremen.de})} \and
Christine Bathke$^\ast$ \and
Ming Jiang\thanks{Department of Information and Computing Sciences, School of Mathematical Sciences, Peking University, Beijing 100871, China (\texttt{ming-jiang@pku.edu.cn})} \and
\and 
Peter Maass$^\ast$
}

\begin{document}
\maketitle

\begin{abstract}
One important property of imaging modalities and related applications is the resolution of image reconstructions which relies on various factors such as instrumentation or data processing.
Restrictions in resolution can have manifold origins, e.g., limited resolution of available data, noise level in the data, and/or inexact model operators.
In this work we investigate a novel data processing approach suited for inexact model operators. 
Here, two different information sources, high-dimensional model information and high-quality measurement on a lower resolution, are comprised in a hybrid approach.
The joint reconstruction of a high resolution image and parameters of the imaging operator are obtained by minimizing a Tikhonov-type functional.
The hybrid approach is analyzed for bilinear operator equations with respect to stability, convergence, and convergence rates. 
We further derive an algorithmic solution exploiting an algebraic reconstruction technique. The study is complemented by numerical results ranging from an academic test case to image reconstruction in magnetic particle imaging.

\end{abstract}

\section{Introduction}

    Enhancing the resolution in   image reconstructions is a never ending challenge in medical imaging.
    Diagnostic quality or the potential for novel applications such as molecular or multi-modal imaging crucially depend on advances in image resolution  \cite{Greenspan2009,Isaac2015SuperRT}.
       Obtaining a better resolution can be achieved by either improving the measurement technology or by advances in data processing. In this paper we will address the second approach in a setting which is motivated by  the particular case of magnetic particle imaging (MPI) to be introduced later in this section.

      For motivation we start  with a general task  of an inverse problem of reconstructing a two-dimensional image $c : \Omega \rightarrow \R$ with $\Omega =  [0,1]^2 $  from measured data $u: {\cal{D}}\rightarrow  \R^{d}$ with ${\cal{D}} \subset \R^{\tilde d}$, i.e.,  $u$ is a $d$-dimensional data set defined on a $\tilde d$-dimensional domain. Image and data are related by a measurement process $A:  X \rightarrow Z$, i.e., $Ac \sim u$, for some suitably defined function spaces $ X, Z$, where $X$ is called image space and $Z$ data space.
      The general task in image reconstruction is to determine an approximation of $c$ from given $u$ and $A$.
      In all applications only a sampled and noisy version $u_\delta$ of $u$ is measured and a discretized reconstruction of $c$ is sought after \cite{Scherzer2008,Siltanen2012,Louis1989,Schuster2012,rieder2003,Burger_2018}.
    
    The achievable resolution in medical image reconstruction can be limited for several reasons:
    \begin{itemize}
    \item the available data $u_\delta$ has  limited resolution,
    \item the noise level of the data prohibits high quality reconstructions
    \item the model operator is inexact and allows only for a limited spatial accuracy in the image space.
    \end{itemize}
    For an overview of the achievable resolution of different medical imaging technologies see \cite{KnoppOnline2016}.
    In the present paper we address a particular setting, which is motivated by modeling the inversion process in magnetic particle imaging (MPI), see \cite{Kluth2018b,Knopp2010e}. The MPI problem is the reconstruction of an unknown distribution of nanoparticle concentration $c: \Omega \rightarrow \R$ inside the body from voltage measurements $u: [0,T] \rightarrow \R^d$ induced by an electromagnetic field of the magnetized nanoparticles. 
    Measurements are commonly obtained from three measurement coil units, which results in $d=3$. 
    Alternatively, the measurement signals are often transformed in Fourier-space, which then results in $\cal{D}=\N$ and $d=6$ after separating real and imaginary parts of the data. We will use this application as motivation but consider the super resolution problem in a general variational setting.
    
    The basic analytical model of MPI and other imaging reconstruction problems, \cite{Engl2000,Louis1989,rieder2003}, is given by a linear integral equation
    \begin{equation}\label{eq:general_problem_intro}
    \int_\Omega \ s(x,t) c(x) \d x =u_\delta (t) \ .
\end{equation}
   Here, $s(x,t)$ is the system or point spread function.
   The system function $s$ can either be determined experimentally by placing a delta probe at position $x_0$, i.e., $c(\cdot)=\delta(\cdot - x_0)$, and measuring the resulting data $u$, which yields $ s(x_0,t)=u(t)$,  or it can be derived from first-principle physical-mathematical modeling.
   However, taking MPI as motivation, we encounter the situation that the experimental approach is very delicate and time consuming. Hence, the experimental approach will yield the precise system function $s(x_i,t)$ but only for a small set of sample points and with a very coarse resolution. This model is called $s_\mathrm{calib}$. 
   On the other hand, a physical-mathematical model of $s$ can be evaluated with arbitrary resolution; however, several models are not suitable for the purpose of image reconstruction as they neglect effects such as particle magnetization dynamics, size effects of the nanoparticles or particle-particle interaction \cite{Kluth2018a,haegele2012magnetic,Knopp2017}. More recently, progress has been made in the development of a suitable model \cite{KluthSzwargulskiKnopp2019}.
   This model will be called $s_\mathrm{mod}$.
   
   Similar situations of having a low-quality, high-resolution and a high-quality, low-resolution model occur in several other imaging applications. E.g., this is typical in bi-modal imaging \cite{Ming2014,Arridge2011} or in molecular imaging (MALDI imaging) \cite{Heeren2019}.
   More specifically, we consider applications where  high-quality
calibration measurements can be performed on a rather coarse grid describing the image-measurement relationship accurately.
Unfortunately, the improved accuracy is then accompanied by a restriction in resolution. 
   One important question is, how to connect this experimental "expert" knowledge with a model based approach.  
One possible answer is to simultaneously determine parameters of the forward operator when reconstructing the desired image.

   In this situation, it is natural to regard $s$ as an additional variable and to consider the bilinear inverse problem with operator 
   \begin{equation}\label{eq:def_operator_intro}
   \begin{tabular}{ccl}
    $B \ :$    & $X \times Y$ & $\rightarrow Z$ \\
        & $(c,s)$ & $\mapsto  \int_\Omega$ \ $s(x,\cdot) c(x) \mathrm{d}x \ $.
   \end{tabular}
\end{equation}
   $Y$ denotes a suitable function space for modeling system functions $s$. In addition we need an operator $P$ linking a high resolution system function to a low resolution approximation, such as a projection operator. Both, $Y$ and $P$
    will be specified in the next section.
   
   Introducing suitable penalty functionals, the inverse problem of reconstructing simultaneously an update for $s_\mathrm{mod}$ and a reconstruction of $c$ can be formulated as a Tikhonov regularization scheme
   defined by the functional
   \begin{equation}\label{Tikhonov_intro}
   J^\delta_{\alpha,\beta,\gamma,\mu}(c,s)=  { {\frac{1}{2} \| B(c,s)-u_\delta \|_Z^2}} +  { {\frac{\gamma}{2} \|s-s_\mathrm{mod}\|_Y^2 + \frac{\mu}{2} \| P(s)-s_\mathrm{calib}\|_Y^2 +\alpha \mathcal{R}_c(c) + \beta \mathcal{R}_s(s)}} \ .
\end{equation}
In contrast to the sole image reconstruction problem, the additional degree of freedom allows to compensate potential errors in the modeled system function $s_\mathrm{mod}$ but also causes a largely underdetermined problem which requires a priori knowledge on $s$ being the parameters of the forward operator. 

For system functions satisfying $s(x,t)=s(x-t)$ this setting is identical to the well known problem of blind deconvolution,  see \cite{Lv_2018,Bleyer:2013cw, Justen2006} and the references therein for an overview of related regularization approaches.
Indeed,  \cite{Bleyer:2013cw} has partially influenced the approach of the present paper. 

Also, there exists a large and somewhat complete body of literature related to general inverse problems in Hilbert and Banach space settings 
\cite{Hofmann:2007us,Grasmair:2008cy,Schuster2012,jin2012review}. \cite{Hofmann:2007us} is of particular importance for the present paper and to some extend one can regard the present paper as specifying their results   to the functional defined in \eqref{Tikhonov_intro}. 
Moreover, machine learning approaches for inverse problems have been investigated intensively in the past few years, for a review of the present state of the art see \cite{Arridge2019}.
 
The precise mathematical setting of the super-resolution problem discussed in the present paper will be defined in the next section.
   We then discuss the analytic properties of the Tikhonov functional as well as the regularization properties of its minimizers in Section 3.
   A Kaczmarz-type algorithm minimizing the considered Tikhonov functional is presented in Section 4.
   Finally, we apply this approach to an academic test problem as well as to real data obtained from an MPI experiment in Section 5 and conclude with discussions in Section 6.



\section{Problem formulation and discussion}

In the following we define the mathematical setting of the problem, formulate a variational approach for its solution, and distinguish possible perspectives on the resulting problem. 
Let $X,Y,Z$ be Hilbert spaces and let $B: X\times Y \rightarrow Z$, $(c,s) \mapsto B(c,s)$ denote a bilinear operator as defined in \eqref{eq:def_operator_intro} where $X\times Y$ is equipped with the canonical inner product $\langle (c_1,s_1), (c_2,s_2) \rangle_{X\times Y} = \langle c_1, c_2 \rangle_X + \langle s_1, s_2 \rangle_Y$.

  The problem of interest is  
 to obtain an approximate solution $(c,s)\in X\times Y$ from noisy measurements $u_\delta\in Z$. As usual we assume that a physically exact solution $(c^\ast,s^\ast)$ exists and that noisy data $u_\delta$ satisfying $\|u_\delta - u^\ast\|_Z \leq \delta$, where  
\begin{equation}\label{eq:general_problem}
  u^\ast=  B(c^\ast,s^\ast) 
\end{equation}
is the true data.
For given $u_\delta$ and $B$, the inverse problem is to determine a suitable approximation for $(c^\ast,s^\ast)$.

In the setting of the present paper we address the problem of super-resolution by including two pieces of information for the system function $s$. We assume that we have two different approximations of the true system function $s^*$:
$s_\mathrm{mod}$ is obtained from a theoretical but incomplete model in the original infinite-dimensional (or high-dimensional) space (type A), $s_\mathrm{calib}$ is obtained in a high-quality calibration procedure but on a finite-dimensional (or lower-dimensional) subspace (type B).   We further distinguish two sub-cases for $s_\mathrm{mod}$.
\begin{itemize}
    \item[(A)] Let $s_\mathrm{mod}, s_{\mathrm{mod},\epsilon}\in Y$. We   either assume that a fixed  model $s_\mathrm{mod}$ is given or that a model hierarchy  $s_{\mathrm{mod},\epsilon}$ with varying accuracy $\epsilon$ is available:
    \begin{itemize}
        \item[(i)] $s_\mathrm{mod}$ is used as a high-resolution reference model of limited accuracy for  the system function.  This will be used for the  formulation of an additional penalty term in a Tikhonov regularization scheme.  
        \item[(ii)] $s_{\mathrm{mod},\epsilon}$ is assumed to be a high-dimensional  model, which can  be obtained with different levels of accuracy   fulfilling $\| s_{\mathrm{mod},\epsilon} - s^\ast \|_Y \leq \epsilon$. This type of information is more suitable for formulating an alternative discrepancy term in a Tikhonov functional.  
    \end{itemize}
    \item[(B)] The low-resolution, high-qualtity  approximation  
    $s_\mathrm{calib}$ of the system functions $s^*$ is modeled by an element 
    in 
    $Y_n \subset Y$, where $Y_n$ is a finite-dimensional space (if $Y$ is already $m$-dimensional with $m<\infty$, let $n<m$).  Let $P:Y \rightarrow Y_n$ be a linear and bounded operator mapping onto $Y_n$ and we assume that $s_\mathrm{calib} \sim P(s^*)\in Y_n$ is an almost perfect but low-dimensional observation of the true system function. This will serve as a reference for $P(s)$ in a penalty term.
    From an application point of view this can be interpreted for example as an observation on a coarse spatial grid. 
    
\end{itemize}

\begin{remark}
Note that we do not explicitly require $P$ to be a projection operator, which would imply $P^2=P$  and $\lVert P \rVert = 1$ for an orthogonal projection. The above assumption, that $P$ is linear and bounded, will be sufficient for the following theoretical analysis.
\end{remark}

Taking into account both sources of information, we formulate a Tikhonov-functional  for $(c,s)$  with a multi-criteria penalty term. 
For case A(i) we consider
\begin{equation}\label{eq:generalFunctional1}
   J^\delta_{\alpha,\beta,\gamma,\mu}(c,s)= \underset{=:D_1((c,s),u_\delta)}{\underbrace{\frac{1}{2} \| B(c,s)-u_\delta \|_Z^2}} + \underset{=:(\alpha,\beta,\gamma,\mu)^t \mathcal{R}_1(c,s)}{\underbrace{\frac{\gamma}{2} \|s-s_\mathrm{mod}\|_Y^2 + \frac{\mu}{2} \| P(s)-s_\mathrm{calib}\|_Y^2 +\alpha \mathcal{R}_c(c) + \beta \mathcal{R}_s(s)}}
\end{equation}
  with $\alpha,\beta,\gamma,\mu \geq 0$ and 
  $$\mathcal{R}_1(c,s) = \left(
  \mathcal{R}_c(c), 
  \mathcal{R}_s(s) ,
  \frac{1}{2  } \|s-s_\mathrm{mod}\|_Y^2 , \frac{1}{2  } \| P(s)-s_\mathrm{calib}\|_Y^2    \right)^t $$
  where $\mathcal{R}_c: X \rightarrow \R_+$ and $\mathcal{R}_s: Y \rightarrow \R_+$ are proper, convex, and weakly lower semi-continuous penalty terms with respect to $c$ and $s$.
  The choice of the regularization parameters $(\alpha, \beta, \gamma,\mu)$ is crucial, of course, and in particular we consider the case, that their ratios $(\beta/\alpha, \gamma/\alpha, \mu/\alpha)$ are fixed.

For A(ii) we fix  $\gamma > 0$ and consider
\begin{equation}\label{eq:generalFunctional2}
   J^{\delta,\epsilon}_{\alpha,\beta,\mu}(c,s)= \underset{=:D_2((c,s),(u_\delta,s_{\mathrm{mod},\epsilon}))}{\underbrace{\frac{1}{2} \| B(c,s)-u_\delta \|_Z^2 + \frac{\gamma}{2} \|s-s_{\mathrm{mod},\epsilon}\|_Y^2}} + \underset{=:(\alpha, \beta, \mu)^t \mathcal{R}_2(c,s)}{\underbrace{\frac{\mu}{2} \| P(s)-s_\mathrm{calib}\|_Y^2 +\alpha \mathcal{R}_c(c) + \beta \mathcal{R}_s(s)}}
\end{equation}
where $\alpha,\beta,\mu,\mathcal{R}_c, \mathcal{R}_s$ and 
$\mathcal{R}_2$ are analogously chosen to case A(i).

\begin{remark}
The precise definition of suitable penalty terms $\mathcal{R}_c$ and $ \mathcal{R}_s$ depends on the desired type of reconstruction. Typical choices are TV- \cite{schoenlieb2017,Bredies2019,Jiang2015} or $L_p$-norms with $p=1,2$ or combinations of such norms.
\end{remark}

\begin{remark}
The general nature of the operator $P$ might allow alternative strategies to obtain higher resolution potentially without model knowledge, i.e., $\gamma=0$. This requires carefully selected choices of $P$ and $\mathcal{R}_s$, for example, specific sampling patterns and a sparsity constraint in a DCT basis have been exploited for magnetic particle imaging \cite{ilbey2019fast}. 
\end{remark}

\section{Problem analysis}
In this section we discuss the basic analytic properties of the Tikhonov functionals \eqref{eq:generalFunctional1} and  \eqref{eq:generalFunctional2}
and of their minimizers. The analysis rests on the framework introduced in \cite{Hofmann:2007us} for  non-linear inverse problems in a general function space setting and on \cite{Justen2006}, which discusses bilinear operator equations in more detail.

We first discuss the assumptions on the bilinear operator $B$ needed for obtaining existence and convergence results. We then discuss the cases A(i) and A(ii) separately.

\subsection{Assumptions on $B$}
As already stated, the analysis of this paper is based on the results of \cite{Hofmann:2007us,Justen2006}
and we will use the same assumptions as introduced in these papers.
\begin{assumption}\label{ass:operatorB}
Let $X,Y,Z$ be Hilbert spaces and let $B: X\times Y \rightarrow Z$, $(c,s) \mapsto B(c,s)$ be a bilinear operator where $X\times Y$ is equipped with the canonical inner product $\langle (c_1,s_1), (c_2,s_2) \rangle_{X\times Y} = \langle c_1, c_2 \rangle_X + \langle s_1, s_2 \rangle_Y$. 
\begin{itemize}
    \item[(i)] There exists a constant $C>0$ such that $\| B(c,s) \|_Z \leq C \| c \|_X \| s \|_Y$ for all $(c,s)\in X \times Y$.
    \item[(ii)] $B$ is sequentially weak-weak continuous, i.e., for any sequence $\{(c^k,s^k)\}_{k\in\N}\subset X\times Y$, $(c^\ast,s^\ast)\in X\times Y$  with $(c^k,s^k) \rightharpoonup (\bar c,\bar s )$ it holds $B(c^k,s^k) \rightharpoonup B(\bar c,\bar s)$.
    In the following we use the slightly shorter term of a weakly continuous operator.
\end{itemize}
\end{assumption}

\begin{remark}
For a fixed $s\in Y$ the resulting \textit{imaging operator} is given by $A:X \rightarrow Z$, $c\mapsto B(c,s)$ and the imaging problem then is to compute a $c\in X$ for a given $u_\delta \in Z$ by solving
\begin{equation}
    A(c) \sim u_\delta.
\end{equation}

Let $s_\mathrm{calib}\in Y_n \subset Y$ be obtained from a true $s^*\in Y$ by $s_\mathrm{calib}=P(s^*)$. Such a projection of a true high-resolution $s^*$ may be derived in a calibration procedure. Then the resulting \textit{reduced imaging operator} $A_n: X \rightarrow Z $, $c\mapsto B(c,P(s^*))=B(c,s_\mathrm{calib})$ allows the formulation of the reduced imaging problem, i.e., finding $c\in X$ for $u_\delta \in Z$ by solving
\begin{equation}
    A_n(c)\sim u_\delta.
\end{equation}
Note that depending on the actual definition of $P$ and $B$ it might be possible to define the reduced imaging operator on subspaces $\tilde{X}\subset X$ and $\tilde{Z} \subset Z$. 
\end{remark}

\begin{remark}\label{rem:example_MPI}
Two simple examples for bilinear operators $B$ are a convolution operator or a Fredholm integration operator of the first kind, where $s$ is either the convolution or a general integral kernel. The latter one can be used to describe the general MPI setup, i.e., let $\Omega\subset \R^3$ be a bounded domain and \mbox{$I=(0,T)$}, \mbox{$0<T<\infty$}, is the time interval in which a measurement is obtained. Choosing the function spaces \mbox{$X=L^2(\Omega)$}, $Y=L^2(\Omega \times I)$, and $Z=L^2(I)$ the bilinear operator of interest is given by
\begin{equation} \label{eq:example_integral_operator}
    B(c,s)(t)= \int_\Omega c(x) s(x,t) \d x,
\end{equation}
which describes the relation between nanoparticle concentration $c$ and voltage measurement $B(c,s)$ for one receive coil unit.
\end{remark}


Note that the subsequent theory is build on the assumption that the operator $B$ is weakly continuous which is a weaker assumption as for example the strong continuity used in \cite{Bleyer:2013cw} for a special case of our functional. Nevertheless, the function space setup in Remark \ref{rem:example_MPI} which is adapted from Example 1  in \cite{Bleyer:2013cw} is not sufficient to show weak continuity of a general $B$ which also implies that strong continuity does not hold.
The following two lemmata show two potential adaptations of the examples in Remark \ref{rem:example_MPI}, which then fulfill the stated assumptions. 
One can either change the infinite-dimensional function space setup by utilising embedding theorems or alternatively use a finite-dimensional adaptation.

\begin{lemma}\label{lem:weaklyB}
Let $\Omega\subset\R^{n}$ and $I\subset\R^{m}$ be bounded domains. Let $X=L^2(\Omega)$, $Y=H_0^s(\Omega \times I)$, and $Z=L^2(I)$ for $s>0$. Then the operator $B:X\times Y \rightarrow Z$, $B(c,s)(t)=\int_\Omega c(x) s(x,t) \d x$ is weakly continuous.
\end{lemma}
\begin{proof}
    Let $\{(c^k,s^k)\}_{k\in\N}\subset X\times Y$, $(c^\ast,s^\ast)\in X\times Y$  with $(c^k,s^k) \rightharpoonup (c^\ast,s^\ast)$. For arbitrary $\phi \in L^2(I)$ let $\psi^k(x)=\int_I s^k(x,t) \phi(t) \d t$ and $\psi^\ast(x)=\int_I s^\ast(x,t) \phi(t) \d t$. We then obtain 
    \begin{equation}
    \| \psi^k - \psi^\ast \|_{L^2(\Omega)} \leq  \|s^k -s^\ast \|_{L^2(\Omega \times I)}  \| \phi \|_{L^2(I)}. 
    \end{equation}
    Weak convergence of $s^k \rightharpoonup s^\ast $ and the compact embedding $H_0^s(\Omega \times I) \hookrightarrow L^2(\Omega \times I)$ imply strong convergence in $L^2(\Omega \times I)$. Thus $\psi^k \rightarrow \psi^\ast$ in $L^2(\Omega)$. 
    Then we consider for arbitrary $\phi \in L^2(I)$
    \begin{align*} 
    | \langle B(c^k,s^k) - B(c^\ast, s^\ast), \phi \rangle | \leq  \left|\int_\Omega (c^k(x) - c^\ast (x) ) \psi^k(x) \d x \right| + \left|\int_\Omega \int_I c^\ast(x) \phi(t) (s^k(x,t) -s^\ast (x,t) ) \d x \d t \right|
    \end{align*}
    As $c^k \rightharpoonup c^\ast$ and $\psi^k \rightarrow \psi^\ast$ in $L^2(\Omega)$, the first term converges to zero.
    Convergence to zero of the second summand follows immediately from $s^k \rightharpoonup s^\ast$ and $(c^\ast \phi) \in L^2(\Omega \times I)$ defining a linear functional.
 As this holds for any $\phi \in L^2(I)$ this concludes the proof.
\end{proof}
Weak continuity can also be proved in a finite dimensional setting, which e.g., can be justified by using a finite dimensional approximation of a compact operator $B$ using its singular value decomposition.
\begin{lemma}
Let $\Omega\subset\R^{n}$ and $I\subset\R^{m}$ be bounded domains. $X=L^2(\Omega)$, $Z=L^2(I)$  and let $ \{\Psi_k\}_{k=1}^K \subset L^2(\Omega \times I)$ denote a set of $K$ orthonormal functions in $ L^2(\Omega \times I)$. 
 For $Y=\R^K$ and $\kappa=(\kappa_1,.., \kappa_K)^t \in Y$ 
  we define
the operator $B:X\times Y \rightarrow Z$ by $B(c,\kappa)(t)=\int_\Omega c(x) \sum_{k=1}^K \kappa_k \Psi_k(x,t) \d x$. Then $B$ is weakly continuous.
\end{lemma}
\begin{proof}
The assertion follows analogously to the proof of Lemma \ref{lem:weaklyB} while weak convergence with respect to $Y$ directly implies strong convergence due to the finite dimension. 
\end{proof}

\begin{remark}
Additionally, one may consider a nonlinear dependence of $B$ on the system function $s$ assuming weak sequantial closedness of the operator. 
Theoretical results on minimizer existence, stability and convergence may be derived in an analogous way. However, the results on convergence rates rely on the bilinearity of $B$. In the present paper we stay with the bilinear setting, the extension to a nonlinear dependence is beyond the scope of this paper and remains future work.
\end{remark}


\subsection{Basic properties}
\label{sec:basic_properties}
We further provide some basic properties of the functionals appearing in our problem formulation.

\begin{lemma}\label{lemma:projectionPenalty}
	Let $P\colon Y \rightarrow Y_n\subset Y$ be a linear and bounded operator, $P \ne 0$, and $s_\mathrm{calib} \in Y_n$. 
	Then the functional
	\begin{equation}
		\begin{split}
			T \colon Y &\rightarrow [0,\infty] \\
			s &\mapsto \lVert P(s)-s_\mathrm{calib}\rVert^2
		\end{split}
	\end{equation} is proper, convex and weakly lower semi-continuous.
\end{lemma}
See Appendix \ref{app:appendix1} for a proof of this Lemma.

In general, an ill-posed inverse problem may have multiple solutions, hence we introduce  the usual concept of an $\penc$-minimizing solution.
\begin{definition}
For $s^\ast \in Y$ and $u^\ast\in range(B) \subset Z$
an element $c^\dag\in X$ is called an $\penc$-minimizing solution if
\begin{equation}
	c^\dag = \arg \min_c \{ \penc(c) \mid B(c,s^\ast) = u^\ast \}.
\end{equation}
\end{definition}

In the next lemma we state the derivative of a bilinear operator in a Hilbert space setting and an inequality to be used later.
\begin{lemma}\label{lemma:FrechetDerivative1} 
For a bilinear operator $B\colon X \times Y \rightarrow Z$ with $\|B(c,s)\|_Z \leq C \|c\|_X \|s\|_Y$ the Fr\'{e}chet derivative at  $(c,s)\in X\times Y$ is given by
	\begin{equation}
		B^\prime(c,s)(x,y) = B(c,y)+B(x,s)
	\end{equation}
	for any $(x,y) \in X \times Y $. The residual of the first degree Taylor expansion satisfies
	\begin{equation} \label{eq:BTaylorRemainder}
		\| B(c+x,s+y)-B(c,s)-B^\prime(c,s)(x,y) \|_Z \leq \frac{C}{2} \| (x,y) \|_{X \times Y}^2.
	\end{equation}\end{lemma}
	Again, this result follows from standard arguments in functional analysis. For completeness a proof is included in Appendix \ref{app:appendix1}.

\subsection{Setup A(i)}
In this setting we exploit a high-resolution reference $s_\mathrm{mod}$ to formulate an additional penalty term for the system function $s$. 
We thus consider the functional
\begin{equation}\label{eq:generalFunctional1_2}
   J^\delta_{\alpha,\beta,\gamma,\mu}(c,s)= \underset{=:D_1((c,s),u_\delta)}{\underbrace{\frac{1}{2} \| B(c,s)-u_\delta \|_Z^2}} + \underset{=:(\alpha,\beta,\gamma,\mu)^t \mathcal{R}_1(c,s)}{\underbrace{\frac{\gamma}{2} \|s-s_\mathrm{mod}\|_Y^2 + \frac{\mu}{2} \| P(s)-s_\mathrm{calib}\|_Y^2 +\alpha \mathcal{R}_c(c) + \beta \mathcal{R}_s(s)}}
\end{equation}
including four regularization parameters $(\alpha,\beta,\gamma,\mu)$. 
For fixed ratios of regularization parameters, i.e.,  $\nu_1=\mu/ \alpha$, $\nu_2= \beta/\alpha$, and $\nu_3=\gamma/\alpha$ the desired regularization properties follow immediately from the general theory in \cite{Hofmann:2007us}. In the following theorem we consider a slightly more general setting, where these ratios are only obtained asymptotically.

\begin{theorem}\label{thm:existence_etc_1}
	Let Assumption \ref{ass:operatorB} be fulfilled. Let $\mathcal{R}_c: X \rightarrow \R_+$ and $\mathcal{R}_s: Y \rightarrow \R_+$ be proper, convex, and weakly lower semi-continuous. Let $P:Y \rightarrow Y_n\subset Y$ be a  linear and bounded operator.  Then the following holds:
	\begin{itemize}
	    \item[(i)] 	The functional $J^\delta_{\alpha,\beta,\gamma,\mu}$ as defined in~\eqref{eq:generalFunctional1_2} has a minimizer. 
	    
	    \item[(ii)] (Continuity for fixed regularization parameters) 	Let the regularization parameters $\alpha, \beta,\gamma, \mu >0$ be fixed. Consider the sequence $(u_{\delta_j})_{j\in\N}$ with $u_{\delta_j} \rightarrow u_\delta$. Let $(c^j,s^j)$ denote a minimizer of $J_{\alpha,\beta,\gamma,\mu}^{\delta_j}$ with noisy $u_{\delta_j}$. Then there exists a weakly convergent subsequence of $(c^j,s^j)$ and the limit  of every weakly convergent subsequence is a minimizer $(\tilde{c},\tilde{s})$ of the functional $J_{\alpha,\beta,\gamma,\mu}^{\delta}$. Moreover, for each weakly convergent subsequence $(c^i,s^i)$, $(\alpha,\beta,\gamma,\mu)^t\mathcal{R}_1((c^i,s^i)) \rightarrow (\alpha,\beta,\gamma,\mu)^t\mathcal{R}_1((\tilde{c},\tilde{s}))$.
	    
	    \item[(iii)] (Convergence for diminishing regularization parameters)	Assume that the data sequence $(u_{\delta_j})_j$ with $\|u_{\delta_j}-u^*\| \leq \delta_j$ for $\delta_j \rightarrow 0$ is given. Assume that the regularization parameters are chosen according to the noise level such that $\alpha_j=\alpha(\delta_j)$, $\beta_j = \beta(\delta_j)$, $\gamma_j = \gamma(\delta_j)$, and $\mu_j=\mu(\delta_j)$ are monotonically decreasing and fulfill $\alpha_j\to 0$, $\beta_j \to 0$, $\gamma_j \to 0$, and $\mu_j \to 0$. Further assume that
		\begin{equation} \label{eq:parameterChoice1}
		\lim_{j\to\infty} \frac{\delta_j^2}{\alpha_j}=0,
		\quad
		\lim_{j\to\infty} \frac{\beta_j}{\alpha_j} = \nu_1 \quad 
		\lim_{j\to\infty} \frac{\gamma_j}{ \alpha_j} = \nu_2
		\quad \text{ and } \quad
		\lim_{j\to\infty} \frac{\mu_j}{ \alpha_j} = \nu_3
	\end{equation}
	holds for some $0<\nu_1,\nu_2, \nu_3<\infty$ such that $\nu_1\leq \beta_j / \alpha_j$, $\nu_2\leq \mu_j/\alpha_j$, and $\gamma_j/\alpha_j \leq \nu_3$.
	Let
	\begin{equation}
		(c^j,s^j)_j \coloneqq \left( c_{\alpha_j,\beta_j,\gamma_j,\mu_j}^{\delta_j},s_{\alpha_j,\beta_j,\gamma_j,\mu_j}^{\delta_j} \right)_j
	\end{equation}
	denote the minimizing sequence of~\eqref{eq:generalFunctional1_2} obtained from noisy $u_{\delta_j}$.

	Then there exists a weakly convergent subsequence of $(c^j,s^j)_j$. The limit of every weakly convergent subsequence of $(c^j,s^j)_j$ is an $(1,\nu_1,\nu_2,\nu_3)^t \mathcal{R}_1$-minimizing solution.
	    
	\end{itemize}
\end{theorem}
\begin{proof}

Assertion (i) and (ii) immediately follow from \cite[Theorems 3.1, 3.2]{Hofmann:2007us}.
For (iii) we exploit that $\alpha_j (1,\nu_1,\nu_2,\nu_3)^t \mathcal{R}_1(c,s) \leq (\alpha_j,\beta_j,\gamma_j, \mu_j)^t \mathcal{R}_1(c,s)$ for any $(c,s) \in X \times Y$. We obtain from 
\begin{equation}
    \frac12 \| B(c^j,s^j) - u_{\delta_j} \|_Z ^2 + (\alpha_j,\beta_j,\gamma_j, \mu_j)^t \mathcal{R}_1(c^j,s^j) \leq \frac12 \delta_j^2 +   (\alpha_j,\beta_j,\gamma_j, \mu_j)^t \mathcal{R}_1(c^\ast,s^\ast)
\end{equation}
that $\lim_{j \to \infty} \| B(c^j,s^j) - u_{\delta_j} \|_Z =0$ and $(1,\nu_1,\nu_2, \nu_3)^t \mathcal{R}_1(c^j,s^j) \leq \frac{ \delta_j^2}{2\alpha_j} +   \frac{1}{\alpha_j}(\alpha_j,\beta_j,\gamma_j, \mu_j)^t \mathcal{R}_1(c^\ast,s^\ast)$.
This implies $$\limsup_{j\to \infty}{(1,\nu_1,\nu_2, \nu_3)^t \mathcal{R}_1(c^j,s^j)}\leq \limsup_{j\to \infty}{\frac{1}{\alpha_j}(\alpha_j,\beta_j,\gamma_j, \mu_j)^t \mathcal{R}_1(c^j,s^j)} \leq (1,\nu_1,\nu_2, \nu_3)^t \mathcal{R}_1(c^\ast,s^\ast). $$
We thus obtain
\begin{align*}
    &\limsup_{j\to \infty}{\left(\frac12 \|B(c^j,s^j)-u_{\delta_j}\|_{Z}^2 + \alpha_0 (1,\nu_1,\nu_2, \nu_3)^t \mathcal{R}_1(c^j,s^j)\right)} \\
    \leq & \limsup_{j\to \infty}{\left( \frac12 \|B(c^j,s^j)-u_{\delta_j}\|_{Z}^2  + \alpha_j (1,\nu_1,\nu_2, \nu_3)^t \mathcal{R}_1(c^j,s^j)\right)} + \limsup_{j\to \infty}{\left( ( \alpha_0 -\alpha_j  )(1,\nu_1,\nu_2, \nu_3)^t \mathcal{R}_1(c^j,s^j)  \right) }\\  \leq & \alpha_0 (1,\nu_1,\nu_2, \nu_3)^t \mathcal{R}_1(c^\ast,s^\ast) < \infty. 
\end{align*}
The assertion (iii) then follows analogously by the remaining steps in the proof of \cite[Theorem 3.5]{Hofmann:2007us}.
\end{proof}

A first convergence rate result can be obtained by making the following general assumption.

\begin{assumption}\label{assumption:general_problem1_source}
Let Assumption \ref{ass:operatorB} be fulfilled. Further assume
	\begin{enumerateAssumption}
	\item There exists an $(1,\nu_1,\nu_2,\nu_3)^t \mathcal{R}_1$-minimizing solution $(c^\ast,s^\ast)\in X\times Y$.
	\item
	There exist constants $\kappa_1 \in [0,1)$, $\kappa_2	\geq 0$,  and a subgradient $(\xi_{c^\ast},\xi_{s^\ast})\in \partial (1,\nu_1,\nu_2,\nu_3)^t \mathcal{R}_1(c^\ast,s^\ast)$, such that 
	\begin{equation}
		\langle (\xi_{c^\ast},\xi_{s^\ast}), (c^\ast -c,s^\ast-s)\rangle \leq \kappa_1 D_{(1,\nu_1,\nu_2,\nu_3)^t \mathcal{R}_1}^{(\xi_{c^\ast},\xi_{s^\ast})} ((c,s),(c^\ast,s^\ast)) + \kappa_2 \lVert B(c,s)-B(c^\ast,s^\ast) \rVert 
	\end{equation}
	for all $(c,s)\in \{ (c,s) \in X \times Y |J^\delta_{\alpha_\mathrm{max},\alpha_\mathrm{max}\nu_1,\alpha_\mathrm{max}\nu_2,\alpha_\mathrm{max}\nu_3}(c,s)  \leq M \}$ and $M > \alpha_\mathrm{max} ((1,\nu_1,\nu_2,\nu_3)^t \mathcal{R}_1(c^\ast,s^\ast) + \delta^2/\alpha)$ for given $0<\delta, \alpha <\infty$, $\alpha \leq \alpha_\mathrm{max}$. 
	\end{enumerateAssumption} 
\end{assumption}
These are the assumptions required for \cite[Theorem 4.4]{Hofmann:2007us} and  we directly obtain the following result on convergence rates.
\begin{theorem}
Let $u_\delta \in Z$ with $\| u^\ast - u_\delta \|_Z \leq \delta $.
Let Assumption \ref{assumption:general_problem1_source} be fulfilled for all $(\delta,\alpha)$ tuples defined below.
For $0<\alpha \leq \alpha_\mathrm{max}< \infty$, $\beta=\nu_1\alpha$, $\gamma=\nu_2\alpha$, and $\mu=\nu_3\alpha$ the minimizer of the functional $J^\delta_{\alpha,\beta,\gamma,\mu}$ as defined in \eqref{eq:generalFunctional1_2} is denoted by $(c^\alpha,s^\alpha)$. Further assume $\alpha \sim \delta$. Then it holds
\begin{equation}
	D_{(1,\nu_1,\nu_2,\nu_3)^t \mathcal{R}_1}^{(\xi_{c^\ast},\xi_{s^\ast})} ((c,s),(c^\ast,s^\ast)) = \mathcal{O}(\delta ) \quad \textnormal{and} \quad
	\lVert B(c^\alpha,s^\alpha)-B(c^\ast,s^\ast) \rVert = \mathcal{O}(\delta).
\end{equation}

\end{theorem}

\subsection{Setup A(ii)}
\label{sec:setupA2}
In this setting we consider  a high-resolution approximation $s_{\mathrm{mod},\epsilon}$ of the true but unknown
system function $s^*$ satisfying $\|s_{\mathrm{mod},\epsilon}-s^*\|_Y \le \epsilon.$
As already discussed in the previous section, one can consider $s_{\mathrm{mod},\epsilon}$ as additional data in this setting. 
Hence, we fix $\gamma>0$ and consider the functional
\begin{equation}\label{eq:generalFunctional2_2}
   J^{\delta,\epsilon}_{\alpha,\beta,\mu}(c,s)= \underset{=:D_2((c,s),(u_\delta,s_{\mathrm{mod},\epsilon}))}{\underbrace{\frac{1}{2} \| B(c,s)-u_\delta \|_Z^2 + \frac{\gamma}{2} \|s-s_{\mathrm{mod},\epsilon}\|_Y^2}} + \underset{=:(\alpha, \beta, \mu)^t \mathcal{R}_2(c,s)}{\underbrace{\frac{\mu}{2} \| P(s)-s_\mathrm{calib}\|_Y^2 +\alpha \mathcal{R}_c(c) + \beta \mathcal{R}_s(s)}} . 
\end{equation}
The multi-criterial penalty term involving three regularization parameters $(\alpha, \beta, \mu)$ can be reduced to the usual single parameter setting by fixing the ratios $\nu_1=\mu/ \alpha$ and $\nu_2= \beta/\alpha$.
As in the previous section, this will allow us to use the available regularization results directly and will be utilized later in this section.
However, we start with an adaptation which allows slightly more freedom in the choice of the regularization parameters.

\begin{theorem}\label{thm:existence_etc_2}
	Let Assumption \ref{ass:operatorB} be fulfilled. Let $\mathcal{R}_c: X \rightarrow \R_+$ and $\mathcal{R}_s: Y \rightarrow \R_+$ be proper, convex, and weakly lower semi-continuous. Let $P:Y \rightarrow Y_n\subset Y$ be a  linear and bounded operator.  Then the following holds:
	\begin{itemize}
	    \item[(i)] 	The functional $\J$ as defined in~\eqref{eq:generalFunctional2_2} has a minimizer. 
	    
	    \item[(ii)] (Continuity for fixed regularization parameters) 	Let the regularization parameters $\alpha, \beta, \mu >0$ be fixed. Consider sequences $(u_{\delta_j})_{j\in\N}$ and $(s_{\epsilon_j})_{j\in\N}$  with $u_{\delta_j} \rightarrow u_\delta$ and $s_{\epsilon_j} \rightarrow s_{\mathrm{mod},\epsilon}$. Let $(c^j,s^j)$ denote a minimizer of $J_{\alpha,\beta,\mu}^{\delta_j,\epsilon_j}$ with noisy $(u_{\delta_j},s_{\epsilon_j})$. Then there exists a weakly convergent subsequence of $(c^j,s^j)$ and the limit  of every weakly convergent subsequence is a minimizer $(\tilde{c},\tilde{s})$ of the functional $\J$. Moreover, for each weakly convergent subsequence $(c^j,s^j)$ we have $$(\alpha, \beta, \mu)^t\mathcal{R}_2(c^j,s^j) \rightarrow (\alpha, \beta, \mu)^t\mathcal{R}_2(\tilde{c},\tilde{s})  \mbox{\ \ and \ \ } J_{\alpha,\beta,\mu}^{\delta_j,\epsilon_j} (c^j,s^j)\rightarrow J_{\alpha,\beta,\mu}^{\delta,\epsilon}(\tilde c, \tilde s) \ . $$
	    
	    \item[(iii)] (Convergence for vanishing noise levels)	Assume that data sequences $(u_{\delta_j})_j$ and $(s_{\epsilon_j})_j$ with $\|u_{\delta_j}-u^*\| \leq \delta_j$ and $\| s_{\epsilon_j}-s^* \| \leq \epsilon_j$ for $\delta_j \rightarrow 0$ and $\epsilon_j \rightarrow 0$ are given. Assume that the regularization parameters are chosen according to the noise level such that $\alpha_j=\alpha(\delta_j,\epsilon_j)$, $\beta_j = \beta(\delta_j,\epsilon_j)$ and $\mu_j=\mu(\delta_j,\epsilon_j)$ are monotonically decreasing and fulfill $\alpha_j\to 0$, $\beta_j \to 0$ and $\mu_j \to 0$. Further assume that
		\begin{equation} \label{eq:parameterChoice}
		\lim_{j\to\infty} \frac{\delta_j^2+\gamma \epsilon_j^2}{\alpha_j}=0,
		\quad
		\lim_{j\to\infty} \frac{\beta_j}{\alpha_j} = \nu_1 \quad \text{ and } \quad \lim_{j\to\infty} \frac{\mu_j}{ \alpha_j} = \nu_2
	\end{equation}
	 for some $0<\nu_1,\nu_2<\infty$ such that $\nu_1\leq \beta_j / \alpha_j$ and $\nu_2\leq \mu_j/\alpha_j$.
	Let
	\begin{equation}
		(c^j,s^j)_j \coloneqq \left( c_{\alpha_j,\beta_j,\mu_j}^{\delta_j,\epsilon_j},s_{\alpha_j,\beta_j,\mu_j}^{\delta_j,\epsilon_j} \right)_j
	\end{equation}
	denote the minimizing sequence of~\eqref{eq:generalFunctional2_2} obtained from data $u_{\delta_j}$ and $s_{\epsilon_j}$.

	Then there exists a weakly convergent subsequence of $(c^j,s^j)_j$ with $s^j\rightharpoonup s^*$ and the limit of every weakly convergent subsequence of $(c^j)_j$ is a $\penc$-minimizing solution.
	    
	\end{itemize}
\end{theorem}
The proof requires minor adaptations of the general approach, see Appendix \ref{app:appendix2}.  
The specific nature of the problem further allows the derivation of a stronger property of the minimizing sequence as can be seen in the following corollary.

\begin{corollary}
 In Theorem \ref{thm:existence_etc_2} (ii) $(s^j)_j$ has a strongly convergent subsequence $(s^k)_k$ and in (iii) the sequence $(s^j)_j$ converges strongly, i.e., Theorem \ref{thm:existence_etc_2}(ii) implies $s^k \to \tilde{s}$ and Theorem \ref{thm:existence_etc_2}(iii) implies $s^j \to s^*$. 
\end{corollary}

\begin{proof} We first consider the situation of Theorem \ref{thm:existence_etc_2}(ii). This implies the existence of a subsequence such that 
$(s^j-s_{\mathrm{mod},\epsilon}) \rightharpoonup (\tilde{s} - s_{\mathrm{mod},\epsilon})$.
This weakly convergent sequences also convergences in norm, if we can show  that $\|s^j-s_{\mathrm{mod},\epsilon}\| \to \|\tilde{s} - s_{\mathrm{mod},\epsilon} \|.$
The weakly lower semi-continuity of the norm further implies, that norm convergence, $s^j - s_{\mathrm{mod},\epsilon}
\to \tilde s - s_{\mathrm{mod},\epsilon}$,  is equivalent to 
\begin{equation*}
		\lVert \tilde{s} - s_{\mathrm{mod},\epsilon} \lVert \geq \limsup_j \lVert s^j 
		- s_{\mathrm{mod},\epsilon} \rVert \ .
	\end{equation*}
 We observe, that for fixed $\gamma$ the minimization property of $(c^j,s^j)$, i.e., $J_{\alpha,\beta,\mu}^{\delta_j,\epsilon_j}(c^j,s^j) \le J_{\alpha,\beta,\mu}^{\delta_j,\epsilon_j}(\tilde c, \tilde s)$,
 implies the boundedness of $\lVert s^j 
		- s_{\mathrm{mod},\epsilon} \rVert$.

 Now, assume that $s^j$  is not converging in norm to $\tilde s$, i.e., there exists  a $\tau_1$ such that
	\begin{equation}
		\tau_1 \coloneqq \limsup_j \lVert s^j -s_{\mathrm{mod},\epsilon}\rVert > \lVert\tilde{s} - s_{\mathrm{mod},\epsilon}\rVert.
	\end{equation}
	Since $\tau_1$ is a limit point of the sequence  
	$\lVert s^j -s_{\mathrm{mod},\epsilon}\rVert$
	 there exists a subsequence $(s^k)_{k\in\N}$, such that $(s^k-s_{\mathrm{mod},\epsilon}) \rightharpoonup (\tilde{s} - s_{\mathrm{mod},\epsilon})$ and $\lVert s^k - s_{\mathrm{mod},\epsilon} \rVert \to \tau_1$. Using the triangle inequality we obtain
	\begin{equation}
		\lVert s^k-s_{\mathrm{mod},\epsilon} \rVert - \lVert s_{\epsilon_k} - s_{\mathrm{mod},\epsilon} \rVert 
		\leq \lVert s^k - s_{\epsilon_k} \rVert
		\leq \lVert s^k-s_{\mathrm{mod},\epsilon} \rVert + \lVert s_{\epsilon_k} - s_{\mathrm{mod},\epsilon} \rVert
	\end{equation}
	and hence we conclude 
	\begin{equation}
		\lim_{k\to\infty} \lVert s^k - s_{\epsilon_k} \rVert 
		= \lim_{k\to\infty} \lVert s^k - s_{\mathrm{mod},\epsilon} \rVert = \tau_1.
	\end{equation}
	Furthermore, the weak-weak continuity of $B$ and $u_{\delta_k} \to u_\delta$
	imply $B(c^k,s^k)-u_{\delta_k} \rightharpoonup B(\tilde c, \tilde s) - u_\delta$.
	We now employ the lower semicontinuity of the norm as well as the convergence
	$J_{\alpha,\beta,\mu}^{\delta_j,\epsilon_j} (c^j,s^j)\rightarrow J_{\alpha,\beta,\mu}^{\delta,\epsilon}(\tilde c, \tilde s) $, see Theorem \ref{thm:existence_etc_2}(ii), and obtain
	   
	\begin{align}
	 {1 \over 2}\|B(\tilde c, \tilde s) - u_\delta\|^2 & \le \liminf{ {1 \over 2}\|B( c^k, s^k) - u_{\delta_k}\|^2} \\
	 &= \liminf \left\{ J_{\alpha,\beta,\mu}^{\delta_k,\epsilon_k} (c^k,s^k)  -{\gamma \over 2}\|s^k-s_{\epsilon_k}\|^2 - (\alpha, \beta, \mu)^t R_2(c^k,s^k)  \right\} \\
	 & = \liminf{J_{\alpha,\beta,\mu}^{\delta_k,\epsilon_k} (c^k,s^k)} -{\gamma \over 2}  \tau_1^2 - (\alpha, \beta, \mu)^t R_2(\tilde c, \tilde s) \\
	 & =  {J_{\alpha,\beta,\mu}^{\delta,\epsilon} (\tilde c,\tilde s)} -{\gamma \over 2}  \tau_1^2 - (\alpha, \beta, \mu)^t R_2(\tilde c, \tilde s)\\
	 & <  {1 \over 2}\|B(\tilde c, \tilde s) - u_\delta\|^2
	\end{align}
	which contradicts the assumption.

		We now consider case (iii) of the previous theorem.
		The pair $(c^j,s^j)$ is a minimizer of the functional $J_{\alpha_j,\beta_j,\mu_j}^{\delta_j,\epsilon_j}$, hence
		\begin{align}
			0 &\leq  J_{\alpha_j,\beta_j,\mu_j}^{\delta_j,\epsilon_j}(c^j,s^j) \leq J_{\alpha_j,\beta_j,\mu_j}^{\delta_j,\epsilon_j} (c^*,s^*) \notag \\
			&= \frac{1}{2} \lVert B(c^*,s^*) - u_{\delta_j}\rVert^2
			+ \frac{\gamma}{2} \lVert s^*-s_{\epsilon_j}\rVert^2 +\frac{\mu_j}{2} \lVert P(s^*) - s_\mathrm{calib}\rVert^2 + \alpha_j \penc(c^*) + \beta_j \pens(s^*) \notag \\
			&\leq \frac{1}{2} \left( \delta_j^2 + \gamma \epsilon_j^2 \right) +\frac{\mu_j}{2} \lVert P(s^*) - s_\mathrm{calib}\rVert^2 + \alpha_j \penc(c^*) + \beta_j \pens(s^*) \to 0, \label{eq:proofConvergenceMimisingProperty}
		\end{align}
		where the convergence of the right hand side follows from the parameter choice $\alpha_j,\beta_j,\mu_j \to 0$ and the noise levels $\delta_j,\epsilon_j \to 0$. This also implies, that
		\begin{equation*}
			\lim_{j\to\infty} \frac{1}{2} \lVert B(c^j,s^j)-u_{\delta_j} \rVert^2 + \frac{\gamma}{2} \lVert s^j - s_{\epsilon_j} \rVert^2 = 0.
		\end{equation*}
		Both terms are non-negative and $\gamma >0$ is fixed, hence, 
		\begin{equation*}
			\lim_{j\to\infty}   \lVert s^j - s_{\epsilon_j} \rVert^2 = 0.
		\end{equation*}
		Now, $\lim_{j \rightarrow \infty} \|s_{\epsilon_j}-s^*\| =0$ implies
	  $s^j \to s^*$.

\end{proof}

\begin{remark}
    Strong convergence for $c^j$ in Theorem \ref{thm:existence_etc_2} can be proven for certain choices of $\mathcal{R}_c$. We will prove this for the case of sparsity-promoting penalty terms in the next subsection. 
\end{remark}

\subsubsection*{Convergence rates for sparsity-promoting penalty terms}

We now prove    convergence rates for  sparsity-promoting penalty terms, see   \cite{daubechies2003iterative,jin2012review,Bangti2017},
\begin{equation} \label{eq:penaltySparsity}
	\Phi_p(c) = \sum_i w_i \lvert \langle c,\varphi_i \rangle \rvert^p
\end{equation}
with weights $0<w_\mathrm{min} \leq w_i<\infty$, $1\leq p \leq 2$, and orthonormal basis $\{\phi_i\}_{i\in \N}\subset X$.
In the remainder we prove two results, which use different source conditions and are applicable for certain ranges of $p$.
Without loss of generality we assume $w_\mathrm{min}=1$. 
The first result holds for  $1<p\leq 2$  and uses the following source condition.

 	\begin{assumption}\label{assumption:easySourceCondition}
Let Assumption \ref{ass:operatorB} be fulfilled. In  \eqref{eq:generalFunctional2} let $$(\alpha,\beta,\mu)^t\mathcal{R}_2(c,s)= \alpha \tilde{\mathcal{R}}(c,s):=\alpha \left(\Phi_p(c)+ \frac{\nu_2}{2} \|P(s)-s_\mathrm{calib} \|^2_Y + \nu_1 \pens(s)\right)$$ for $0<\nu_1,\nu_2<\infty$ (i.e., $\beta=\nu_2 \alpha$, $\mu=\nu_1 \alpha$). For $c^* \in X$ and $s^* \in Y$ we assume
		\begin{enumerate}
		\item[(i)] \emph{Source condition}: There exists an $\omega \in Z$ such that
		\begin{equation}
			(\xi_{c^*},\xi_{s^*}) = {B^\prime(c^*,s^*)}^\ast \omega
		\end{equation}
		with $(\xi_{c^*},\xi_{s^*}) \in \partial \tilde{\mathcal{R}}(c^*,s^*)$.
		\item[(ii)] \emph{Smallness assumption}: For $C$ from Assumption \ref{ass:operatorB} and $\gamma, \alpha$ in \eqref{eq:generalFunctional2} it holds
		\begin{equation}
			C \|\omega\| < \min \left\{ 1,\frac{\gamma}{2\alpha} \right\}.
		\end{equation}
		 		
	\end{enumerate}
	\end{assumption}

Again, the following proofs require only minor adaptations of the general theory. We start by computing the Bregman distance of the penalty $\tilde{\mathcal{R}}$ from the previous assumption after the following remark.
\begin{remark}
In the product space setting, for $(\xi_{c^\ast},\xi_{s^\ast})\in\partial \tilde{\mathcal{R}}(c^\ast,s^\ast)$ the Bregman distance is defined as
\begin{equation}
	D_{\tilde{\mathcal{R}}}^{(\xi_{c^\ast},\xi_{s^\ast})} ((c,s),(c^\ast,s^\ast))
	= \tilde{\mathcal{R}}(c,s)-\tilde{\mathcal{R}}(c^\ast,s^\ast)
	- \langle (\xi_{c^\ast},\xi_{s^\ast}), (x,y)-(c^\ast,s^\ast)\rangle \ .
\end{equation}
\end{remark}
Computing the Bregman distances term by term yields the following corollary.
\begin{corollary}
 \label{lemma:BregmanDistancePhisp}
For $1\leq p \leq 2$ the Bregman distance of $\tilde{\mathcal{R}}$ is given by
\begin{equation}\label{eq:BregmanDistancePhisp}
	D_{\tilde{\mathcal{R}}}^{(\xi_{c^\ast},\xi_{s^\ast})} ((c,s),(c^\ast,s^\ast)) = D_{\Phi_p}^{\xi_{c^\ast}}(c,c^\ast) + \nu_1 D_\pens^{\zeta_{s^\ast}} (s,s^\ast)+\frac{\nu_2}{2} \lVert P(s-s^\ast) \rVert^2
	\end{equation}	
with $\zeta_{s^\ast} \in \nu_1 \partial \pens(s^\ast)$ where $\xi_{s^\ast}= \zeta_{s^\ast} + \nu_2 P^\ast (Ps^\ast-s_\mathrm{calib})$.
\end{corollary}

We thus obtain the following result for the convergence rates.

\begin{theorem}[$1<p\leq 2$] \label{theorem:sparsityConvergenceEasy}
	Let $u_\delta \in Z$ with $\| B(c^*,s^*) - u_\delta \| \leq \delta$ and $\| s^* -s_{\mathrm{mod},\epsilon} \| \leq \epsilon$. Let $1<p\leq 2$ and let $c^*$ be a $\Phi_p$-minimizing solution. Furthermore, let Assumption \ref{assumption:easySourceCondition} be fulfilled, in particular we assume that $\alpha_{max}$ is chosen  s.t. Assumption   \ref{assumption:easySourceCondition}  (ii) is fulfilled for all $\alpha$ with $0<\alpha \leq \alpha_{max}< \infty$.  For $\beta=\nu_1 \alpha$, and $\mu=\nu_2 \alpha$ the minimizer of the functional $J^{\delta,\epsilon}_{\alpha,\beta,\mu}$ as defined in~\eqref{eq:generalFunctional2} is denoted by $(c^\alpha,s^\alpha)$. 
	
	Then, with $\alpha \sim \delta + \epsilon$ we have the convergence rates
	\begin{equation}
		\|B(c^\alpha,s^\alpha)-B(c^*,s^*) \| = \mathcal{O}(\delta + \epsilon) \text{ and }
		D_{\tilde{\mathcal{R}}}^{\xi^*} \left( (c^\alpha,s^\alpha),(c^*,s^*) \right) = \mathcal{O}(\delta + \epsilon).
	\end{equation}	\end{theorem}
	
	\begin{proof}
	We follow the general outline of \cite{jin2012review} for proving convergence rates for sparsity constrained Tikhonov functionals.
	First,   from the minimizing property of $(c^\alpha,s^\alpha)$ we have
	\begin{equation*}
		J^{\delta,\epsilon}_{\alpha,\nu_1 \alpha,\nu_2 \alpha} (c^\alpha,s^\alpha) \leq J^{\delta,\epsilon}_{\alpha,\nu_1 \alpha,\nu_2 \alpha}(c^*,s^*)
		\leq \frac{\delta^2 + \gamma \epsilon^2 }{2} + \alpha \tilde{\mathcal{R}}(c^*,s^*).
	\end{equation*}

The definition of the Bregman distance of $\tilde{\mathcal{R}}$ with  subgradient
$\xi^*=(\xi_{c^*},\xi_{s^*}) \in \partial \tilde{\mathcal{R}} (c^*,s^*)$ allows to rewrite the
above equation as
\begin{align}
	\frac{1}{2} \|B(c^\alpha,&s^\alpha)-u_\delta \|^2 + \frac{\gamma}{2} \|s^\alpha-s_{\mathrm{mod},\epsilon}\|^2  \notag \\ \label{eq:ConvergenceRatesInequalityBregmanDistanceSp}
			&\leq \frac{\delta^2 + \gamma \epsilon^2 }{2} - \alpha \left(
			D_{\tilde{\mathcal{R}}}^{\xi^\ast} ((c^\alpha,s^\alpha),(c^\ast,s^\ast))
			+ \langle (\xi_{c^\ast},\xi_{s^\ast}),(c^\alpha,s^\alpha)-(c^\ast,s^\ast) \rangle \right).
\end{align}

Second, from the parallelogram law we obtain  
\begin{equation*}
	\frac{\gamma}{2} \|s^\alpha - s^* \|^2 \leq \gamma \|s^\alpha - s_{\mathrm{mod},\epsilon} \|^2 + \gamma \epsilon^2
\end{equation*}
\begin{equation*}
	\frac{1}{2} \| B(c^\alpha,s^\alpha)-B(c^*,s^*)\|^2
	\leq \|B(c^\alpha,s^\alpha)-u_\delta \|^2 + \delta^2
\end{equation*}
which yields in combination 
with~\eqref{eq:ConvergenceRatesInequalityBregmanDistanceSp}  
the following estimate
\begin{align}
	\frac{1}{4} \| B(c^\alpha,&s^\alpha)-B(c^*,s^*)\|^2 + \frac{\gamma}{4} \|s^\alpha - s^* \|^2 \notag \\
		&\leq \delta^2 + \gamma \epsilon^2  - \alpha  \left(
			D_{\tilde{\mathcal{R}}}^{\xi^*} ((c^\alpha,s^\alpha),(c^*,s^*))
			+ \langle (\xi_{c^*},\xi_{s^*}),(c^\alpha,s^\alpha)-(c^*,s^*) \rangle \right). \label{eqinproof:ConvergenceRatesBeforeScalarProductSp}
\end{align}

	Third, we exploit the source condition.
	We define $r\coloneqq B(c^\alpha,s^\alpha)-B(c^*,s^*)-B^\prime(c^*,s^*)((c^\alpha,s^\alpha)-(c^*,s^*))$ as the residual of the first degree Taylor expansion.
	$B$ is a bilinear operator, hence, with $C$ as in Assumption 3.1 we have
	$\|r\|\le {1 \over 2} C \|(c^*,s^*)-(c^\alpha,s^\alpha)\|^2$.
	Then, using the source condition
	 (Assumption~\ref{assumption:easySourceCondition} (i))   we can estimate the last term of the above inequality as
	\begin{align*}
		- \langle (\xi_{c^*},\xi_{s^*}),(c^\alpha,s^\alpha)-&(c^*,s^*) \rangle \\
		&= -\langle {B^\prime(c^*,s^*)}^\ast\omega ,(c^\alpha,s^\alpha)-(c^*,s^*)\rangle	\\
		&= - \left\langle \omega ,B^\prime(c^*,s^*)\left((c^\alpha,s^\alpha)-(c^*,s^*)\right)\right\rangle \\
		&= \langle \omega, B(c^*,s^*)-B(c^\alpha,s^\alpha)+r \rangle \\
		&\leq \|\omega \| \|B(c^\alpha,s^\alpha)-B(c^*,s^*)\| + \frac{C}{2}
		\|\omega \| \| (c^\alpha,s^\alpha)-(c^*,s^*)\|^2 \\
		&\leq \|\omega \| \|B(c^\alpha,s^\alpha)-B(c^*,s^*)\| + \frac{C}{2}
		\|\omega \| \| c^\alpha-c^*\|^2 +\frac{C}{2}
		\|\omega \| \| s^\alpha-s^*\|^2 \\
		&\leq \|\omega \| \|B(c^\alpha,s^\alpha)-B(c^*,s^*)\| + \frac{C}{2}
		\|\omega \| \underset{\leq D_{\Phi_p}^{\xi_{c^*}}(c^\alpha,c^*)}{\underbrace{D_{\ell^p}^{\xi_{c^*}}(c^\alpha,c^*)}} +\frac{C}{2}
		\|\omega \| \| s^\alpha-s^*\|^2.
	\end{align*}
	The estimation in the last step follows from the 2-convexity of the $\ell^p$-spaces for $1<p\leq 2$.

Thus inequality~\eqref{eqinproof:ConvergenceRatesBeforeScalarProductSp} becomes
\begin{align*}
	\frac{1}{4} \| B(c^\alpha,&s^\alpha)-B(c^*,s^*)\|^2 + \alpha D_{\tilde{\mathcal{R}}}^{\xi^*} ((c^\alpha,s^\alpha),(c^*,s^*)) \notag \\
	&\leq \delta^2 + \gamma \epsilon^2 -  \frac{\gamma}{4} \|s^\alpha - s^* \|^2 + \alpha \|\omega \| \|B(c^\alpha,s^\alpha)-B(c^*,s^*)\|\\
	 &+ \frac{C}{2}
		\alpha\|\omega \| D_{\Phi_p}^{\xi_{c^*}}(c^\alpha,c^*) +\frac{C}{2}
		\alpha\|\omega \| \| s^\alpha-s^*\|^2\\ 
					&= \delta^2 + \gamma \epsilon^2   + \alpha
			\|\omega \| \|B(c^\alpha,s^\alpha)-B(c^*,s^*)\| \\
			&+  \frac{1}{2} (\alpha C \|\omega \|- \frac{\gamma}{2}) \|s^\alpha - s^* \|^2+ \alpha \frac{C}{2}	\|\omega \| D_{\Phi_p}^{\xi_{c^*}}(c^\alpha,c^*).
\end{align*}
By using the Bregman distance~\eqref{eq:BregmanDistancePhisp} we obtain
\begin{align}
	\frac{1}{4} \| B(c^\alpha,&s^\alpha)-B(c^*,s^*)\|^2 +  \alpha (1-\frac{C}{2}	\|\omega \|) D_{\Phi_p}^{\xi_{c^*}}(c^\alpha,c^*) + \alpha \nu_1 D_\pens^{\zeta_{s^\ast}} (s^\alpha,s^*) + \alpha \frac{\nu_2}{2} \|P(s^\alpha-s^*)\|^2 \notag \\
					&\leq \delta^2 + \gamma \epsilon^2   + \alpha
			\|\omega \| \|B(c^\alpha,s^\alpha)-B(c^*,s^*)\| \notag \\
			&+  \frac{1}{2} \left(\alpha C \|\omega \| - \frac{\gamma}{2}\right) \|s^\alpha - s^* \|^2 \notag\\
			&\leq \delta^2 + \gamma \epsilon^2   + \alpha
			\|\omega \| \|B(c^\alpha,s^\alpha)-B(c^*,s^*)\|,
			\label{eq:ConvergenceRatesInequalityBregmanDistanceRSp}
\end{align}
where the last step follows from the smallness assumption (Assumption~\ref{assumption:easySourceCondition} (ii)).
This smallness assumption also tells us, that $(1-C	\|\omega \|)$ is positive. As the Bregman distance is non-negative, we can deduce
\begin{equation*}
	\frac{1}{4} \| B(c^\alpha,s^\alpha)-B(c^*,s^*)\|^2
		    - \alpha \|\omega \| \|B(c^\alpha,s^\alpha)-B(c^*,s^*)\|
		-(\delta^2 + \gamma \epsilon^2) \leq 0.
\end{equation*}
	This is a quadratic equation with a non-negative argument, hence,:	
	\begin{equation} \label{eq:ConvergenceRatesRoot}
		\| B(c^\alpha,s^\alpha)-B(c^*,s^*)\| \leq 2 \alpha \|\omega \|
		+ 2 \sqrt{\alpha^2 \|\omega\|^2 + \delta^2 + \gamma \epsilon^2}.
	\end{equation}

In the fourth step we use the parameter choice rule $\alpha\sim \delta + \epsilon$, i.e., $\alpha \le M(\delta + \epsilon)$ with some constant $M>0$, such that we can conclude
	\begin{align*}
		\| B(c^\alpha,s^\alpha)-B(c^*,s^*)\| &\leq 2 M(\delta + \epsilon) \|\omega \|
		+ 2  \sqrt{M^2(\delta + \epsilon)^2 \|\omega\|^2 + \delta^2 + \gamma \epsilon^2} \\ 
			&\leq (\delta + \epsilon)\left( 2 M \|\omega \| +  2 \sqrt{2 M^2 \|\omega\|^2 + \max(1,\gamma)}\right)
	\end{align*}
and thus $\| B(c^\alpha,s^\alpha)-B(c^*,s^*)\|= \mathcal{O}(\delta+\epsilon)$.

	Finally, in the last step, in order to get the convergence rate for the Bregman distance we first need to estimate the single terms  on the left hand side of~\eqref{eq:ConvergenceRatesInequalityBregmanDistanceRSp}, which are all positive. Using~\eqref{eq:ConvergenceRatesRoot} we derive
	\begin{equation*}
		D_{\ell^p}^{\xi_{c^*}}(c^\alpha,c^*)\leq D_{\Phi_p}^{\xi_{c^*}}(c^\alpha,c^*)
				\leq  \frac{1}{(1-\frac{C}{2}	\|\omega \|)} \left( \frac{\delta^2 + \gamma \epsilon^2}{\alpha}   + 2 \alpha
				\|\omega \|^2 + 2\|\omega\| \sqrt{\alpha^2 \|\omega\|^2 + \delta^2 + \gamma \epsilon^2} \right)
	\end{equation*}
	and
	\begin{equation*}
	 \nu_1 D_\pens^{\zeta_{s^\ast}} (s^\alpha,s^\ast)
				\leq \frac{\delta^2 + \gamma \epsilon^2}{ \alpha}   + 2\alpha\|\omega\|^2+ 2 \|\omega\|\sqrt{\alpha^2 \|\omega\|^2 + \delta^2 + \gamma \epsilon^2}
	\end{equation*}
	and
	\begin{equation*}
		\nu_2 \|P (s^\alpha -s^\ast)\|^2
		\leq \frac{\delta^2 + \gamma \epsilon^2}{ \alpha} + 2\alpha \| \omega\|^2 + 2 \|\omega \| \sqrt{\alpha^2 \|\omega\|^2 + \delta^2 + \gamma \epsilon^2}.
	\end{equation*}
	By using the parameter choice rule $\alpha \sim \delta + \epsilon$ and estimating both, $\delta$ and $\epsilon$, with $(\delta+\epsilon)$ again
	we get that all of the three terms above are in the order $\delta+\epsilon$. Hence, in total we have the convergence rate
	\begin{equation*}
		D_{\tilde{\mathcal{R}}}^{\xi^*} \left( (c^\alpha,s^\alpha),(c^*,s^*)\right)
		= \mathcal{O}(\delta + \epsilon).
	\end{equation*}
\end{proof}
For $p=2$ this also yields the convergence rates for a quadratic penalty term.
 The following corollary shows the convergence in the Hilbert space norm.
\begin{corollary}\label{corollary:NormConvergenceSparsity}
	From Theorem~\ref{theorem:sparsityConvergenceEasy} directly follows a convergence rate of $\lVert c^\alpha - c^\ast \rVert = \mathcal{O}(\sqrt{\delta+\epsilon})$.
\end{corollary}
\begin{proof}
	Since $D^{\xi^\ast}_{\tilde{\mathcal{R}}}$  is the sum of non-negative Bregman distances including $D_{\Phi_p}^{\xi_{c^\ast}}$, we have
	\begin{equation}
		D_{\Phi_p}^{\xi_{c^\ast}} (c^\alpha,c^\ast) = \mathcal{O}(\delta + \epsilon).
	\end{equation}
We continue to estimate the Bregman distance by the Hilbert space norm using the definition of the subgradient for the sparsity term and the 2-convexity of the $\ell^p$-spaces for $1<p\leq 2$ 
\begin{align*}
	D_{\Phi_p}(c^\alpha,c^\ast) &= \sum_i w_i \lvert \langle \varphi_i,c^\alpha \rvert^p - \sum_i w_i \lvert \langle \varphi_i,c^\ast \rangle \rvert^p- \left\langle \sum_i w_i \lvert \langle \varphi_i,c^\ast \rangle \rvert^{p-1} \sign (\langle \varphi_i,c^\ast\rangle) \varphi_i,c^\alpha-c^\ast \right\rangle \\
	&\geq w_\textnormal{min} \sum_i \left( \lvert \langle \varphi_i,c^\alpha \rvert^p - \lvert \langle \varphi_i,c^\ast \rangle \rvert^p  -  \left\langle \lvert \langle \varphi_i,c^\ast \rangle \rvert^{p-1} \sign (\langle \varphi_i,c^\ast\rangle) \varphi_i,c^\alpha-c^\ast \right\rangle \right) \\
	&= w_\textnormal{min} D_{\ell^p}^{\xi_{c^\ast}}(c^\alpha,c^\ast) \\
	&\geq C \lVert c^\alpha -c^\ast \rVert^2.
\end{align*}
Thus, we get
\begin{equation*}
	\lVert c^\alpha - c^\ast \rVert^2 \leq \frac{1}{C} D_{\Phi_p}^{\xi_{c^\ast}}(c^\alpha,c^\ast),
\end{equation*}
where  $C$ is a positive constant.
\end{proof}

\begin{remark}
	Assumption~\ref{assumption:easySourceCondition} does not yield convergence rates for $p=1$ because 
	$\ell^1$ is not 2-convex.
	\end{remark}

In order to obtain a convergence result for $p=1$ we consider an alternative source condition, which is a generalization of the previous source condition in certain cases.   
	
\begin{assumption}\label{assumption:sparseApproximation}
Let Assumption \ref{ass:operatorB} be fulfilled. In  \eqref{eq:generalFunctional2} let $$(\alpha,\beta,\mu)^t\mathcal{R}_2(c,s)= \alpha \tilde{\mathcal{R}}(c,s):=\alpha \left(\Phi_p(c)+ \frac{\nu_2}{2} \|P(s)-s_\mathrm{calib} \|^2_Y + \nu_1 \pens(s)\right)$$ for $0<\nu_1,\nu_2<\infty$ (i.e., $\beta=\nu_2 \alpha$, $\mu=\nu_1 \alpha$). 

We further assume for an $0 <\alpha_\mathrm{max}<\infty$ and $\gamma$ from \eqref{eq:generalFunctional2}
	\begin{enumerateAssumption}
	\item There exists an $\Phi_p$-minimizing solution $c^\ast\in X$ for $s^\ast \in Y$ and $u^\ast \in Z$.
	\item
	There exist constants $\kappa_1 \in [0,1), \kappa_2,\kappa_3	\geq 0$, $\kappa_3<\min (1, \frac{\gamma}{2 \alpha_\mathrm{max}})$, and a subgradient $\xi^\ast:=(\xi_{c^\ast},\xi_{s^\ast})\in \partial \tilde{\mathcal{R}}(c^\ast,s^\ast)$, such that 
	\begin{equation}\label{eq:source_cond2}
		\langle (\xi_{c^\ast},\xi_{s^\ast}), (c^\ast -c,s^\ast-s)\rangle \leq \kappa_1 D_{\tilde{\mathcal{R}}}^{\xi^\ast} ((c,s),(c^\ast,s^\ast)) + \kappa_2 \lVert B(c,s)-B(c^\ast,s^\ast) \rVert + \kappa_3 \lVert s-s^\ast\rVert^2
	\end{equation}
	for all $(c,s)\in X \times Y$. 
	\end{enumerateAssumption} 
\end{assumption}
\begin{remark}
A similar type of condition was first introduced for more general penalty terms in \cite{Hofmann:2007us} and particularly for sparsity regularization in \cite{Grasmair:2008cy}. The general source condition in \cite[Ass. 4.1]{Hofmann:2007us} can also be applied to the product space setting defined in the proof of Theorem \ref{thm:existence_etc_2} resulting in the same convergence rate. In line with assumption \cite[Ass. 4.1]{Hofmann:2007us} restricting \eqref{eq:source_cond2} to hold true for $(c,s) \in \mathcal{M}_{\alpha_\mathrm{max}}^{\delta,\epsilon}(M):=\{ (c,s) \in X\times Y | J^{\delta,\epsilon}_{\alpha_\mathrm{max},\nu_1 \alpha_\mathrm{max},\nu_2 \alpha_\mathrm{max}}(c,s)\leq M \}$ where $M > \alpha_\mathrm{max} \left( \tilde{\mathcal{R}} (c^\ast,s^\ast) + \frac{\delta^2+ \gamma \epsilon^2}{\alpha} \right)$ for given $(\delta,\epsilon,\alpha)$ tuples (e.g., those used in Theorem \ref{theorem:ConvergenceRatesSparsityFull}) may allow for the following statement. If $\mathcal{M}_{\alpha_\mathrm{max}}^{\delta,\epsilon}(M)\subset X \times B_1(s^\ast)$ holds true, we could then obtain \cite[Ass. 4.1(5)]{Hofmann:2007us} from \eqref{eq:source_cond2} exploiting $\| s - s^\ast\| \leq 1$ and the equivalence of p-norms in the product space norm of $X \times Y$. In this particular case \cite[Ass. 4.1]{Hofmann:2007us} would be a generalization of the used source condition in Assumption \ref{assumption:sparseApproximation} equipped with the previously described restriction. However, proving this relation is beyond the scope of the present work.
\end{remark}

For $1<p\leq 2$ the condition in Assumption \ref{assumption:sparseApproximation} is a generalisation of the previous source condition in Assumption \ref{assumption:easySourceCondition} as the next proposition will show.
\begin{proposition}
Let Assumption~\ref{assumption:easySourceCondition} be fulfilled and let $1<p\leq 2$. Then there exist $\kappa_1 \in [0,1)$, $\kappa_2 \geq 0$ and $\kappa_3 < \min (1,\frac{\gamma}{2\alpha_\mathrm {max}})$ such that 
\begin{equation}
	\langle (\xi_{c^\ast},\xi_{s^\ast}), (c^\ast -c,s^\ast-s)\rangle \leq \kappa_1 D_{\tilde{\mathcal{R}}}^{\xi^\ast} ((c,s),(c^\ast,s^\ast)) + \kappa_2 \lVert B(c,s)-B(c^\ast,s^\ast) \rVert + \kappa_3 \lVert s-s^\ast \rVert^2. 
	\end{equation}
	\end{proposition}
\begin{proof}
We start by using the source condition (i) of Assumption~\ref{assumption:easySourceCondition} and the Cauchy-Schwartz inequality to estimate the left hand side of the assertion.
\begin{align*}
	\langle &(\xi_{c^\ast},\xi_{s^\ast}),(c^\ast-c,s^\ast-s) \rangle \\
	&= \langle \omega, B^\prime(c^\ast,s^\ast)(c^\ast-c,s^\ast-s) \rangle \\
	&\leq \lVert \omega \rVert \lVert B(c,s)-B(c^\ast,s^\ast) - B(c,s) + B(c^\ast,s^\ast)+ B^\prime (c^\ast,s^\ast)(c^\ast-c,s^\ast-s) \rVert \\
	&\leq \lVert \omega \rVert \lVert B(c,s)-B(c^\ast,s^\ast) - B^\prime(c^\ast,s^\ast) (c^\ast-c,s^\ast-s) \rVert + \lVert \omega \rVert \lVert B(c,s) - B(c^\ast,s^\ast) \rVert. 
\end{align*} 
Defining $\kappa_2=\lVert \omega \rVert$, we only need to estimate the first term now. By using \eqref{eq:BTaylorRemainder} and the 2-convexity of the $\ell^p$-spaces for $1<p\leq 2$ we obtain
\begin{align*}
	\lVert \omega \rVert &\lVert B(c,s)-B(c^\ast,s^\ast) - B^\prime(c^\ast,s^\ast) (c^\ast-c,s^\ast-s) \rVert \\
	&\leq \lVert \omega \rVert \frac{C}{2} \lVert (c^\ast-c,s^\ast-s) \rVert^2 \\
	&= \lVert \omega \rVert \frac{C}{2} \left( \lVert c^\ast-c\rVert^2 + \lVert s^\ast-s\rVert^2 \right) \\
	&\leq \lVert \omega \rVert \frac{C}{2} \left( D_{\ell^p}^{\xi_{c^\ast}} (c^\ast,c) + \lVert s^\ast-s\rVert^2 \right) \\
	&\leq \underbrace{\omega \frac{C}{2}}_{\eqqcolon \kappa_1} D_{\tilde{\mathcal{R}}}^{\xi^\ast} ((c^\ast,s^\ast),(c,s)) + \frac{1}{2}  \underbrace{\lVert \omega \rVert C}_{\eqqcolon \kappa_3} \lVert s^\ast-s \rVert^2,
\end{align*}
where the last estimate follows from $D_{\Phi_p}^{\xi_{c^\ast}}\geq D_{\ell^p}^{\xi_{c^\ast}}$ being part of the Bregman distance $D_{\tilde{\mathcal{R}}}^{\xi^\ast}$. From the smallness assumption we conclude that $\kappa_1<1$ and $\kappa_3 <\min (1, \frac{\gamma}{2\alpha_\mathrm{max}})$.
\end{proof}

In the following we present a convergence rate result based on the source condition in Assumption \ref{assumption:sparseApproximation} which now includes the case $p=1$ where we obtain the same order of convergence as in Theorem~\ref{theorem:sparsityConvergenceEasy}.
\begin{theorem}\label{theorem:ConvergenceRatesSparsityFull}
Let $u_\delta \in Z$ with $\| u^\ast - u_\delta \| \leq \delta$ and $\| s^\ast -s_{\mathrm{mod},\epsilon} \| \leq \epsilon$. Let $1\leq p\leq 2$ and let Assumption \ref{assumption:sparseApproximation} be fulfilled. For $0<\alpha \leq \alpha_\mathrm{max}< \infty$, $\beta=\nu_1 \alpha$, and $\mu=\nu_2 \alpha$ the minimizer of the functional $J^{\delta,\epsilon}_{\alpha,\beta,\mu}$ as defined in~\eqref{eq:generalFunctional2} is denoted by $(c^\alpha,s^\alpha)$. 
Further assume that $\alpha \sim (\delta+\epsilon)$. Then it holds
\begin{equation}
	D_{\tilde{\mathcal{R}}}^{\xi^\ast}((c^\alpha,s^\alpha),(c^\ast,s^\ast)) = \mathcal{O}(\delta + \epsilon) \quad \textnormal{and} \quad
	\lVert B(c^\alpha,s^\alpha)-B(c^\ast,s^\ast) \rVert = \mathcal{O}(\delta + \epsilon).
\end{equation}
\end{theorem}
\begin{proof}
	For the first two steps we refer to step (i) and (ii) in the proof of Theorem \ref{theorem:sparsityConvergenceEasy}
	(up to equation ~\eqref{eqinproof:ConvergenceRatesBeforeScalarProductSp}). Thus, we start with
	\begin{align*}
		\frac{1}{4} &\lVert B(c^\alpha,s^\alpha)-B(c^\ast,s^\ast) \rVert^2
		+ \frac{\gamma}{2} \lVert s^\alpha -s^\ast \rVert^2 \\
		&\leq \delta^2 + \gamma \epsilon^2 
		- \alpha \left( D_{\tilde{\mathcal{R}}}^{\xi^\ast}
	 ((c^\alpha,s^\alpha),(c^\ast,s^\ast))
	+ \langle (\xi_{c^\ast},\xi_{s^\ast}),(c^\alpha,s^\alpha)-(c^\ast,s^\ast) \rangle \right).
	\end{align*}
Using the source condition in Assumption~\ref{assumption:sparseApproximation}, the above equation transforms to 
	\begin{align*}
		\frac{1}{4} &\lVert B(c^\alpha,s^\alpha)-B(c^\ast,s^\ast) \rVert^2
		+ \frac{\gamma}{2} \lVert s^\alpha -s^\ast \rVert^2 \\
		&\leq \delta^2 + \gamma \epsilon^2 
		- \alpha D_{\tilde{\mathcal{R}}}^{\xi^\ast}
	 ((c^\alpha,s^\alpha),(c^\ast,s^\ast)) \\
	&+ \alpha \left( \kappa_1 D_{\tilde{\mathcal{R}}}^{\xi^\ast}((c^\alpha,s^\alpha),(c^\ast,s^\ast)) + 
	\kappa_2 \lVert B(c^\alpha,s^\alpha)- B(c^\ast,s^\ast) \rVert
	+ \frac{\kappa_3}{2} \lVert s^\alpha-s^\ast \rVert^2 \right),
	\end{align*}
	which we reorder to 
	\begin{align}
		\frac{1}{4} &\lVert B(c^\alpha,s^\alpha)-B(c^\ast,s^\ast) \rVert^2
		+ \alpha (1-\kappa_1) D_{\tilde{\mathcal{R}}}^{\xi^\ast}((c^\alpha,s^\alpha),(c^\ast,s^\ast)) \notag \\
		&\leq \delta^2 + \gamma \epsilon^2 + \kappa_2 \lVert B(c^\alpha,s^\alpha)-B(c^\ast,s^\ast) \rVert + 
		\frac{1}{2} \left(\alpha\kappa_3-\frac{\gamma}{2}\right) \lVert s^\alpha-s^\ast\rVert^2 \notag \\
		&\leq \delta^2 + \gamma \epsilon^2 + \kappa_2 \lVert B(c^\alpha,s^\alpha)-B(c^\ast,s^\ast) \rVert. \label{eqinproof:ConvergenceRatesPrelEstimate}
	\end{align}
	Since the factor $(1-\kappa_1)$ is positive we can use the same line of reasoning as in the proof of Theorem~\ref{theorem:sparsityConvergenceEasy} (starting after equation~\eqref{eq:ConvergenceRatesInequalityBregmanDistanceRSp}) and get
	\begin{equation}
		\lVert B(c^\alpha,s^\alpha)-B(c^\ast,s^\ast) \rVert = \mathcal{O}(\delta+\epsilon) \quad \textnormal{for } \alpha\sim \delta+\epsilon.
	\end{equation}
	With this we immediately deduce from equation~\eqref{eqinproof:ConvergenceRatesPrelEstimate}, that also
	\begin{equation}
		D_{\tilde{\mathcal{R}}}^{\xi^\ast}((c^\alpha,s^\alpha),(c^\ast,s^\ast)) 
		= \mathcal{O}(\delta+\epsilon).
	\end{equation}
	This concludes the proof.
\end{proof}
As before we can directly infer convergence rates in the norm for $1 < p \leq 2$. 
\begin{corollary}
From Theorem~\ref{theorem:ConvergenceRatesSparsityFull} the convergence rate
\begin{equation}
	\lVert c^\alpha - c^\ast \rVert = \mathcal{O}(\sqrt{\delta+\epsilon}), \quad 1<p\leq 2,
\end{equation}
follows. 
\end{corollary}
\begin{proof}
	The assertion follows analogously to the proof of Corollary~\ref{corollary:NormConvergenceSparsity}.
\end{proof}

\begin{remark}
The case $p=1$ requires additional assumptions to connect the Bregman distance to the norm, such as a sparsity assumption on the minimum norm solution in a certain basis like for example in \cite{Bredies:2008gb}. A possible proof for the case $p=1$ may follow the proof for \cite[Thm. 3.54]{Scherzer:2009tt}.
\end{remark}

Finally, we conclude this section with the a convergence rate result of $s$ with respect to the norm topology.  
\begin{corollary}
From Theorem~\ref{theorem:ConvergenceRatesSparsityFull} it follows the convergence rate
\begin{equation}
	\lVert s^\alpha - s^\ast \rVert = \mathcal{O}(\sqrt{\delta+\epsilon}).
\end{equation}
\end{corollary}
\begin{proof}
We exploit the estimate 
\begin{equation} 
\frac{\gamma}{2} \| s^\alpha - s_{\mathrm{mod},\epsilon} \|^2 \leq J_\alpha^{\delta,\epsilon}(c^\alpha, s^\alpha) \leq J_\alpha ^{\delta,\epsilon}(c^\ast,s^\ast) \leq \underset{\leq \frac12 \max (1,\gamma) (\delta + \epsilon)^2}{\underbrace{\frac{\delta^2 + \gamma \epsilon^2}{2}}} + \alpha \tilde{\mathcal{R}}(c^\ast,s^\ast)
\end{equation}
and the parameter choice rule $\alpha \sim \delta+ \epsilon$ to estimate
\begin{equation}
    \| s^\alpha - s^\ast\|^2 \leq 2\| s^\alpha - s_{\mathrm{mod},\epsilon} \|^2 + 2\epsilon^2 \leq ( 2\max (1,\gamma) / \gamma  + 2 ) (\delta + \epsilon)^2 + C (\delta + \epsilon) 
\end{equation}
which concludes the proof.
\end{proof}
\begin{remark}
We have not chosen $\mathcal{R}_s$, yet. Depending on the specific penalty one might derive better convergence rates in the norm topology. Further investigations in this direction are beyond the scope of the present work.
\end{remark}

\section{Algorithmic solution}
\label{sec:algorithm}

The algorithmic solution is derived for the discretized integral operator in \eqref{eq:example_integral_operator}. 
After discretization in suitable finite-dimensional subspaces of the spaces $X$, $Y$, and $Z$, we obtain the discretized problem in terms of a matrix-vector product, i.e., the discretized operator 
$\tilde{B}: \R^{M} \times \K^{K\times M} \rightarrow \K^K$
with 
\begin{equation}
    (c,S) \mapsto Sc
\end{equation}
where $\K=\R,\C$ and $N,M,K \in \N$. 
Information source of type A analogously translates into the discrete setup, where either $S_\mathrm{mod} \in \K^{K\times M}$ or $S_{\mathrm{mod},\epsilon} \in \K^{K \times M}$ is available.
For information source of type B we assume a given $S_\mathrm{calib} \in \K^{K \times N}$, $N<M$.
And an operator $\tilde{P}: \K^{K\times M} \rightarrow \K^{K \times N}$ derived from the linear operator $P$ which affects the integral kernel in its space variable $x$ only. 
The measurement dimension remains unaffected. 
In this specific case the operator $\tilde{P}$ can be represented by a matrix $Q\in \K^{M \times N} $, i.e., $S \mapsto \tilde{P}(S)=SQ$.  
According to the particular discretizations $R_c:\R^M \rightarrow \R_+$ and $R_s: \K^{K\times M} \rightarrow \R_+$ are obtained from $\mathcal{R}_c$ and $\mathcal{R}_s$.

In this work, the joint reconstruction of an image $c$ and the system matrix $S$ is obtained by minimizing
\begin{equation}
    \label{eq:joint:min:c:S}
    J (c, S) =
    \frac{1}{2} \| S c - u\|^2 + \frac{\gamma}{2} \|S - S_{\mathrm{mod}}\|^2_F 
    + \frac{\mu}{2} \|S Q - S_{\mathrm{calib}}\|^2_F
    + \alpha R_c (c) + \beta R_s (S)
\end{equation}
for given $\alpha,\beta,\gamma,\mu \geq 0$ which we solve by alternatingly minimizing $J (c, S)$ with respect to $c$ and $S$ following two steps:
\begin{align}
    \label{eq:alter:min:c:penalized:1}
    \min_{c \in \R^{M}} 
    & \underset{=:J^c (c) }{\underbrace{\frac{1}{2} \| S c - u\|^2 + \alpha R_c (c)}},\\
    \label{eq:alter:min:S:penalized:1}
    \min_{S \in \K^{K\times M}}
    & \underset{=:J^S (S) }{\underbrace{\frac{1}{2} \| S c - u\|^2 
    + \frac{\gamma}{2} \|S - S_{\mathrm{mod}}\|^2_F 
    + \frac{\mu}{2} \|S Q - S_{\mathrm{calib}}\|^2_F
    + \beta R_s (S)}}.
\end{align}
For the regularization of particle density, based on previous experience, we use the following combination of $\ell^1$ and $\ell^2$ regularization
\begin{equation}
    \label{eq:c:reg:L1+L2}
    R_c (c) = |c|_2^2+ \frac{\lambda}{\alpha} |c|_1.
\end{equation}
For the regularization of the system matrix, the term with $S_{\mathrm{mod}}$ including the Frobenius norm provides an $\ell^2$-type regularization. The term with $S_{\mathrm{calib}}$ provides a priori information on the high resolution system matrix. Here, we do not impose further regularization on the system matrix with our prior knowledge on MPI. Therefore, the reconstruction functionals are set to 
\begin{align}
    \label{eq:alter:min:c:penalized}
    J^c (c) 
    & = \| S c - u\|^2 + \tilde{\alpha}^2 |c|_2^2 + \lambda |c|_1,\\
    \label{eq:alter:min:S:penalized}
    J^S (S) 
    & = \| S c - u\|^2 
    + \tilde{\gamma}^2 \|S - S_{\mathrm{mod}}\|^2_F 
    + \tilde{\mu}^2 \|S Q - S_{\mathrm{calib}}\|^2_F
\end{align}
where we have substituted regularization parameters for notational simplicity of the numerical algorithms presentations, i.e., $\tilde{\alpha}=\sqrt{\frac12\alpha}$, $\tilde\gamma$ and $\tilde\mu$ analogously. 

In the remainder we focus on the algorithmic derivation for $\K = \C$ as the system matrix in MPI is commonly complex-valued. Furthermore $c$ is equipped with a positivity constraint.

\subsection{Regularized Kaczmarz algorithm}
\label{subsec:reg:kacamarz}

When the regularization parameters are fixed, we derive the following reconstruction algorithm with the regularization terms in \eqref{eq:c:reg:L1+L2} in the following. The Kaczmarz algorithm is chosen for the following reasons. For MPI a Kaczmarz-type algorithm \cite{dax1993row} still defines one of the standard methods since the very beginning of MPI research \cite{Knopp2017} which seems to benefit qualitatively from an advantageous influence of early stopping. 
More recently this has also been observed quantitatively on real MPI data \cite{KluthJin2019b_preprint}.    
Moreover, because of its row-action nature, the Kaczmarz algorithm provides us the flexibility for choosing the order of using the measured data $u$ and steer the reconstruction process. Because our formulation is an unconventional mixture of real and complex spaces and with additional regularization terms, we provide the derivation of the Kaczmarz algorithm with the regularization terms used in this work. The algorithms are developed by applying an alternating minimizing approach, or the incremental gradient descent method, to the reconstruction functionals in \eqref{eq:alter:min:c:penalized} and \eqref{eq:alter:min:S:penalized}  with appropriate decomposition of both functionals.

Because the Kaczmarz algorithm is based on the projection onto hyperplanes \cite{censor2009note}, we need the projection onto hyperplanes in the complex $\C^M$. For a hyperplane $H$ in $\C^M$, given by $a\in\C^M$ and $b\in \C$, i.e.,
\begin{equation}
H = \{ z \in \C^M:\, \sum_{m = 1}^M a_m z_m  = b \},
\end{equation}
the projection $P_H$ for any $z \in \C^M$ to $H$ is equal to
\begin{equation}
P_H[z] = z -  \frac{\sum_{m = 1}^M a_m z_m - b}{\|a\|^2} \bar{a},
\end{equation}
where $\bar{\cdot}$ is the complex conjugate, and vectors $z$ and $a$ are column vectors. For a $\ell^2$-regularized least-squares problem as
\begin{align}
    \label{eq:LS:L2:penalized}
    \| A z - b\|^2 + \eta^2 |z - z_0|_2^2,
\end{align}
for $A \in \C^{K \times M}$, $z_0 \in \C^M$, and $\eta > 0$, 
by introducing an auxiliary variable $v \in \C^K$, the following system of equations is consistent \cite{herman_lent_hurwitz_1980, herman_2009},
\begin{align}
    \label{eq:LS:L2:v:z}
    \begin{pmatrix}
    \eta I & A
    \end{pmatrix}
    \begin{pmatrix}
    v \\
    z
    \end{pmatrix}
     =  b - A z_0,
\end{align}
where $I$ is the $K \times K$ identity matrix. If
$\begin{pmatrix}
v \\ z
\end{pmatrix}$
is the minimal norm solution of \eqref{eq:LS:L2:v:z}, then
\begin{equation}
    z^* = z + z_0
\end{equation}
is a minimizer for \eqref{eq:LS:L2:penalized} \cite{herman_lent_hurwitz_1980, herman_2009}. One Kaczmarz step for \eqref{eq:LS:L2:v:z} and an iteration-to-row index mapping $g_A:\N \rightarrow \{1,\hdots,K\}$ is 
\begin{align}
    \label{eq:LS:L2:v:z:v:it}
    v^{j+1} 
    & = v^j + 
      \eta \tau
      \frac{b_k - A_k z_0 - \eta v^j_k - A_k z^j}{\eta^2 + \|A_k\|^2} \overline{I_{k}}^{\textrm{Tr}},\\
    \label{eq:LS:L2:v:z:z:it}
    z^{j+1} 
    & = z^j + 
      \tau 
      \frac{b_k - A_k z_0 - \eta v^j_k - A_k z^j}{\eta^2 + \|A_k\|^2} 
      \overline{A_{k}}^{\textrm{Tr}},
\end{align}
with $k=g_A(j)$ and where ${\cdot}^{\textrm{Tr}}$ is the transpose, $I_k$ is the $k$-th column of the $K \times K$ identity matrix $I$ and $A_k$ is the $k$-th row of $A$ for $1 \le k \le K$. Each step involves two equations in the complex form, and thus four equations in the real space. Hence, the iterations in \eqref{eq:LS:L2:v:z:v:it} and \eqref{eq:LS:L2:v:z:z:it} are block-iterative of block size 4, unlike the original Kaczmarz iteration (in real space) of block size 1. Because the system of equations \eqref{eq:LS:L2:v:z} is consistent, the above Kaczmarz algorithm will converge to the minimal norm solution of \eqref{eq:LS:L2:v:z} with the relaxation parameter $\tau \in (0, 2)$ and zero initial values for $v^j$ and $z^j$ \cite{jiang_ge_sart_2003}. One equivalent form of iteration in \eqref{eq:LS:L2:v:z:v:it} and \eqref{eq:LS:L2:v:z:z:it} is as follows, by choosing the initial values $v^0 = 0$ and $z^0 = z_0$ \cite{herman_lent_hurwitz_1980, herman_2009}, 
\begin{align}
    v^{j+1} 
    & = v^j + 
      \eta \tau
      \frac{b_k - A_k z^j - \eta v^j_k}{\eta^2 + \|A_k\|^2} 
      \overline{I_{k}}^{\textrm{Tr}},\\
    z^{j+1} 
    & = z^j + 
      \tau 
      \frac{b_k - A_k z^j - \eta v^j_k}{\eta^2 + \|A_k\|^2} 
      \overline{A_{k}}^{\textrm{Tr}},
\end{align}
for $k=g_A(j)$. 
Because $I_k$ is equal to 1 at its $k$-th component and zero otherwise, the above iteration can be reduced to
\begin{align}
    \label{eq:LS:L2:v:z:v:it:s}
    v^{j+1}_k 
    & = v^j_k + 
      \eta \tau 
      \frac{b_k - A_k z^j - \eta v^j_k}{\eta^2 + \|A_k\|^2}, \\
    \label{eq:LS:L2:v:z:z:it:s}
    z^{j+1} 
    & = z^j + 
      \tau 
      \frac{b_k - A_k z^j - \eta v^j_k}{\eta^2 + \|A_k\|^2} 
      \overline{A_{k}}^{\textrm{Tr}},
\end{align}
for $k=g_A(j)$. 

In the following we apply this technique to the joint reconstruction problem. Thus, let the image $c = (c_1, \cdots, c_M)^{\textrm{Tr}} \in \R^M$, the measurement $u = (u_1, \cdots, u_K)^{\textrm{Tr}} \in \C^K$, and $S_k$ is the $k$-th row of the system $S \in \C^{K \times M}$. 

\subsubsection{Regularized Kaczmarz algorithm for $J^c (c)$}
\label{subsubsec:reg:kacamarz:c}
The reconstruction functional for the image $c$ is decomposed into the following two terms
\begin{align}
    J^c (c) 
    & = J^c_{\ell^2}(c) + J^c_{\ell^1}(c),\\
    \intertext{with}
    J^c_{\ell^2}(c) & = \| S c - u\|^2 + \tilde{\alpha} ^2 |c|_2^2,\\
    J^c_{\ell^1}(c) & = \lambda |c|_1.
\end{align}

For $J^c_{\ell^2}(c)$, by applying \eqref{eq:LS:L2:v:z:v:it} and \eqref{eq:LS:L2:v:z:z:it}, one Kaczmarz step for $c$ is then
\begin{align}
    \label{eq:kacamarz:c:D}
    c^{j+1} 
    & = c^j + 
    \tau \frac{u_k - S_k c^j - \tilde{\alpha} v^j_k}{\tilde{\alpha}^2 \lambda + \|S_k\|^2} \overline{S_{k}}^{\textrm{Tr}},\\
    v^{j+1}_k 
    & = v^j_k + 
    \tilde{\alpha} \tau \frac{u_k - S_k c^j - \tilde{\alpha} v^j_k}{\tilde{\alpha}^2 \lambda + \|S_k\|^2},
\end{align}
with $k=g_A(j)$ and where ${\cdot}^{\textrm{Tr}}$ is the transpose, $v^j \in \C^K$ is auxiliary for computing the iteration of $c$. The initial values must be $c^0 = 0$ and $v^0 = 0$. One sweep of the Kaczmarz algorithm consists of applying \eqref{eq:kacamarz:c:D} for all measured data $u_k$ for $k = 1, \hdots, K$. Of course, the iteration order over the measured data can be crucial for the reconstructed image quality within finite number of iterations. Here, we use the simple order given by $g_A(j)=(j \ \mathrm{mod} \ K) +1$.  

Because the particle density $c$ is real and non-negative, we apply a projection $P_{+}$ onto the non-negative quadrant of $\R^M$ after each Kaczmarz sweep over the entire matrix. 

The $\ell^1$ regularization is applied after the projection $P_{+}$. It is performed by applying the soft thresholding operator \cite{Boyd_2011}, $T_\lambda: \R \rightarrow \R$, 
\begin{align}
    T_\lambda (c_m)
    & = (c_m - \lambda)_+ - (- c_m - \lambda)_+
\end{align}
where $(\cdot)_+=\max (0,\cdot)$.
The regularized Kaczmarz algorithm for solving \eqref{eq:alter:min:c:penalized} is summarized in Algorithm~\ref{algorithm:Kaczmarz:c}. The convergence criteria usually consist of the following: a given limit of iteration number, a threshold on the difference between the current and last iterates, such as $\|c^{j+1} - c^j\|$, or a relative difference such as $\frac{\|c^{j+1} - c^j\|}{\|c^j\|}$.


\begin{algorithm}[H]
\SetAlgoLined
\KwResult{Iterative reconstruction for partical density $c$ with given regularization parameters $\tilde{\alpha}$, $\lambda$}
 initialization\;
 choose a sweep order for measured data\;
 set $c = 0 \in \C^M$\;
 set $v = 0 \in \C^K$\;
 set $\tau = \tau_0 \in (0, 2)$\;
 \While{convergence criteria do not meet}{
    Kaczmarz sweep:
    \begin{align}
    K_\textrm{aczmarz} & \leftarrow 
    \tau \frac{u_k - S_k c - \tilde{\alpha} v_k }{\tilde{\alpha}^2 + \|S_k\|^2}\;\\
    c & \leftarrow c + K_\textrm{aczmarz}  \overline{S_{k}}^{\textrm{Tr}}\;\\
    v_k & \leftarrow v_k + \tilde{\alpha}  K_\textrm{aczmarz} \;
    \end{align}
    by the sweep order $k= 1,\hdots, K$ for measured data $u_k$.\;
    
    Projection onto the real space: 
    \begin{equation}
        c \leftarrow P_{+} [c]\;
    \end{equation}
    
    Soft thresholding:
    \begin{equation}
    c_m \leftarrow (c_m - \lambda)_+ - (- c_m - \lambda)_+\;
    \end{equation}
    for any $m=1,\hdots,M$.\;
 
  Continue\;
 }
 \caption{Kaczmarz algorithm for particle density with the $\ell^1$ and $\ell^2$ regularization terms}
 \label{algorithm:Kaczmarz:c}
\end{algorithm}

\subsubsection{Regularized Kaczmarz algorithm for $J^S (S)$}
\label{subsubsec:reg:kacamarz:S}

Let $Q_n$ be the $n$-th column of $Q$. Then we have
\begin{align}
    J^S (S)
    = 
    \sum_{k=1}^{K}
    | S_k c - u_k|^2
    +
    \sum_{k=1}^{K} 
    \sum_{n=1}^N
     | \tilde{\mu} (S_k Q_n - S_{\mathrm{calib}, k, n})|^2
    + \tilde{\gamma}^2 \|S - S_{\mathrm{mod}}\|^2_F,
\end{align}
By applying \eqref{eq:LS:L2:v:z:v:it} and \eqref{eq:LS:L2:v:z:z:it}, one Kaczmarz step for the $k$-th row $S_k$ is then either update by $c$
\begin{align}
\label{eq:kacamarz:S:c}
S_k^{j+1} 
    & = S_k^j 
    + \tau 
    \frac{u_k - S_k^j c - \tilde{\gamma} v_k^j}{\tilde{\gamma}^2 + \|c\|^2} 
      c^{\textrm{Tr}},\\
\label{eq:kacamarz:v:c}
v^{j+1}_k 
    & = v^j_k 
    + \tilde{\gamma} \tau \frac{u_k - S_k^j c - \tilde{\gamma} v^j_k}{\tilde{\gamma}^2 + \|c\|^2},
\end{align}
or update by $S_{\mathrm{calib}}$
\begin{align}
\label{eq:kacamarz:S:Scalib}
S_k^{j+1} 
    & = S_k^j 
    + \tau 
    \frac{\tilde{\mu} (S_{\mathrm{calib}, k , n} - S_k^j Q_n) - \tilde{\gamma} w^j_{k, n}}{\tilde{\gamma}^2 + \tilde{\mu}^2 \|Q_n\|^2} \tilde{\mu} \overline{Q_n}^{\textrm{Tr}},\\
\label{eq:kacamarz:v:Scalib}
w^{j+1}_{k, n} 
    & = w^j_{k, n} 
    + \tilde{\gamma} \tau 
    \frac{\tilde{\mu} (S_{\mathrm{calib}, k , n} - S_k^j Q_n) - \tilde{\gamma} w^j_{k, n}}{\tilde{\gamma}^2 + \tilde{\mu}^2 \|Q_n\|^2},
\end{align}
with $n=g_Q(j)$ where $g_Q:\N \rightarrow \{1,\hdots,N\}$ defines the order of  the columns of $Q$ during the iteration.
Note that we need two auxiliary variables $v^j \in \C^K$  and $w^j \in \C^{K \times N}$. The initial values must be $S^0 = S_{\mathrm{mod}}$ and $v^0 = 0$. One sweep of the Kaczmarz algorithm can be performed by updating any row $S_k$, $k=1,\hdots,K$, with $c$ by applying \eqref{eq:kacamarz:S:c} and \eqref{eq:kacamarz:v:c} and by updating with $S_{\mathrm{calib}}$ by applying \eqref{eq:kacamarz:S:Scalib} and \eqref{eq:kacamarz:v:Scalib} for all calibration data measured data $S_{\mathrm{calib}, k , n}$ for $n = 1, \cdots, N$, i.e., $g_Q(j)=(j \ \mathrm{mod} \ N) +1$. Again, the sweeping order of the Kaczmarz algorithm should not be overlooked. The regularized Kaczmarz algorithm for solving \eqref{eq:alter:min:S:penalized} is summarized in Algorithm~\ref{algorithm:Kaczmarz:S}.

\begin{algorithm}[h]
\SetAlgoLined
\KwResult{Iterative reconstruction for system matrix $S$ with given regularization parameters $\tilde{\gamma}$, $\tilde{\mu}$}
 initialization\;
 choose a sweep order for measured data\;
 choose a sweep order for calibration data\;
 set $S = S_{\mathrm{mod}} \in \C^{K \times M}$\;
 set $v = 0 \in \C^K$\;
 set $w = 0 \in C^{K \times N}$\;
 set $\tau = \tau_0$\;
 set $\eta = \eta_0$\;
 
 \While{convergence criteria do not meet}{
Kaczmarz sweep by $c$
    \begin{align}
    K_\textrm{aczmarz} & \leftarrow 
    \tau \frac{u_k - S_k c - \tilde{\gamma} v_k}{\tilde{\gamma}^2 + \|c\|^2},\\
    S_k & \leftarrow S_k + K_\textrm{aczmarz}  c^{\textrm{Tr}},\\
    v_k & \leftarrow v_k + \tilde{\gamma}  K_\textrm{aczmarz},
    \end{align}
    for any $k= 1,\cdots, K$\;
    
Kaczmarz sweep by $Q$
    \begin{align}
    K_\textrm{aczmarz} & \leftarrow 
    \tau \frac{\tilde{\mu} (S_{\mathrm{calib}, k , n} - S_k Q_n) - \tilde{\gamma} w_{k, n}}{\tilde{\gamma}^2 + \tilde{\mu}^2 \|Q_n\|^2},\\
    S_k & \leftarrow
    S_k + K_\textrm{aczmarz}  \tilde{\mu}  \overline{Q_n}^{\textrm{Tr}},\\
    w_{k, n} & \leftarrow w_{k, n} + \tilde{\gamma}  K_\textrm{aczmarz},
    \end{align}
    by the sweep order $n=1,\hdots,N$ for any $k= 1,\hdots, K$\;

  Continue\;
 }
 \caption{Kaczmarz algorithm for system matrix}
 \label{algorithm:Kaczmarz:S}
\end{algorithm}



\section{Numerical results}
We illustrate the proposed method by  numerical examples including an academic test problem and the application to the imaging problem in MPI. 
Results using a standard integral operator are presented to highlight the characteristic behavior of the method for optimal parameter settings obtained via $\ell^2$-error or SSIM.
Furthermore, we provide numerical results for measured phantom data in MPI.

\subsection{Academic example - integral operator}
We investigate the behavior of the proposed method first in an academic test example.
Here, we choose a discretized standard problem defined by the integral operator in \eqref{eq:example_integral_operator} for $\Omega,I=(0,T)$, $T\in\N$, and $s(x,t)=\chi_{[0,x)}(t)$, i.e., the corresponding discretized bilinear operator $\tilde{B}:\R^M \times \R^{K\times M} \rightarrow \R^K$  is given by $\tilde{B}(S,c)=Sc$.
By choosing $K=M=T$, an equidistant grid, and piecewise constant basis functions, the true operator $S^\ast\in \R^{K\times M}$ is thus a lower triangular matrix filled with ones. Please not that $\cdot^\ast$ does not denote the adjoint matrix in this subsection. 
$S_\epsilon$ is obtained by adding a Gaussian matrix $\eta \in \R^{K\times M}$ with entries being i.i.d. and normally distributed with zero mean and standard deviation $\sigma$, i.e., $\eta_{i,j} \sim \mathcal{N}(0,\sigma)$ for any $i=1,\hdots,K$, $j=1,\hdots,M$. 
Analogously, we obtain the noisy measurement $u_\delta = u^\ast + \xi$ where the noise $\xi \in \R^K$ is also i.i.d. and normally distributed with zero mean and standard deviation $\sigma$.
The calibrated matrix $S_\mathrm{calib}\in \R^{K\times N}$, $N=\frac12 M$, on a coarser resolution is obtained using the operator $\tilde{P}$, respectively the matrix $Q\in \R^{M\times N}$ which encodes computing the sum of two consecutive columns of the high resolution matrix, i.e., $Q$ is a sparse matrix with two ones in each column.
Each column of the high resolution matrix contributes only to one single column of the low resolution system matrix. 
We thus use $S_{\mathrm{mod},\epsilon}=S_\epsilon =S^\ast +\eta$, $S_\mathrm{calib}=S^\ast Q$, and $u_\delta = u^\ast + \xi$ in \eqref{eq:joint:min:c:S}.
A solution for $M=50$ is obtained by minimizing the functional in \eqref{eq:joint:min:c:S} using the Kaczmarz-type method outlined in Section \ref{sec:algorithm}.
Here, in each outer iteration of the alternating minimization problem we us 500 Kaczmarz sweeps for $J^c(c)$ and 300 ones for $J^S(S)$.
In total 100 outer iterations are computed. 
In the following we compare 
\begin{itemize} 
\item the $(c,S)$-reconstruction, which is the joint reconstruction of $c$ and $S$ minimizing $J$,
\item the sole $c$-reconstruction minimizing $J^c$ for fixed $S=S^\ast$, which is the ideal and desired situation where the operator is accurately known, and
\item the sole $c$-reconstruction minimizing $J^c$ for fixed $S=S_\epsilon$, which represents the other end of the range where the noisy/inaccurate operator is considered as true.
\end{itemize}
In order to compare the different methods we performed a discrete regularization parameter search for $\gamma \in \{2^{-i}|i=0,\hdots, 18 \}$, $\mu \in \{2^{-i}|i=0,\hdots, 18 \}$, $\alpha \in \{2^{-i}|i=10,\hdots, 18  \}$, and $\lambda \in \{2^{-i}|i=1,\hdots, 12  \}$ (38988 parameter combinations in total) optimizing either the $\ell^2$-error or the SSIM when compared to the true solution $c^\ast$. 
For the $(c,S)$-reconstruction the last outer iterate is used to determine the optimal parameter.

The results for the $\ell^2$-error optimized regularization parameters for varying noise levels (in terms of standard deviation $\sigma$) are illustrated in Figure \ref{fig:integral_l2_recos} and Table \ref{tab:integral_l2_params}. 
In all three cases the $(c,S)$-reconstruction method improves the reconstruction quality quantitatively when compared to the $c$-reconstruction ($S=S_\epsilon$) which is the desired impact of the joint reconstruction.
The $(c,S)$-reconstructions tend to approximate the true solution $c^\ast$ quantitatively and qualitatively better but do not reach the same error level which is due the noisy nature of the operator and thus is not expected.
When decreasing the noise level $\sigma$, we can also observe a decrease of the $\ell^2$-error for the $c$-reconstruction ($S=S^\ast$) as predicted by the theory.
The $(c,S)$-reconstruction also follows this trend reaching the $\ell^2$-error values 0.3158, 0.1124, and 0.1027 in the last outer iterate for decreasing $\sigma$.

Analogous observations can be made when using SSIM to obtain optimal parameters which are illustrated in Figure \ref{fig:integral_ssim_recos} and Table \ref{tab:integral_ssim_params}.

\begin{figure}
    \centering
    \begin{tabular}{c|c|c}
    \rotatebox{90}{$\sigma=0.05$} & 
                    \includegraphics[width=0.45\textwidth]{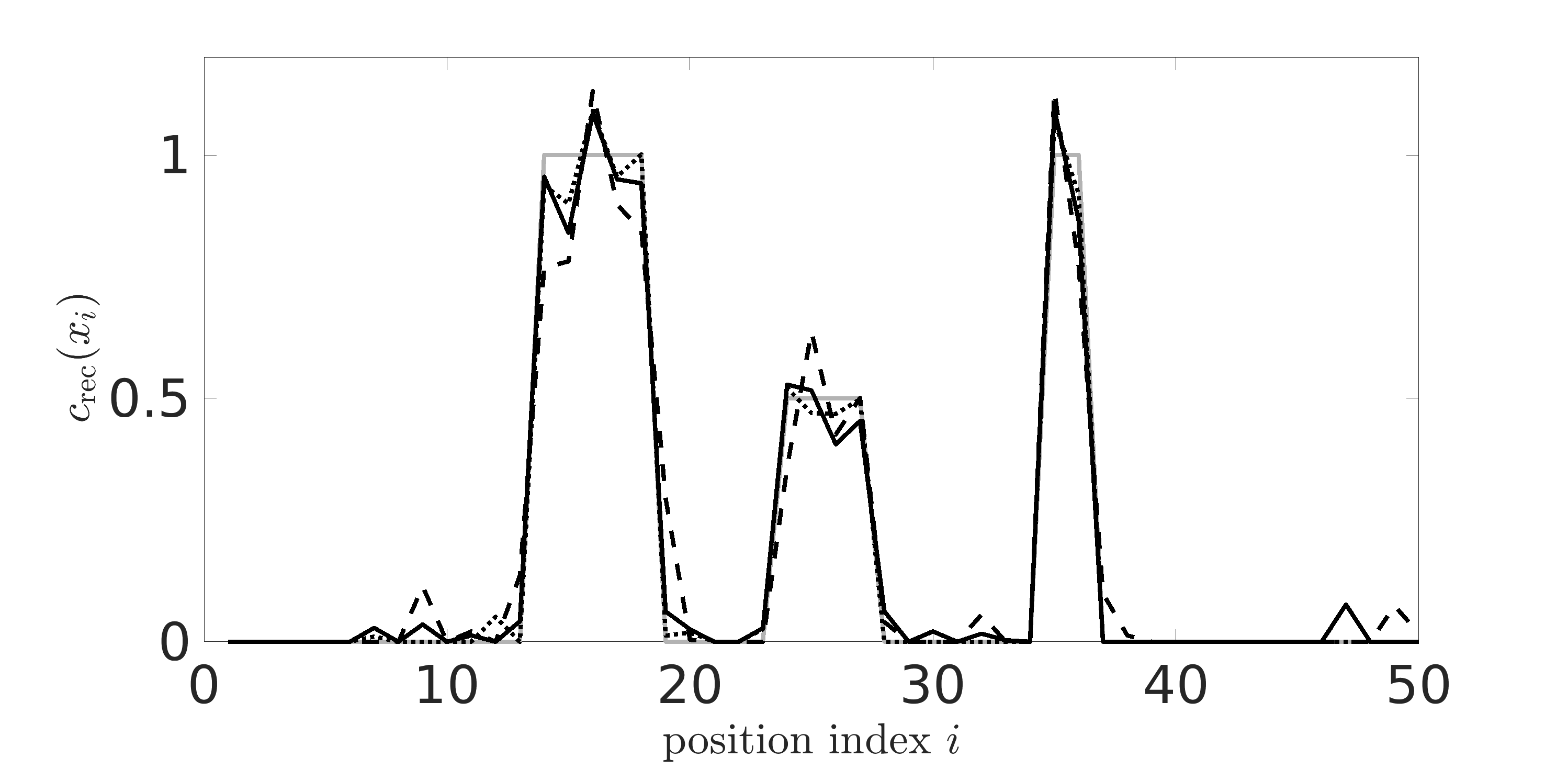} &
    \includegraphics[width=0.45\textwidth]{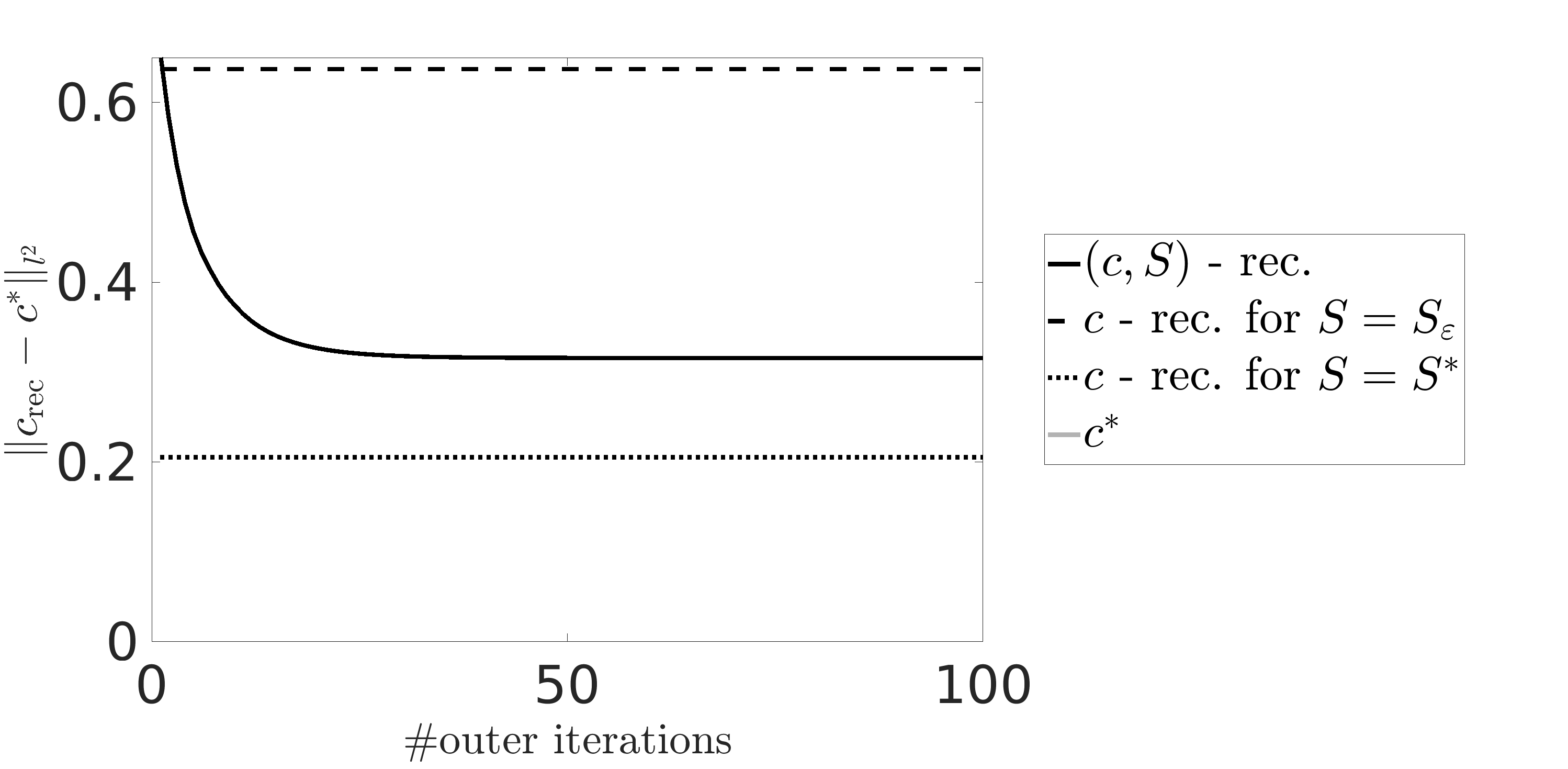} \\

        \hline
            \rotatebox{90}{$\sigma=0.025$} & 
                            \includegraphics[width=0.45\textwidth]{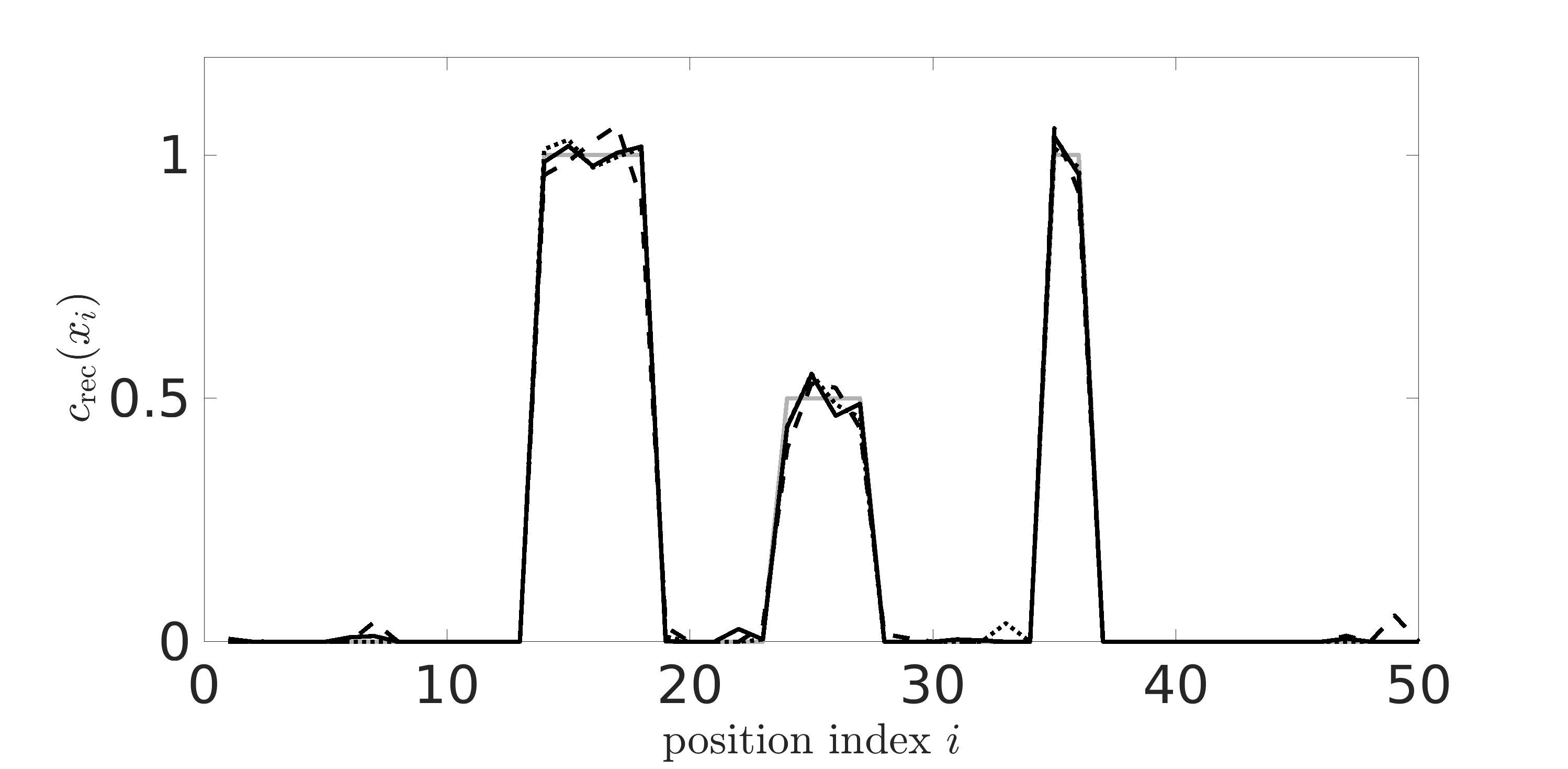} &
    \includegraphics[width=0.45\textwidth]{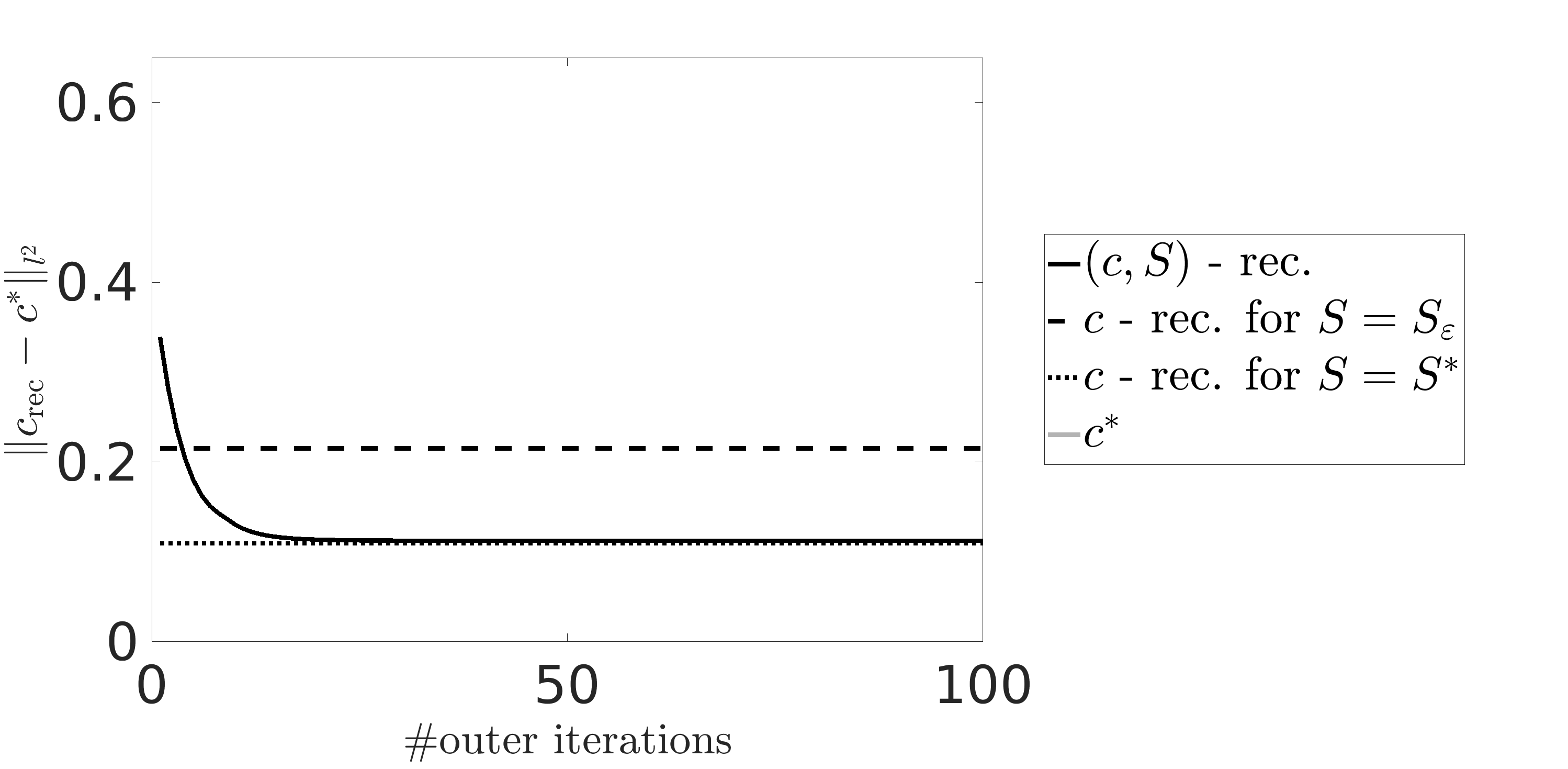} \\

        \hline
            \rotatebox{90}{$\sigma=0.0125$} & 
                            \includegraphics[width=0.45\textwidth]{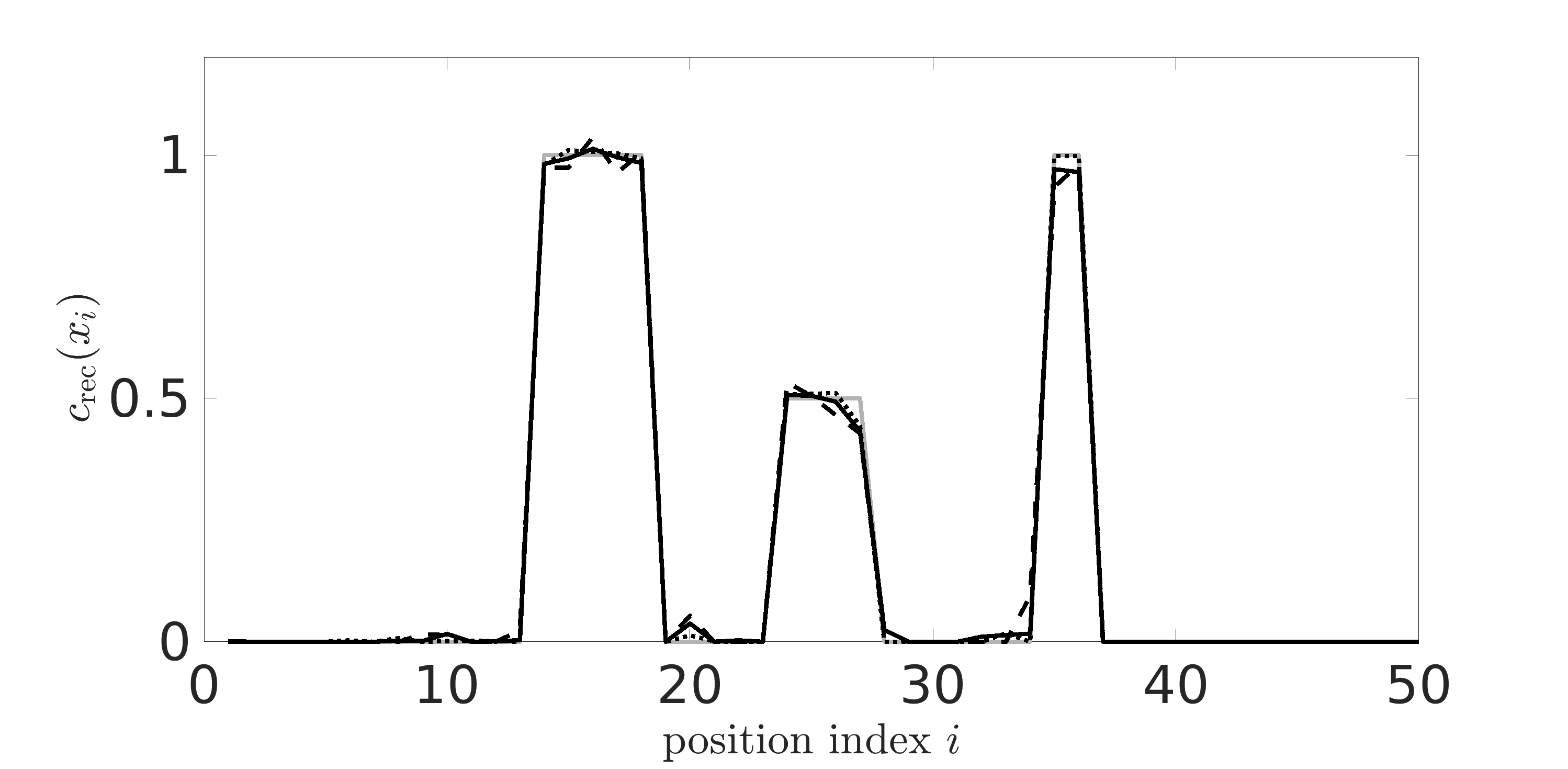} &
    \includegraphics[width=0.45\textwidth]{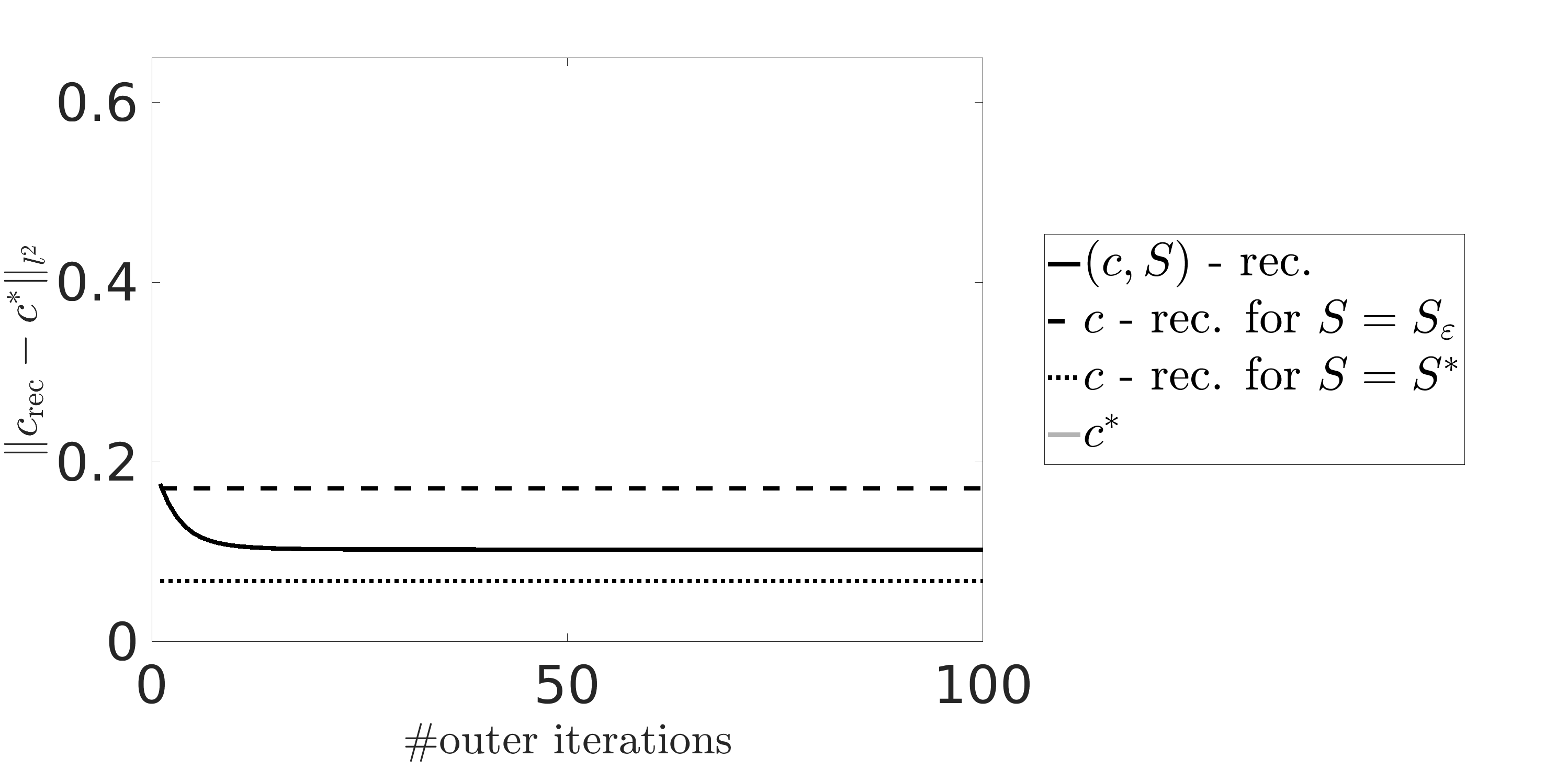} \\

    \end{tabular}

    \caption{Phantom reconstructions (left) and $\ell^2$-error (right) of proposed method for decreasing standard deviation from top to bottom. Regularization parameters are chosen such that the $\ell^2$-(reconstruction) error is minimized. For $(c,S)$-reconstruction method the last outer iteration is used to determine the regularization parameters, which can be found in Table \ref{tab:integral_l2_params}. }
    \label{fig:integral_l2_recos}
\end{figure}

\begin{figure}
    \centering
    \begin{tabular}{c|c|c}
    \rotatebox{90}{$\sigma=0.05$} & 
                    \includegraphics[width=0.45\textwidth]{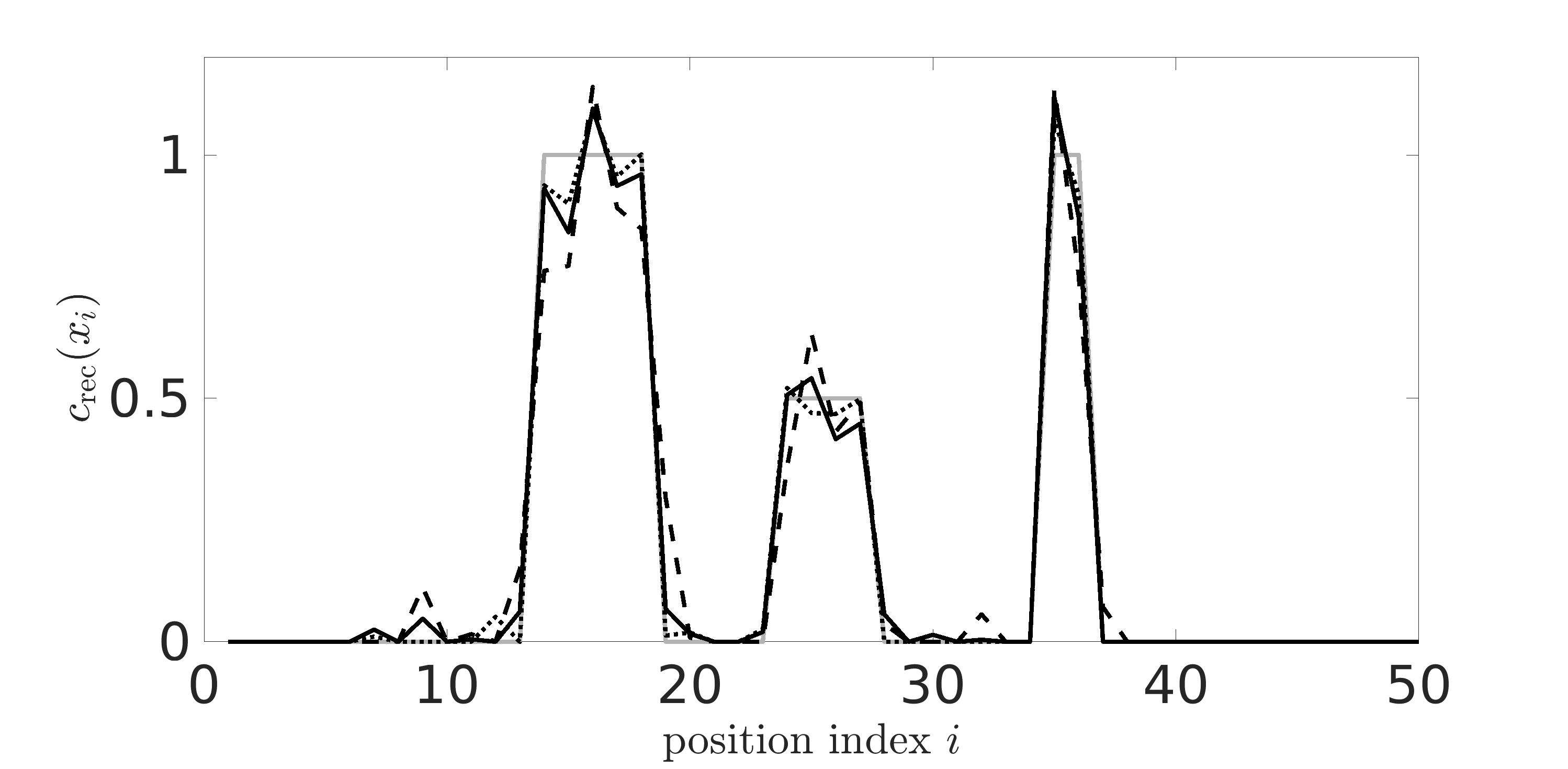} &
    \includegraphics[width=0.45\textwidth]{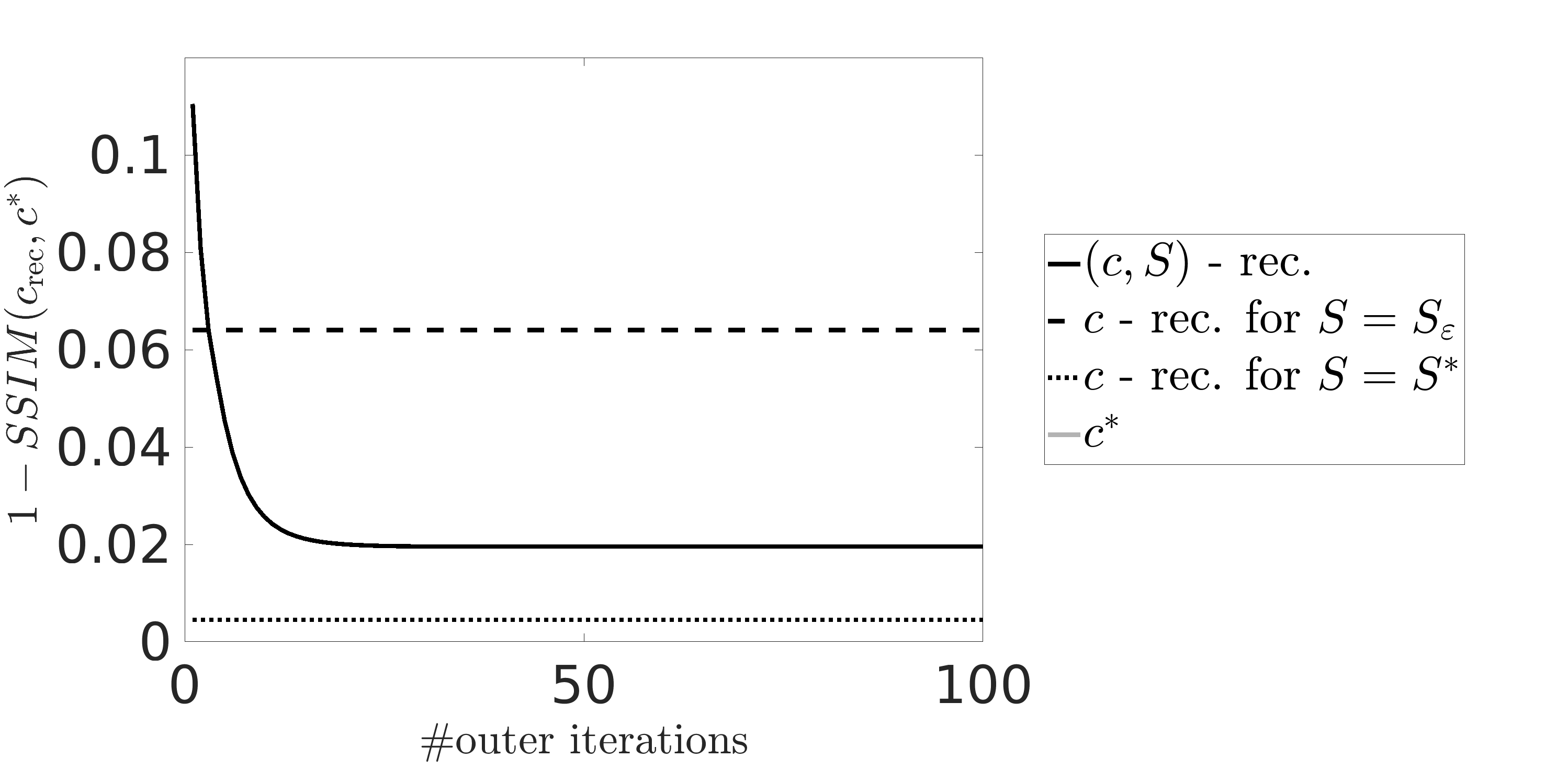} \\

        \hline
            \rotatebox{90}{$\sigma=0.025$} & 
                            \includegraphics[width=0.45\textwidth]{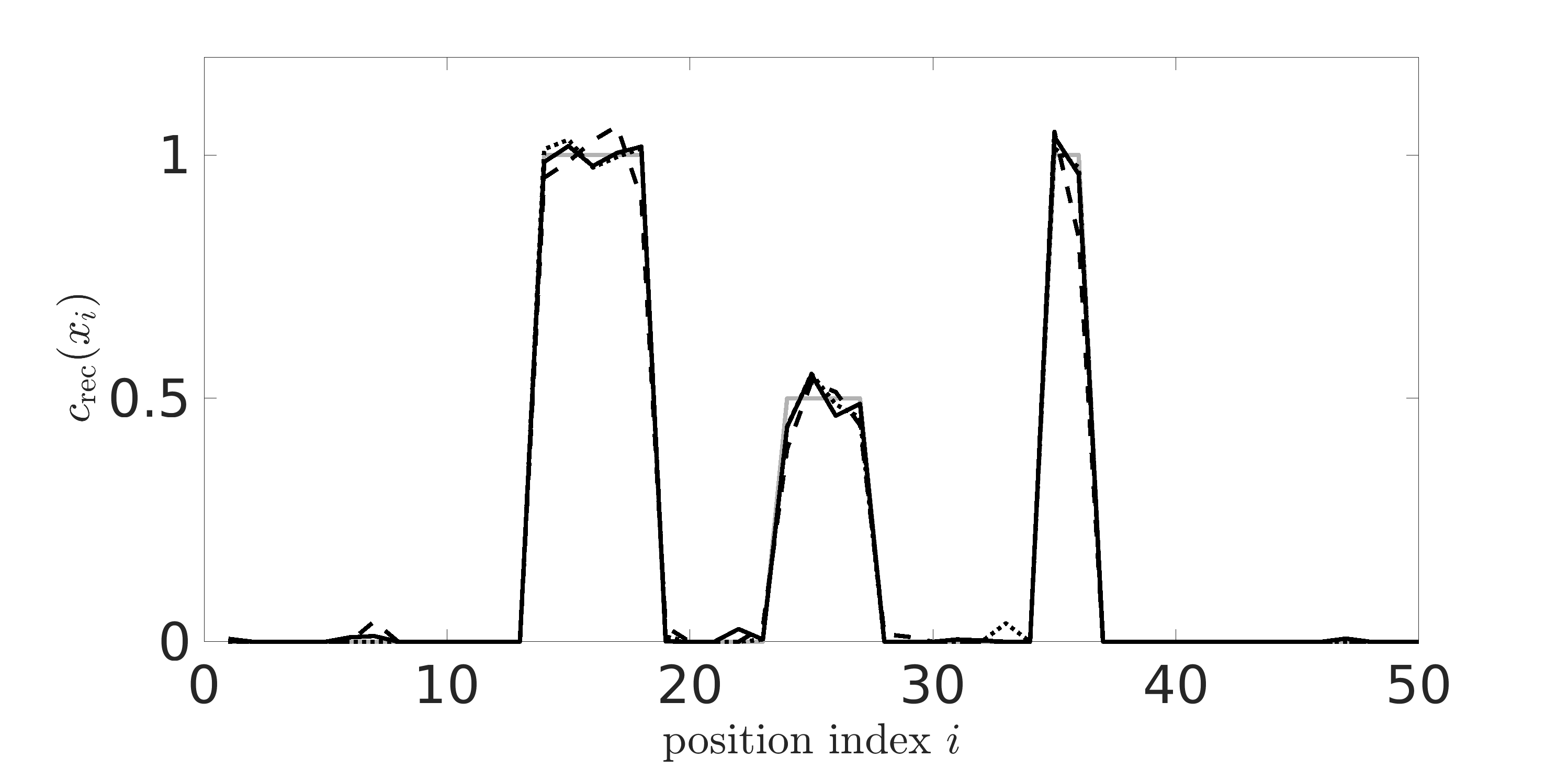} &
    \includegraphics[width=0.45\textwidth]{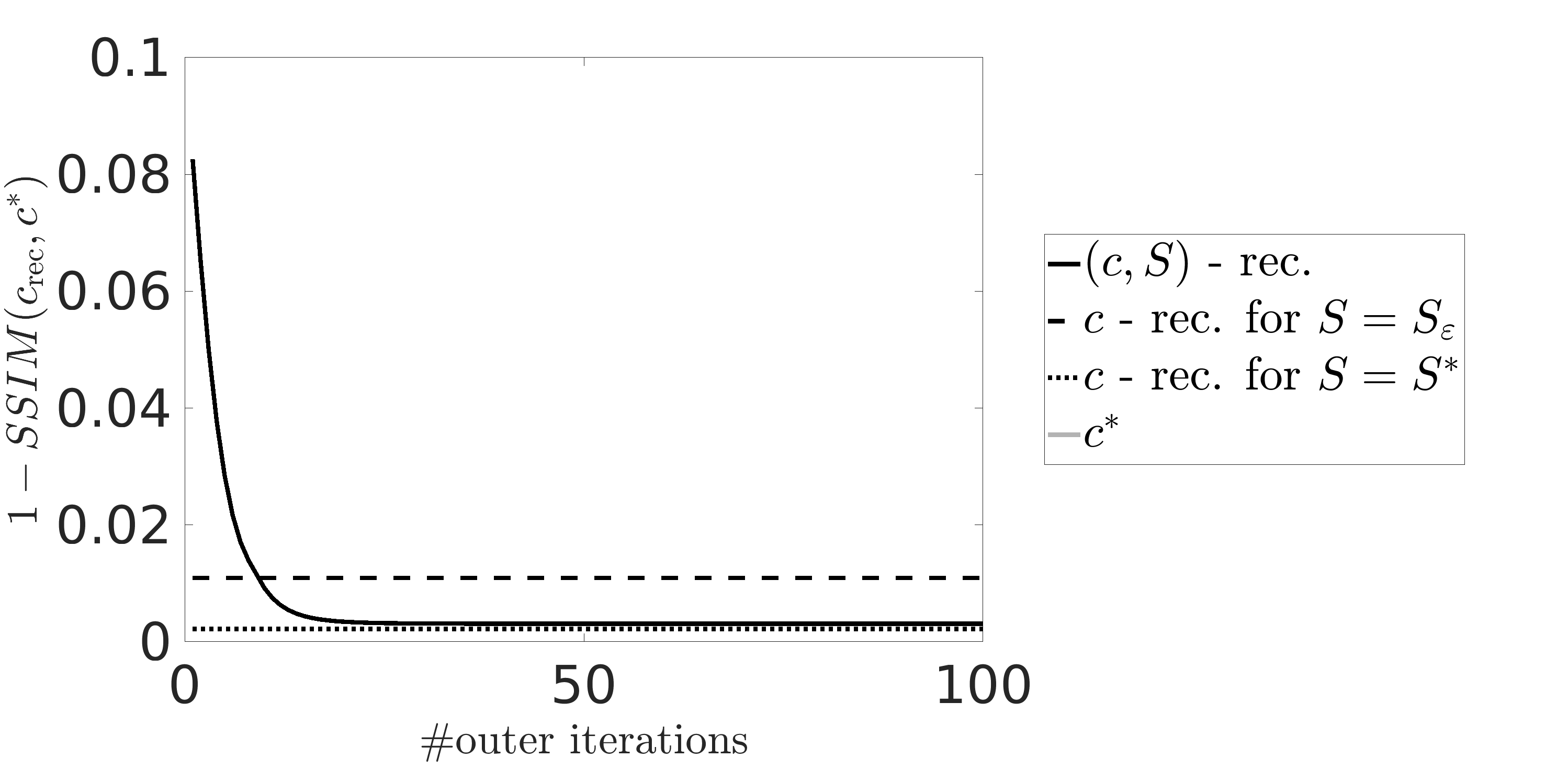} \\

        \hline
            \rotatebox{90}{$\sigma=0.0125$} & 
                            \includegraphics[width=0.45\textwidth]{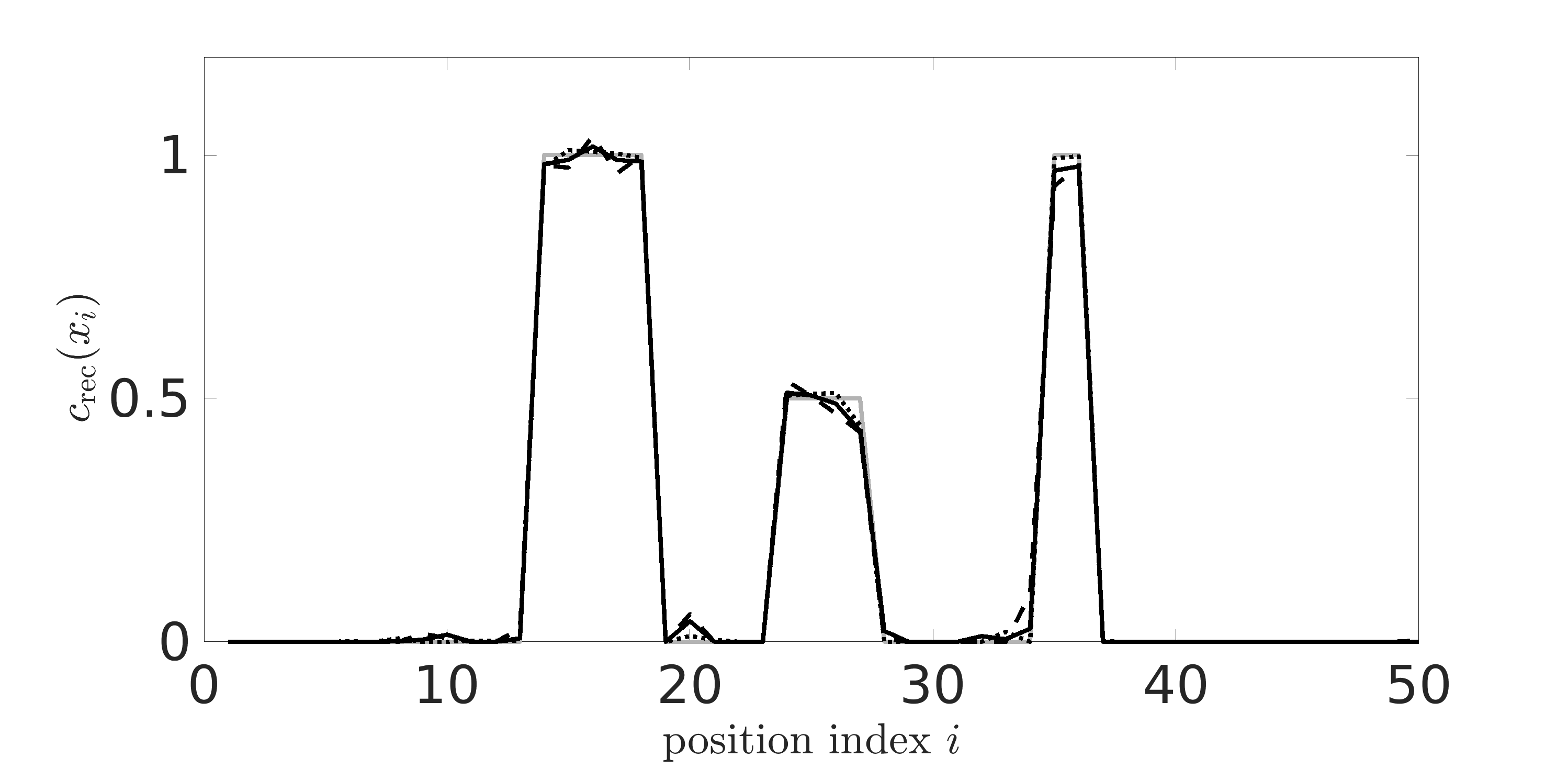} &
    \includegraphics[width=0.45\textwidth]{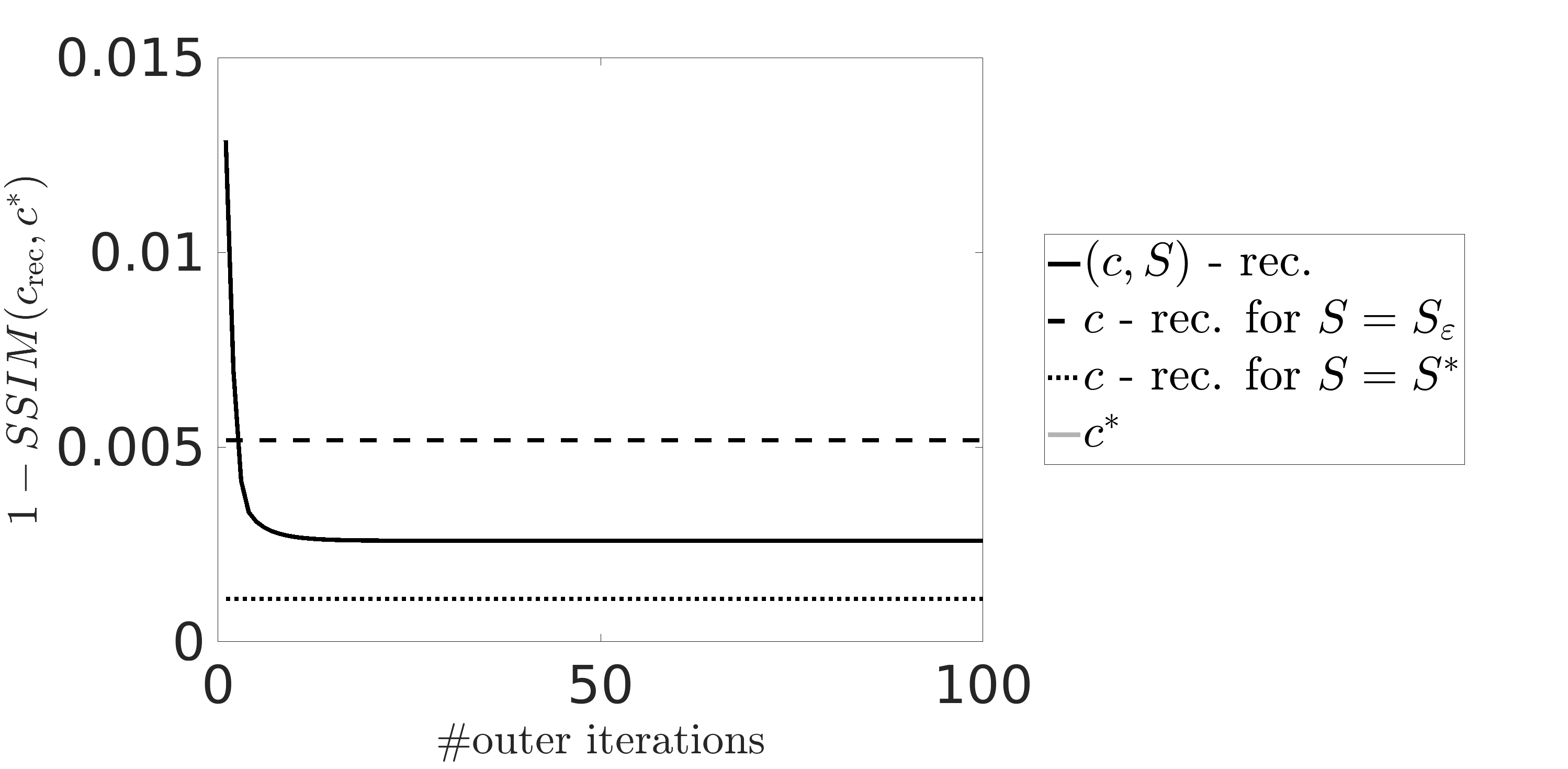} \\

    \end{tabular}
    \caption{Phantom reconstructions (left) and SSIM-measure (right) of proposed method for decreasing standard deviation from top to bottom. Regularization parameters are chosen such that the $1-SSIM$ is minimized. For $(c,S)$-reconstruction method the last outer iteration is used to determine the regularization parameters, which can be found in Table \ref{tab:integral_ssim_params} }
    \label{fig:integral_ssim_recos}

\end{figure}

\begin{table}
    \centering
    \begin{tabular}{c|l|c|c|c|c}
    \hline
 $\sigma$ & Method  & $\gamma$ & $\mu$ & $\alpha$ &$\lambda$ \\
    \hline
 &  $(c,S)$ - rec.       & $0.25$ & $1.0$ & $1.53 \times 10^{-5}$ & $4.88 \times 10^{-4}$ \\
 0.05 &  $c$ - rec., $S=S_\epsilon$    & -- & -- & $6.10\times 10^{-5}$  & $9.77\times 10^{-4}$  \\
 &   $c$ - rec., $S=S^\ast$    & -- & -- & $3.05\times 10^{-5}$  & $3.9\times 10^{-3}$  \\
    \hline
 &   $(c,S)$ - rec.       & $0.50$ & $1.0$ & $3.82\times 10^{-6}$ & $4.88\times 10^{-4}$\\
  0.025 & $c$ - rec., $S=S_\epsilon$    & -- & -- & $3.05\times 10^{-5}$ & $3.9 \times 10^{-3}$ \\
   & $c$ - rec., $S=S^\ast$    & -- & -- & $7.6\times 10^{-6}$ & $2.0 \times 10^{-3}$ \\
    \hline
 &   $(c,S)$ - rec.       & $0.50$ & $1.0$ & $3.82\times 10^{-6}$ & $2.44\times 10^{-4}$\\
 0.0125  & $c$ - rec., $S=S_\epsilon$    & -- & -- & $7.63\times 10^{-6}$ & $9.77\times 10^{-4}$ \\
  &  $c$ - rec., $S=S^\ast$    & -- & -- & $3.82\times 10^{-6}$ & $2.44\times 10^{-4}$\\
    \hline
    \end{tabular}
    \caption{Optimal regularization parameters with respect to $\ell^2$-(reconstruction) error.}
    \label{tab:integral_l2_params}
\end{table}

\begin{table}
    \centering
    \begin{tabular}{c|l|c|c|c|c}
    \hline
 $\sigma$ & Method  & $\gamma$ & $\mu$ & $\alpha$ &$\lambda$ \\
    \hline
  & $(c,S)$ - rec.       & $0.5$ & $1.0$ & $1.53 \times 10^{-5}$ & $9.77 \times 10^{-4}$ \\
 0.05 &  $c$ - rec., $S=S_\epsilon$    & -- & -- & $6.10\times 10^{-5}$  & $3.9\times 10^{-3}$  \\
  &  $c$ - rec., $S=S^\ast$    & -- & -- & $3.05\times 10^{-5}$  & $3.9\times 10^{-3}$  \\
    \hline
  & $(c,S)$ - rec.       & $0.50$ & $1.0$ & $3.82\times 10^{-6}$ & $4.88\times 10^{-4}$\\
 0.025 &  $c$ - rec., $S=S_\epsilon$    & -- & -- & $3.05\times 10^{-5}$ & $7.8 \times 10^{-3}$ \\
  &  $c$ - rec., $S=S^\ast$    & -- & -- & $7.63\times 10^{-6}$ & $2.0 \times 10^{-3}$ \\
    \hline
  & $(c,S)$ - rec.       & $1.0$ & $1.0$ & $3.82\times 10^{-6}$ & $2.44\times 10^{-4}$\\
 0.0125 &  $c$ - rec., $S=S_\epsilon$    & -- & -- & $3.82\times 10^{-6}$ & $9.77\times 10^{-4}$ \\
  &  $c$ - rec., $S=S^\ast$    & -- & -- & $7.63\times 10^{-6}$ & $2.44\times 10^{-4}$\\
  \hline
    \end{tabular}
    \caption{Optimal regularization parameters with respect to $1-SSIM$.}
    \label{tab:integral_ssim_params}--
\end{table}

\subsection{Application to magnetic particle imaging}

As a second example we consider MPI which also motivated the general  problem setup of the present work.
Precisely modeling MPI, resp.\ formulating a physically accurate integral kernel for image reconstruction is still an unsolved problem. Various modeling aspects, e.g., the magnetization dynamics and particle-particle interactions, make it a challenging task such that the integral kernel is commonly determined in a time-consuming calibration procedure. For further information on the modeling aspects, the interested reader is referred to the survey paper \cite{Kluth2018b} as well as to the review article \cite{Knopp2017} for further details on the MPI methodology. 
The numerical results are obtained from the recently improved modeled approach in \cite{KluthSzwargulskiKnopp2019} and the real data example in a 2D field-free-point (FFP) setup therein.
Thus, the setup is as follows: 
     $S_\mathrm{mod}\in \C^{K \times M}$, respectively $S_{\mathrm{mod},\epsilon}$ for one particular $\epsilon$ is given by the improved model B3 in \cite{KluthSzwargulskiKnopp2019} which exploits a space-dependent anisotropy in a N\'{e}el rotation model for ensembles of nanoparticles. The authors fitted the analog filter function to calibration measurements in a previous step. For this work we exploited this model to obtain the refined resolution of $M=60 \times 60$ voxels corresponding to voxels of size 0.5mm $\times$ 0.5mm $\times$ 1mm. 
      
     In the measurement process of the time-dependent voltage signal, we have $K_\mathrm{max}= 2 \times 817$ frequencies available after applying the Fourier transform (2 channels  $\times$ the entire available spectrum). To obtain reasonable reconstructions the frequencies are restricted to a subset in a preprocessing step (a combination of SNR and NRMSE thresholding, see \cite{KluthSzwargulskiKnopp2019} for further details) resulting in $K\leq K_\mathrm{max}$.
     $S_\mathrm{calib}\in \R^{K\times N}$ is obtained in a calibration procedure with a delta sample of size 1mm $\times$ 1mm $\times$ 1mm resulting in $30 \times 30=N$ voxels. $K$ is as described for $S_\mathrm{mod}$.
     
     The real phantom consisted of 3 capillary filled with tracer having a concentration of 250 mmol/l. In total a tracer volume of $2.19 \times 10^{-5}$l was used. For a photo  of the phantom we refer to \cite[Fig. 6]{KluthSzwargulskiKnopp2019}. $u_\delta\in \C^K$ is thus the MPI measurement after applying the previously described frequency selection.
     
     Solutions are obtained by minimizing the functional in \eqref{eq:joint:min:c:S} using the Kaczmarz-type method outlined in Section \ref{sec:algorithm}.
We use 10 outer iterations each comprising 75 Kaczmarz sweeps for $J^c(c)$ and 20 ones for $J^S(S)$.
Due to the missing ground truth for the system matrix, we cannot compare any reconstruction to the $c$-reconstruction using $S=S^\ast$. As an alternative we compare them to low resolution reconstructions using $S=S_\mathrm{calib}$ which also represents the standard reconstruction in MPI. 
Thus, in the following we compare 
\begin{itemize} 
\item the $(c,S)$-reconstruction, which is again the joint reconstruction of $c$ and $S$ minimizing $J$,
\item the sole $c$-reconstruction minimizing $J^c$ for fixed $S=S_\mathrm{mod}$, which represents the pure model-based reconstruction on the refined resolution, and
\item the sole $c$-reconstruction minimizing $J^c$ for fixed $S=S_\mathrm{calib}$, which represents the standard reconstruction method in MPI.
\end{itemize}
For comparison we computed reconstructions for various parameter combinations, i.e., $\gamma \in \{10^{-i}|i=0,\hdots, 5 \}$, $\mu \in \{10^{-i}|i=0,\hdots, 5 \}$, $\alpha \in \{2^{-i}|i=28,30,\hdots, 48  \}$, and $\lambda \in \{2^{-i}|i=3,\hdots, 10  \}$, and as a ground truth phantom is not available for computing an image quality measure, we exploited the quantitative information on the volume of the used tracer material. We sorted the reconstructions by their absolute deviation to the desired volume in an ascending order which are illustrated in Figure \ref{fig:MPI_reconstructions} and Table \ref{tab:MPI_reg_parameters}.

For the $(c,S)$-reconstructions we can observe that the capillary reconstructions become sharper when increasing the number of outer iterations. Two main kinds of reconstructions can be found for the $(c,S)$-reconstruction in the illustrated results. 
Rows 1-4 result from a large $\mu$ while rows 5-8 are based on a small $\mu$. This shows the influence of the reconstructed system matrix on the reconstruction when changing the influence of the term including the additional information on the operator $P$. 
For large $\mu$ (rows 1-4) and thus a smaller influence of the regularization term including $S_\mathrm{mod}$, we can observe high frequent noise patterns in the reconstructions. 
When decreasing $\mu$ we obtain improved reconstructions (rows 5-8), which illustrates the importance of additional a priori information in the reconstruction process, either provided via $S_\mathrm{mod}$ or another penalty term ${R}_s$ which may include further a priori information on $S$. 
For pure $c$-reconstruction using $S=S_\mathrm{mod}$ the best reconstructions (by visual judgement) can be found in rows 1 and 6. These reconstructions are of similar quality when compared to the $(c,S)$-reconstruction, while the latter one results in slightly sharper reconstructions. 
Compared to the $c$-reconstruction ($S=S_\mathrm{calib}$) on the coarser grid, both high resolution reconstruction methods improve the separation of the capillaries in the phantom.


\begin{figure}
    \centering
\begin{tabular}{p{0.5cm}|c|ccc|c}
\textbf{row no.} &\textbf{$c$-rec., $S=S_\mathrm{mod}$}& \multicolumn{3}{c|}{\textbf{$(c,S)$ - rec.}} & \textbf{$c$-rec., $S=S_\mathrm{calib}$} \\
\hline 
& -- & 1 iter. & 5 iter. & 10 iter. & -- \\
\hline
 1 &
\includegraphics[width=0.16\textwidth]{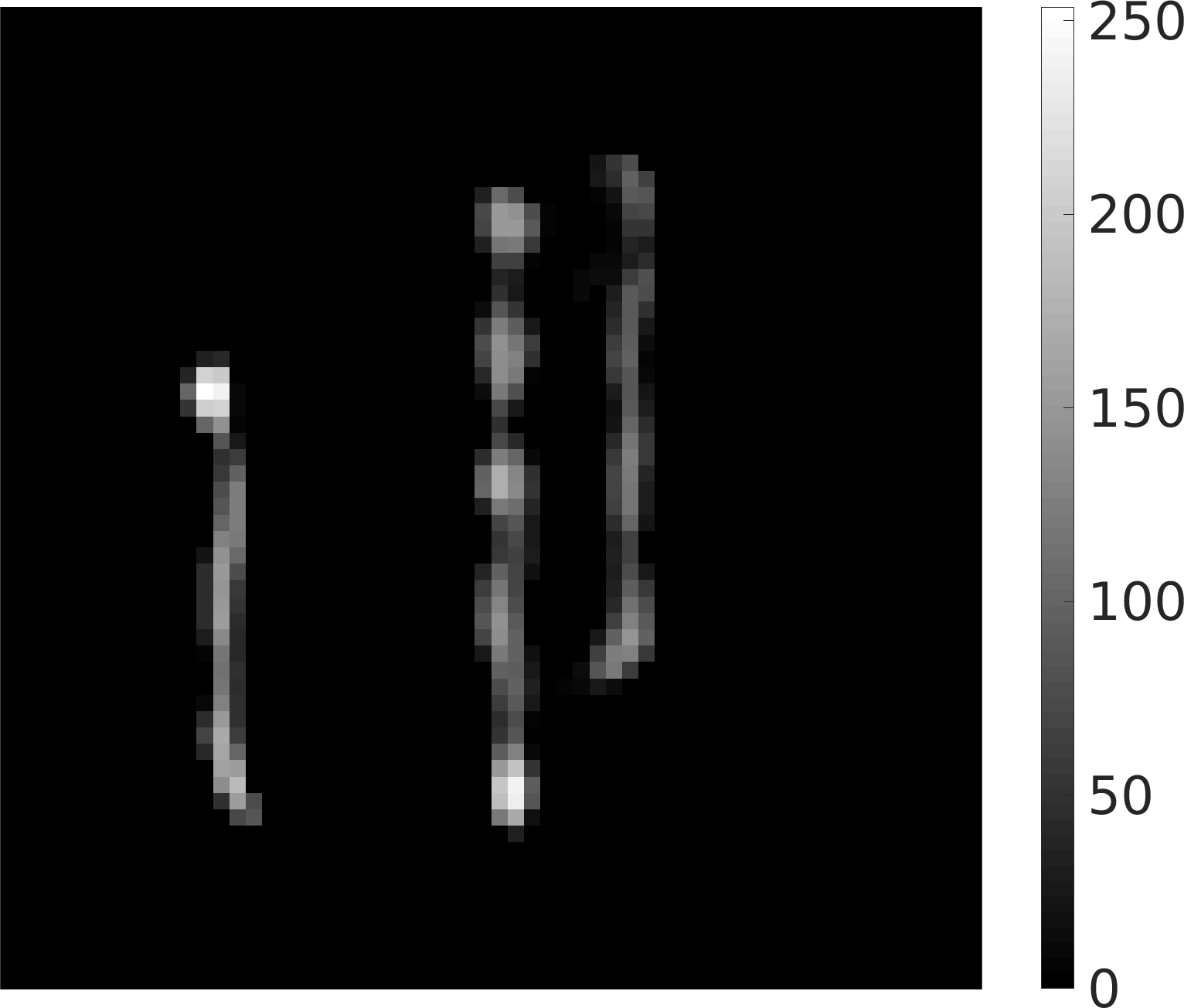}&
\includegraphics[width=0.16\textwidth]{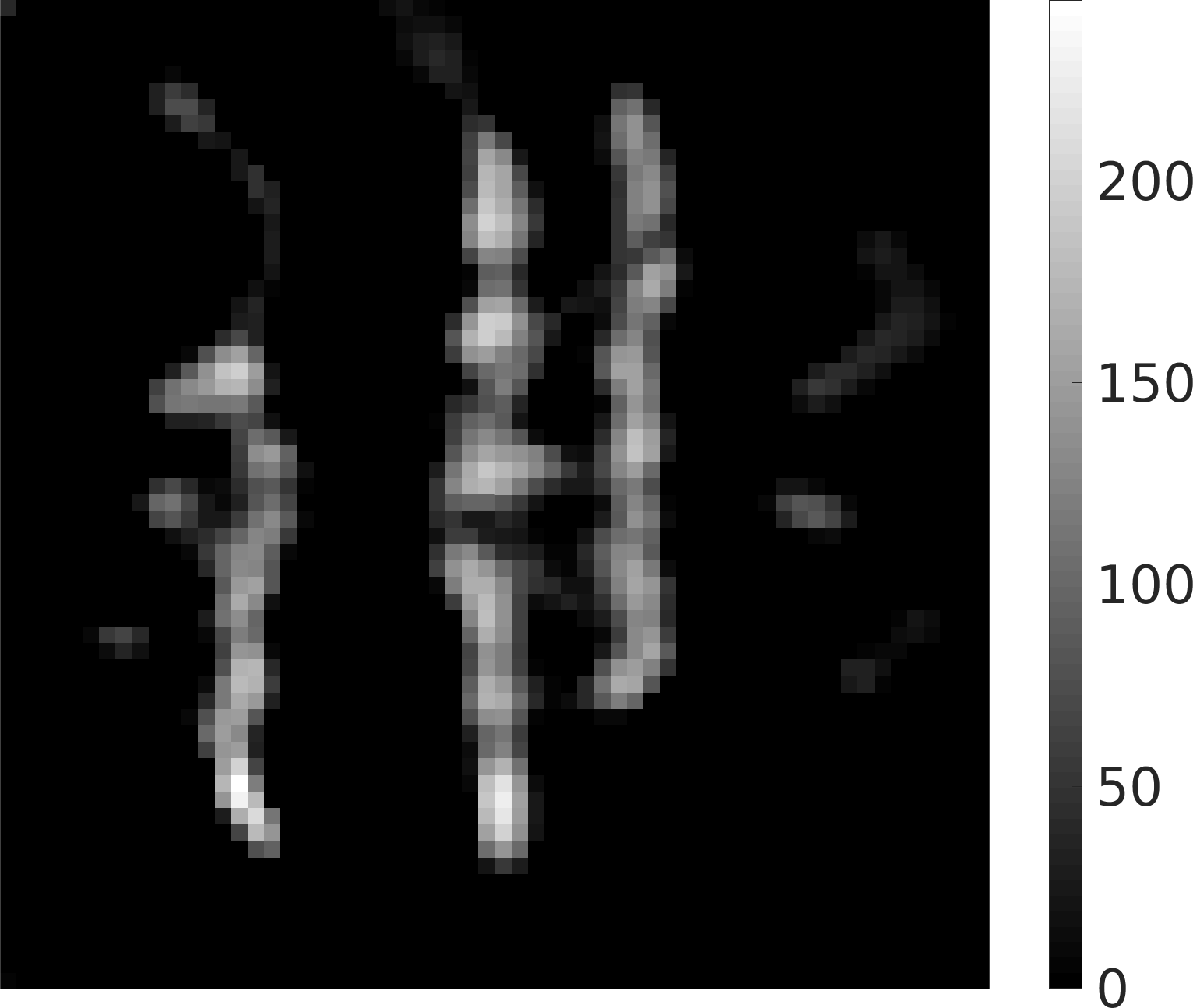}&
\includegraphics[width=0.16\textwidth]{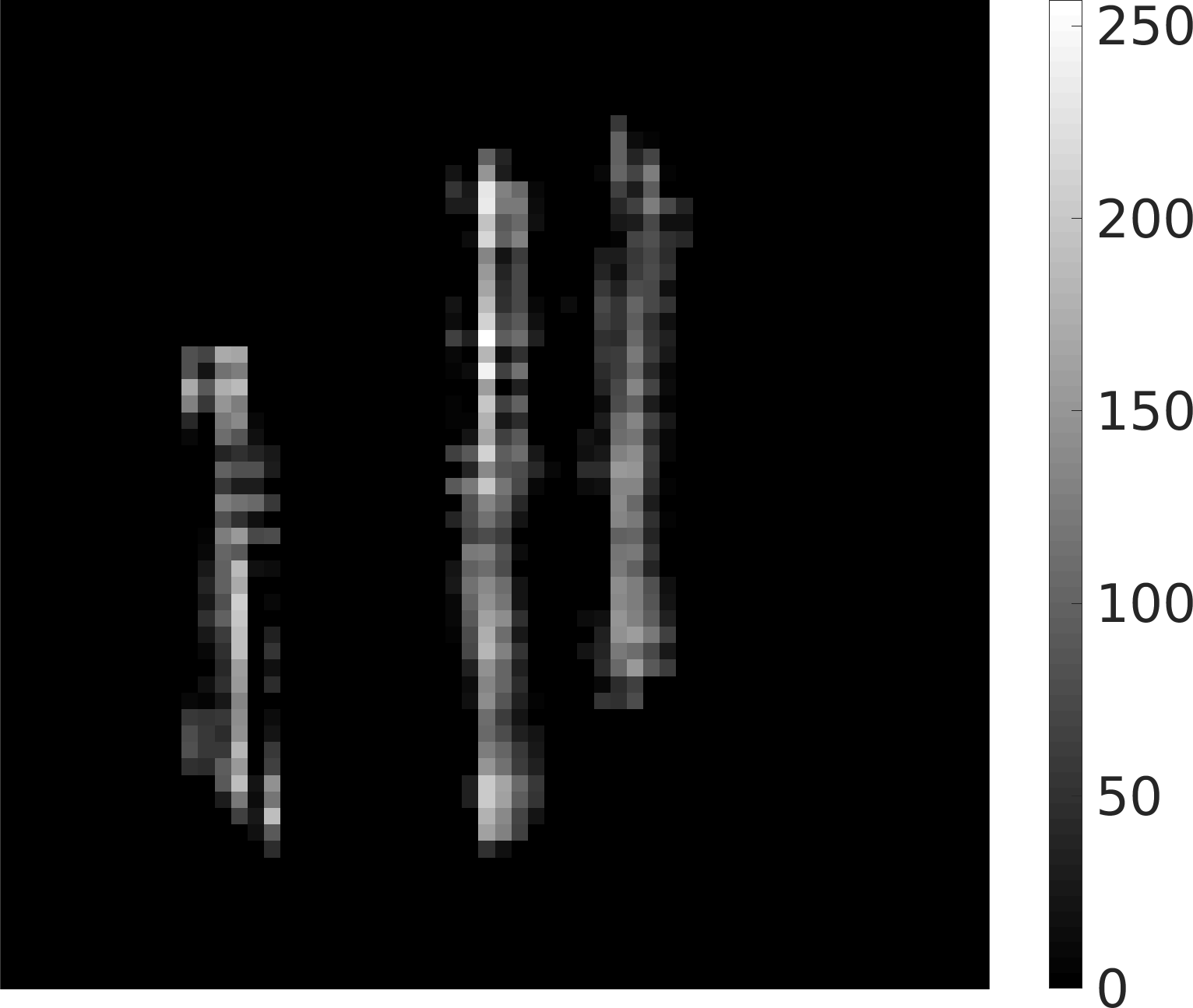}&
\includegraphics[width=0.16\textwidth]{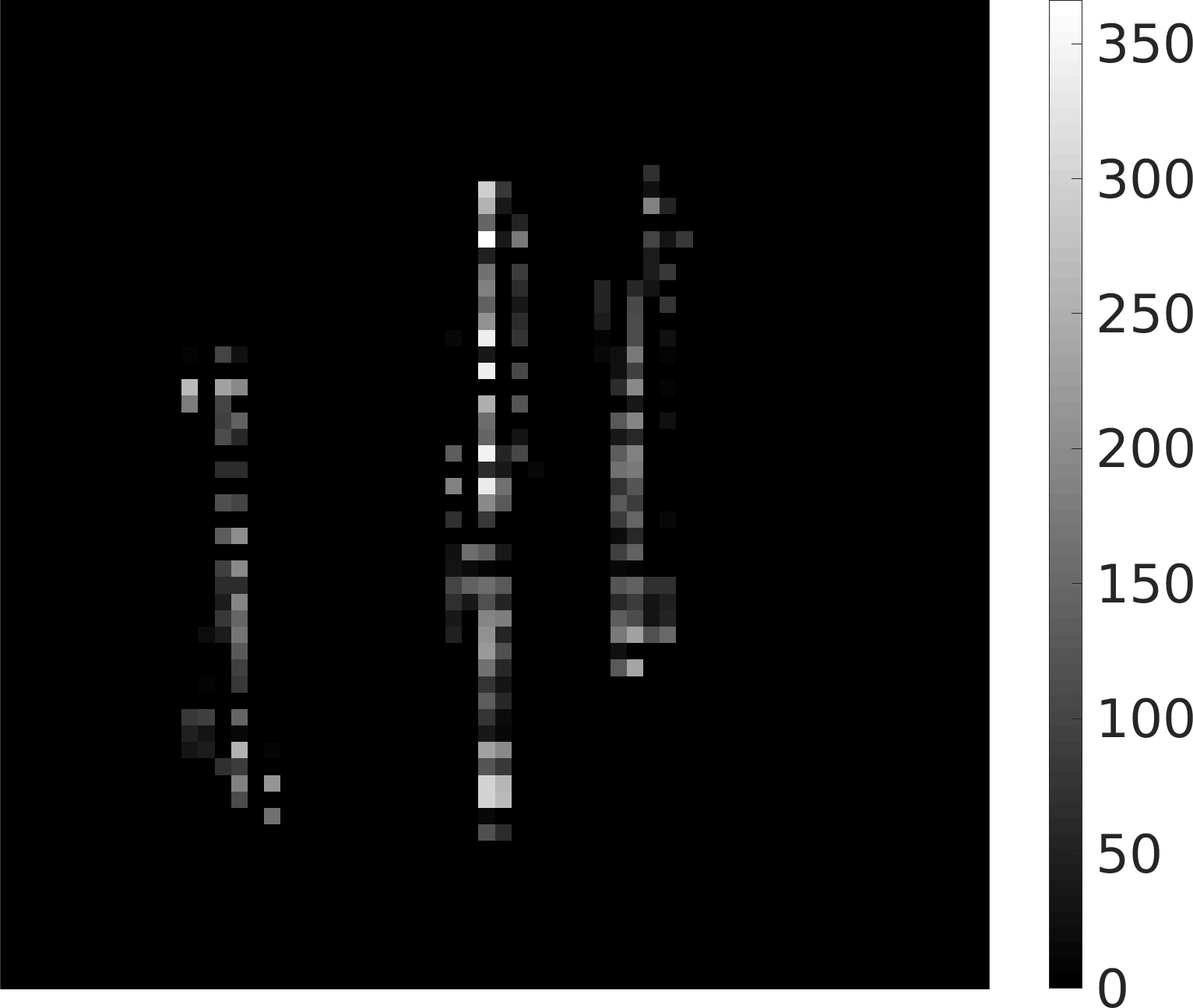}&
\includegraphics[width=0.16\textwidth]{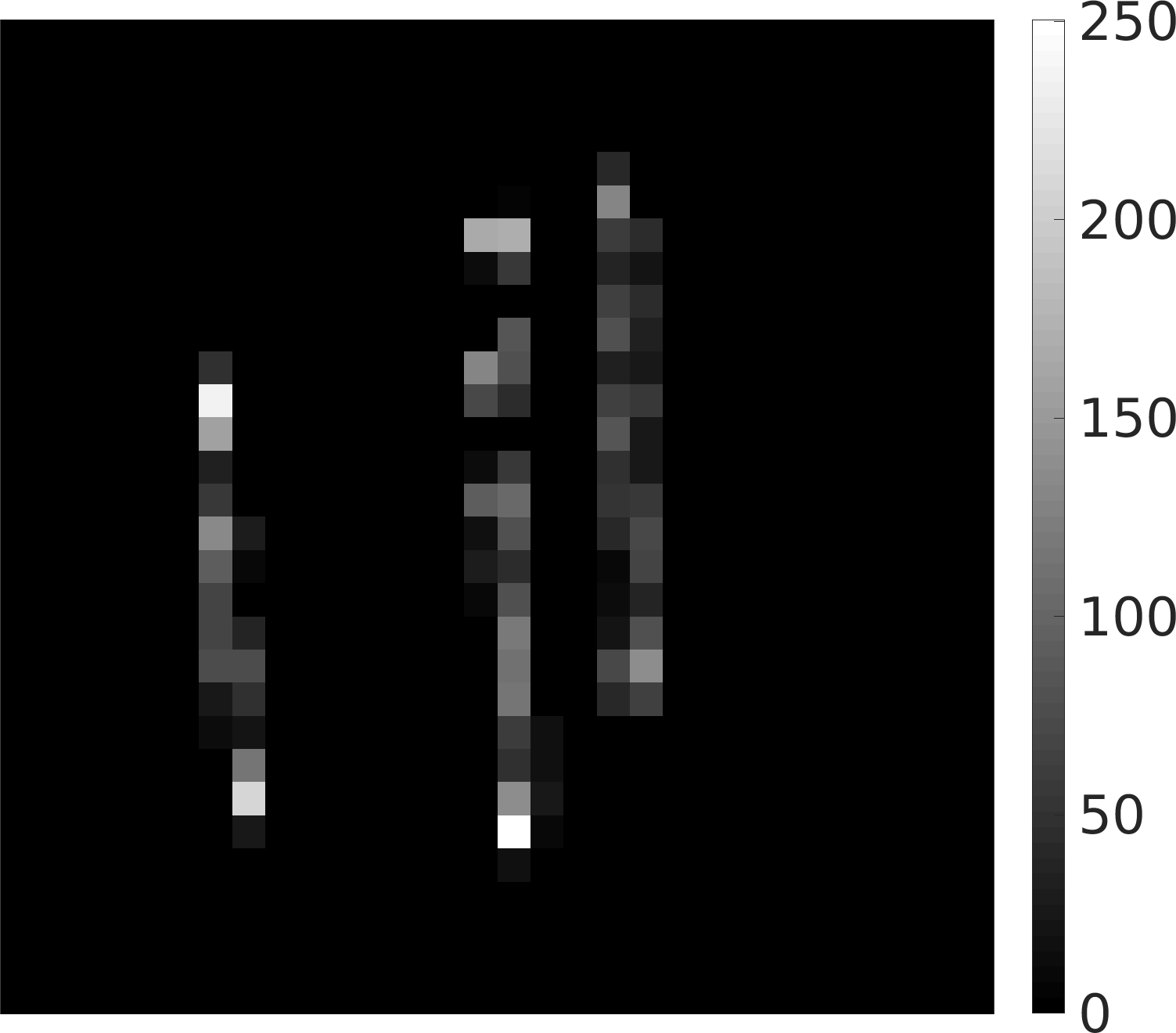}\\

\hline
2 &
\includegraphics[width=0.16\textwidth]{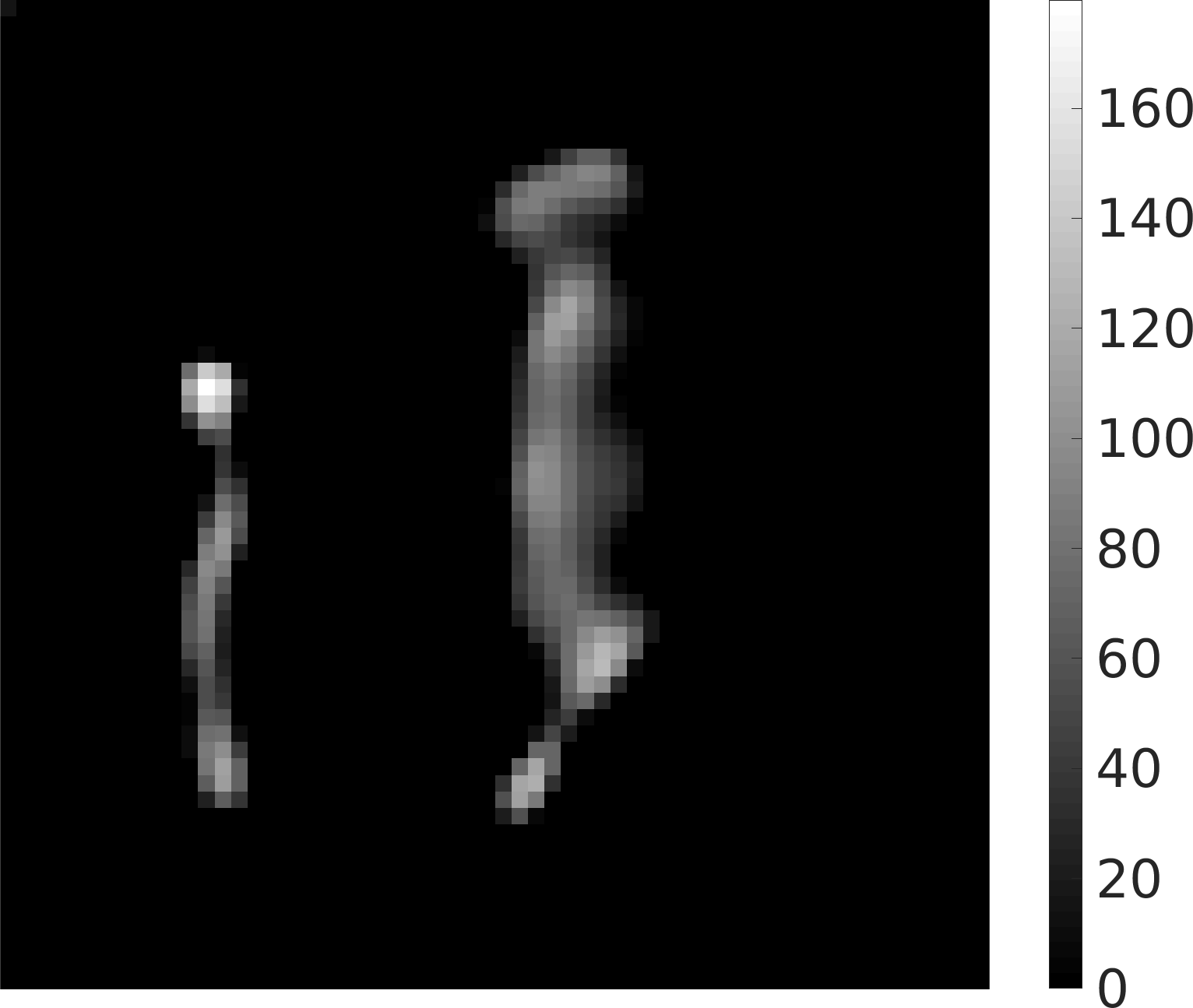}&
\includegraphics[width=0.16\textwidth]{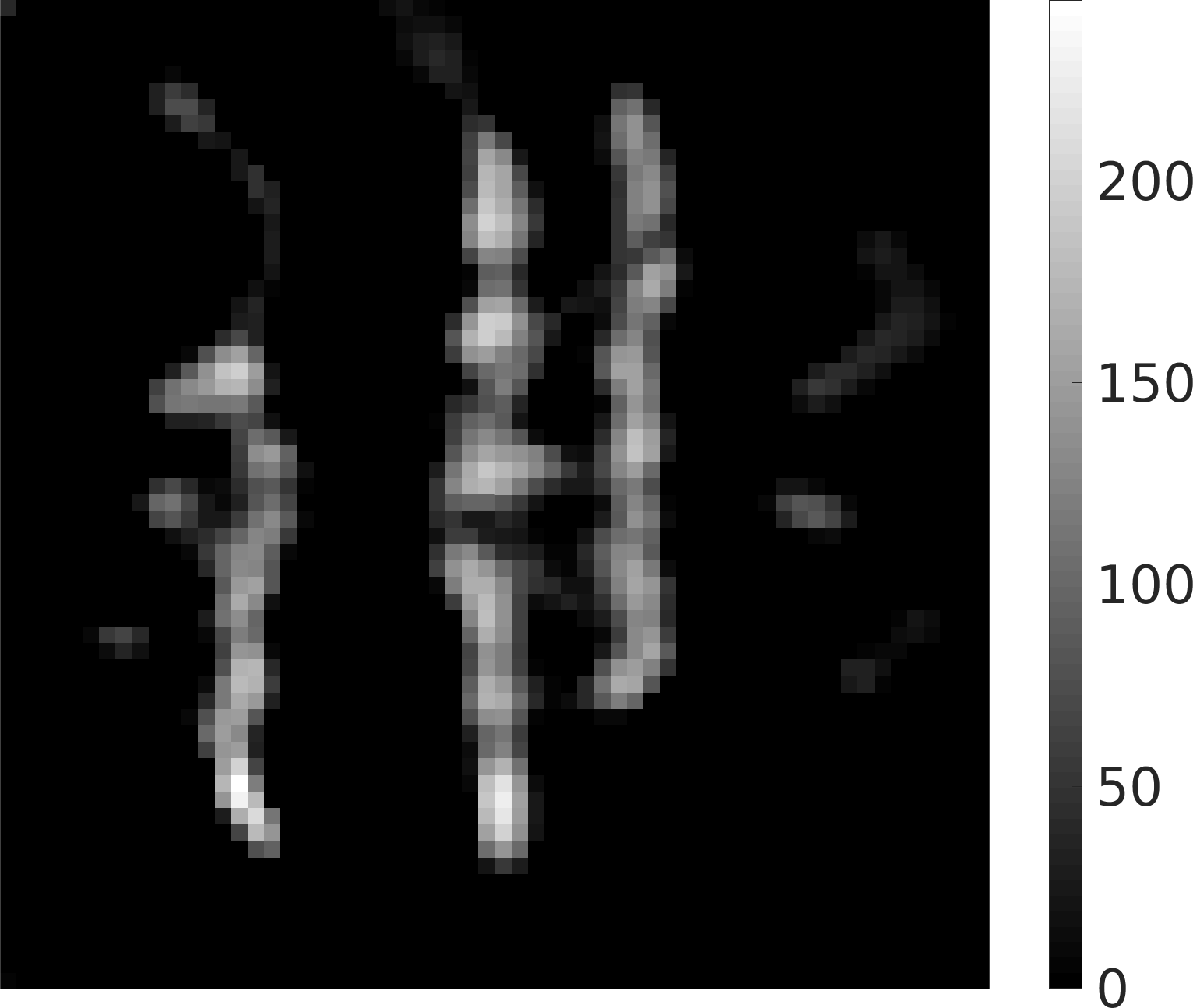}&
\includegraphics[width=0.16\textwidth]{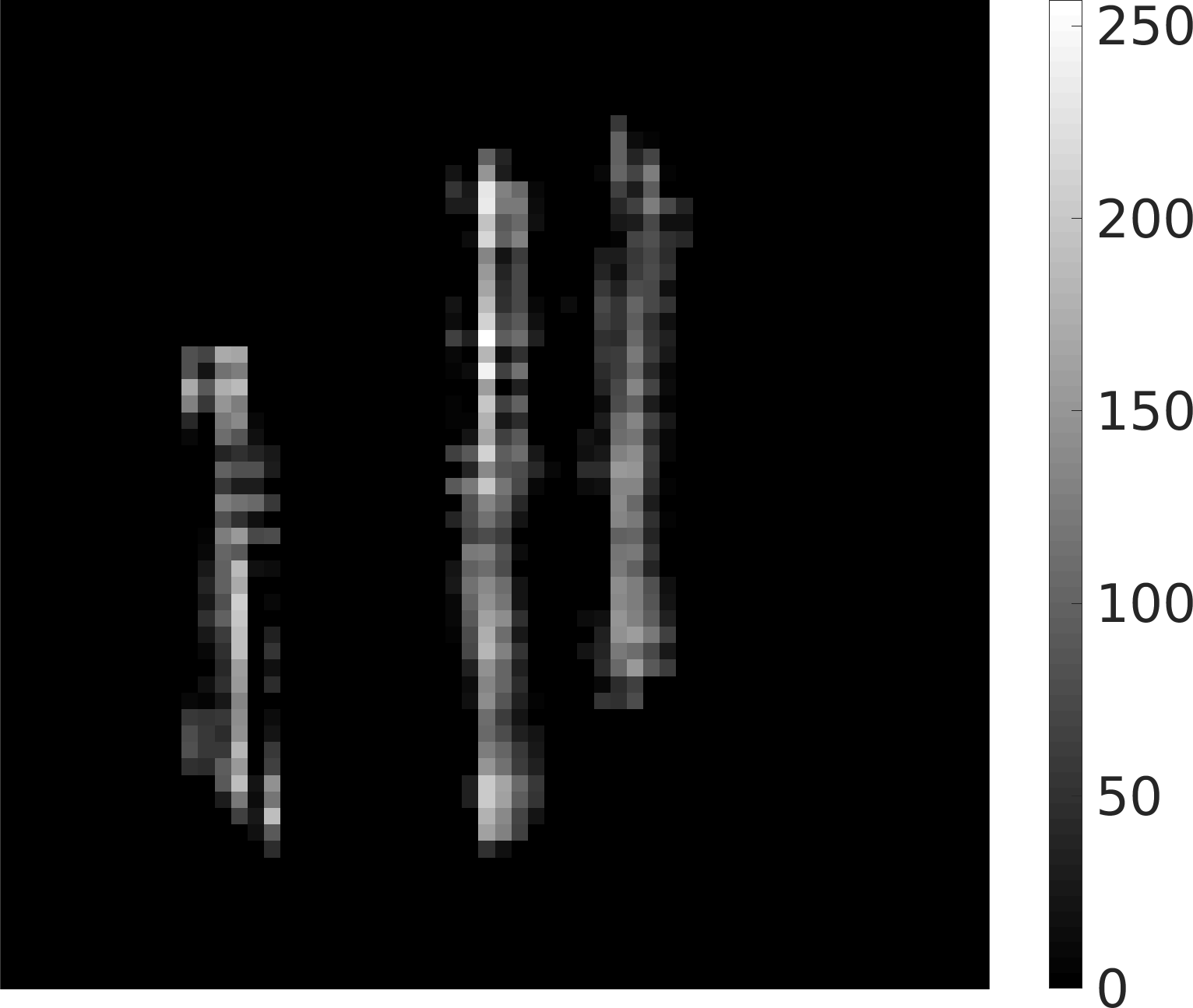}&
\includegraphics[width=0.16\textwidth]{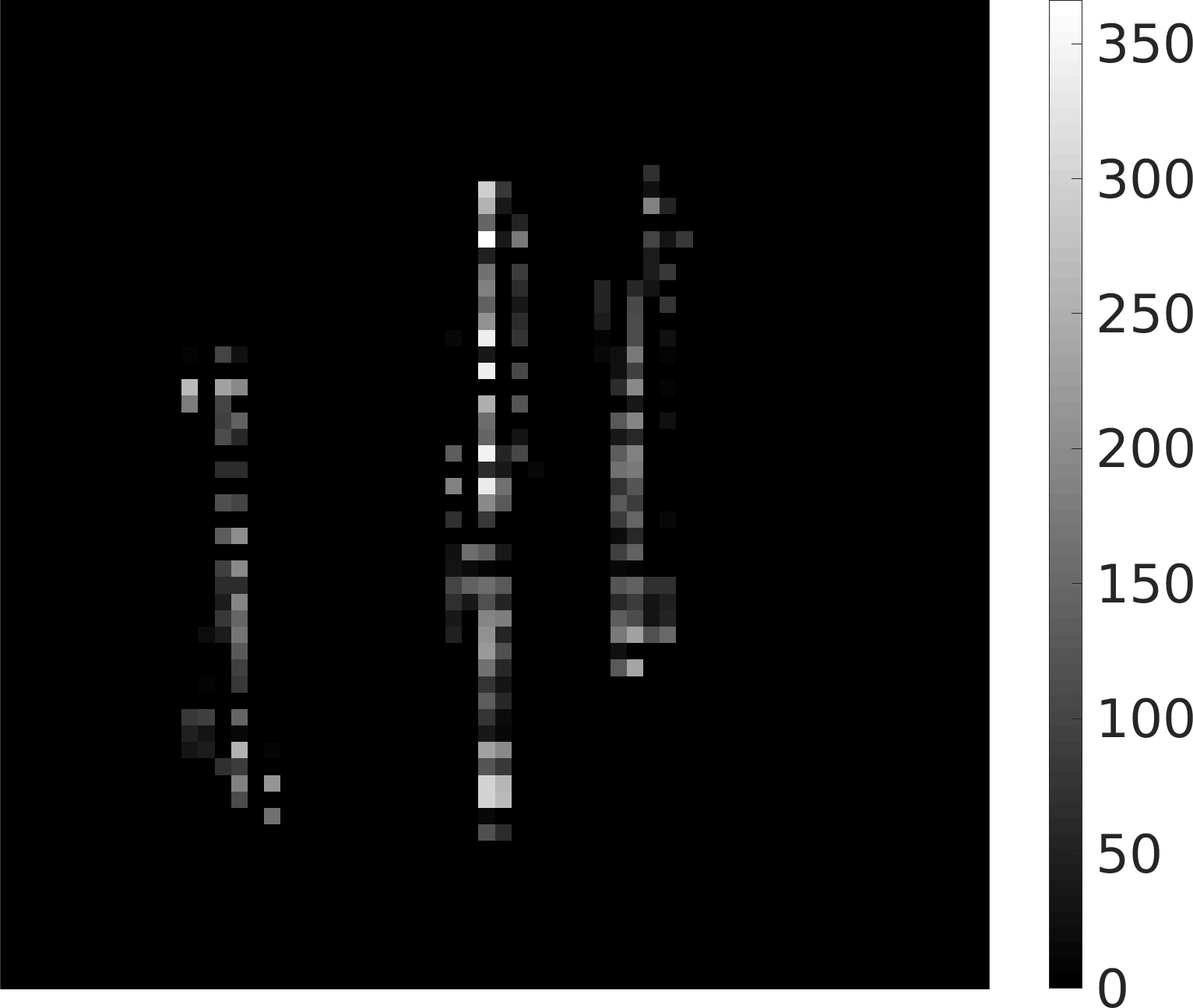}&
\includegraphics[width=0.16\textwidth]{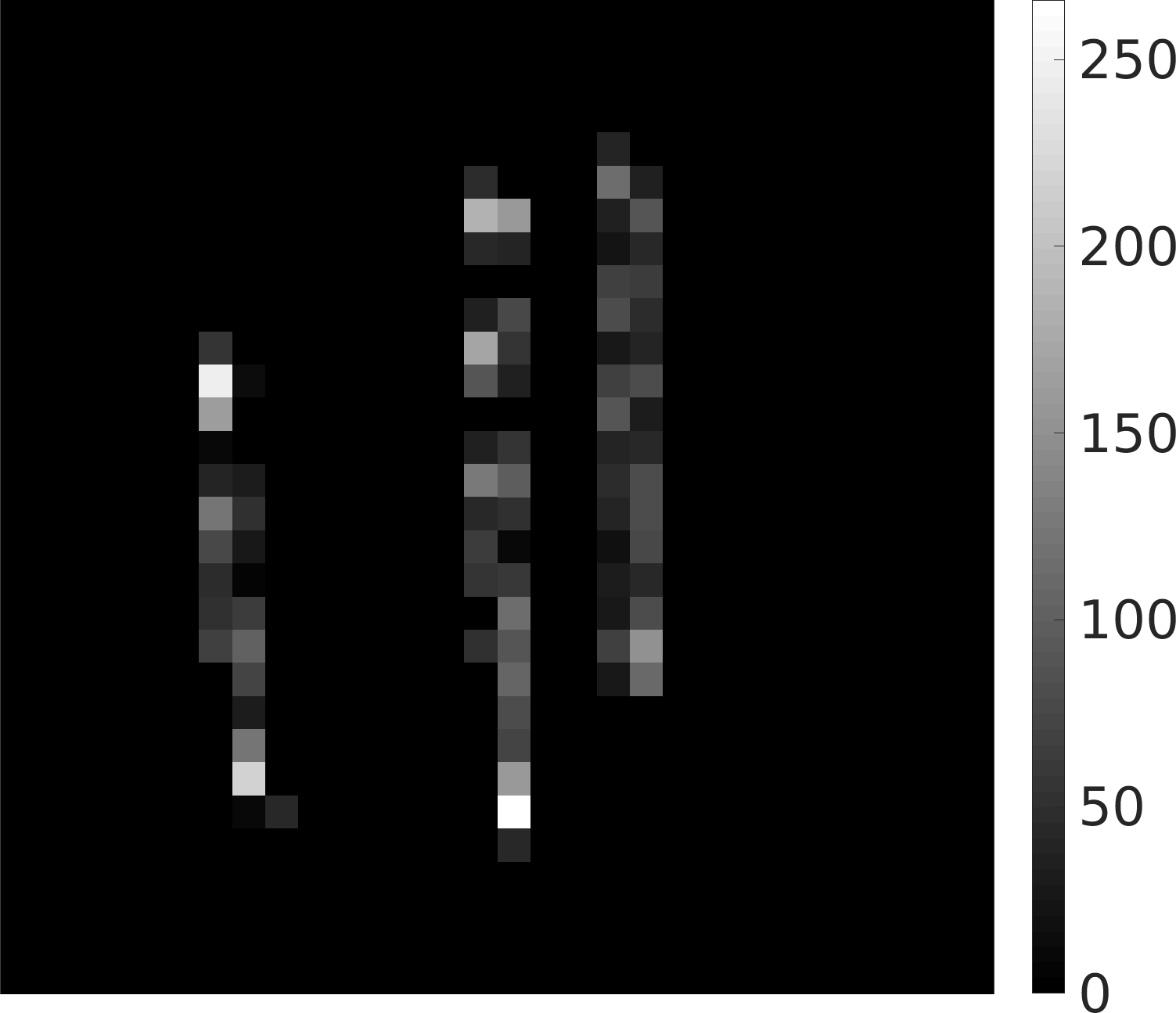}\\

\hline
3 &
\includegraphics[width=0.16\textwidth]{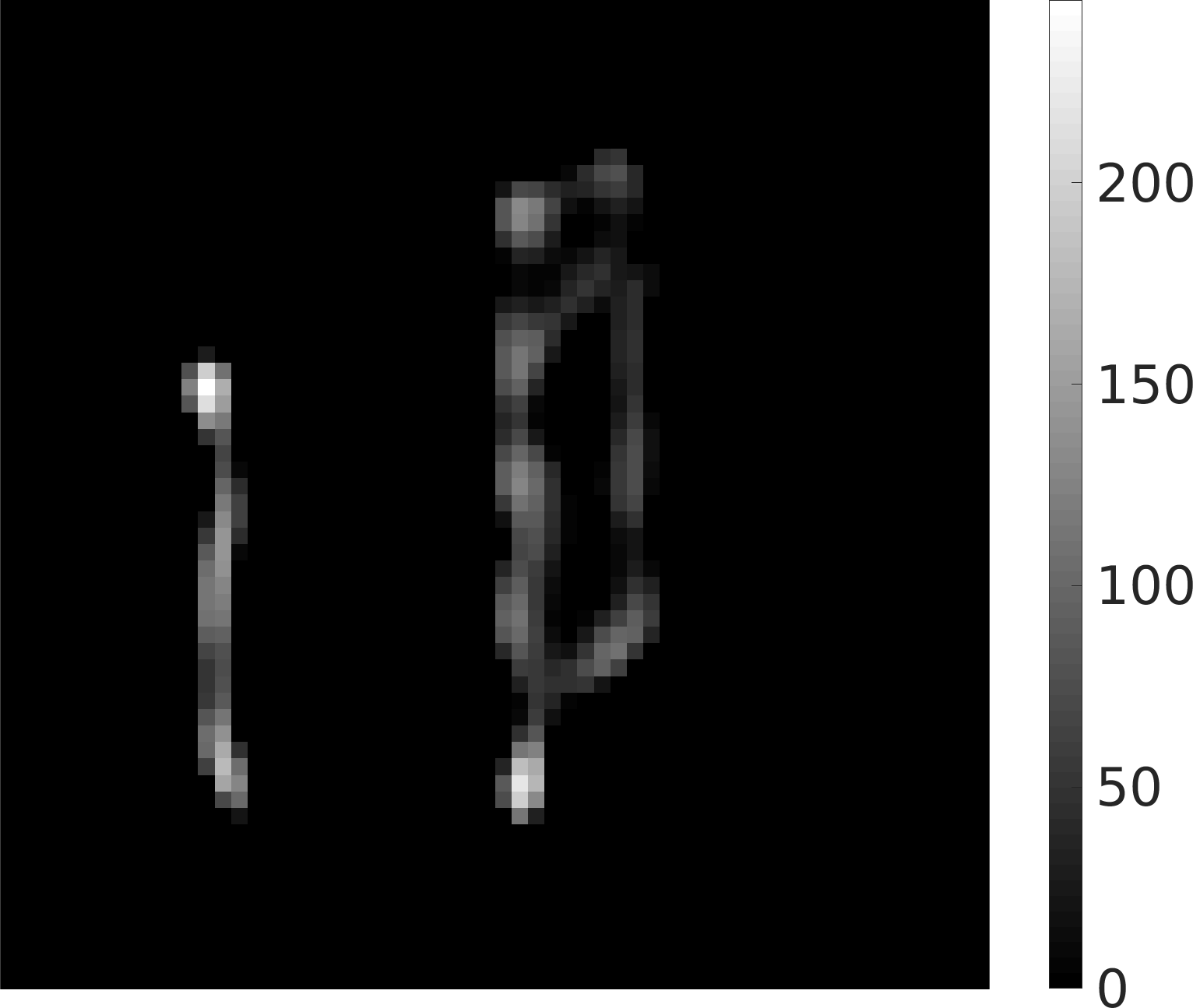}&
\includegraphics[width=0.16\textwidth]{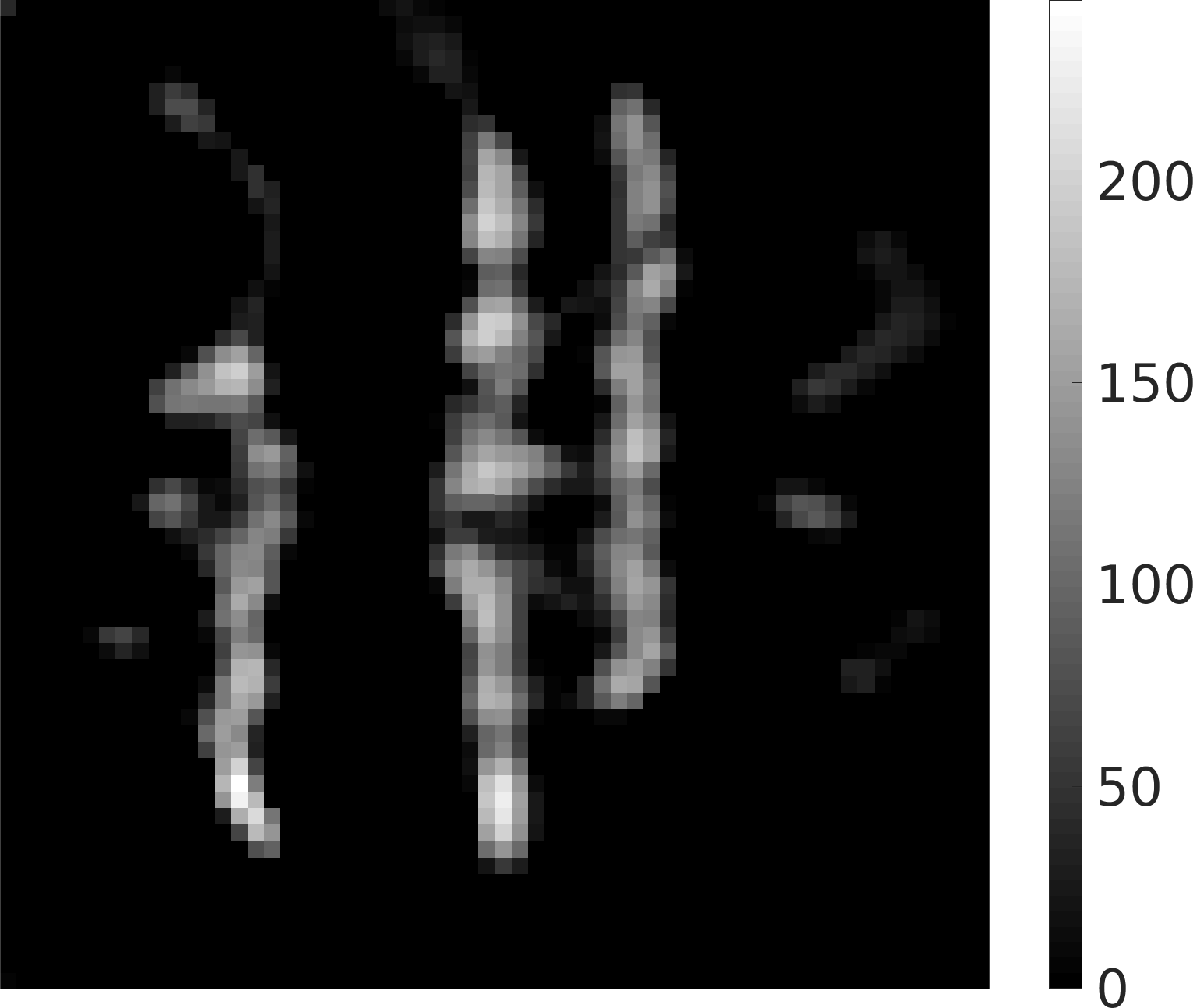}&
\includegraphics[width=0.16\textwidth]{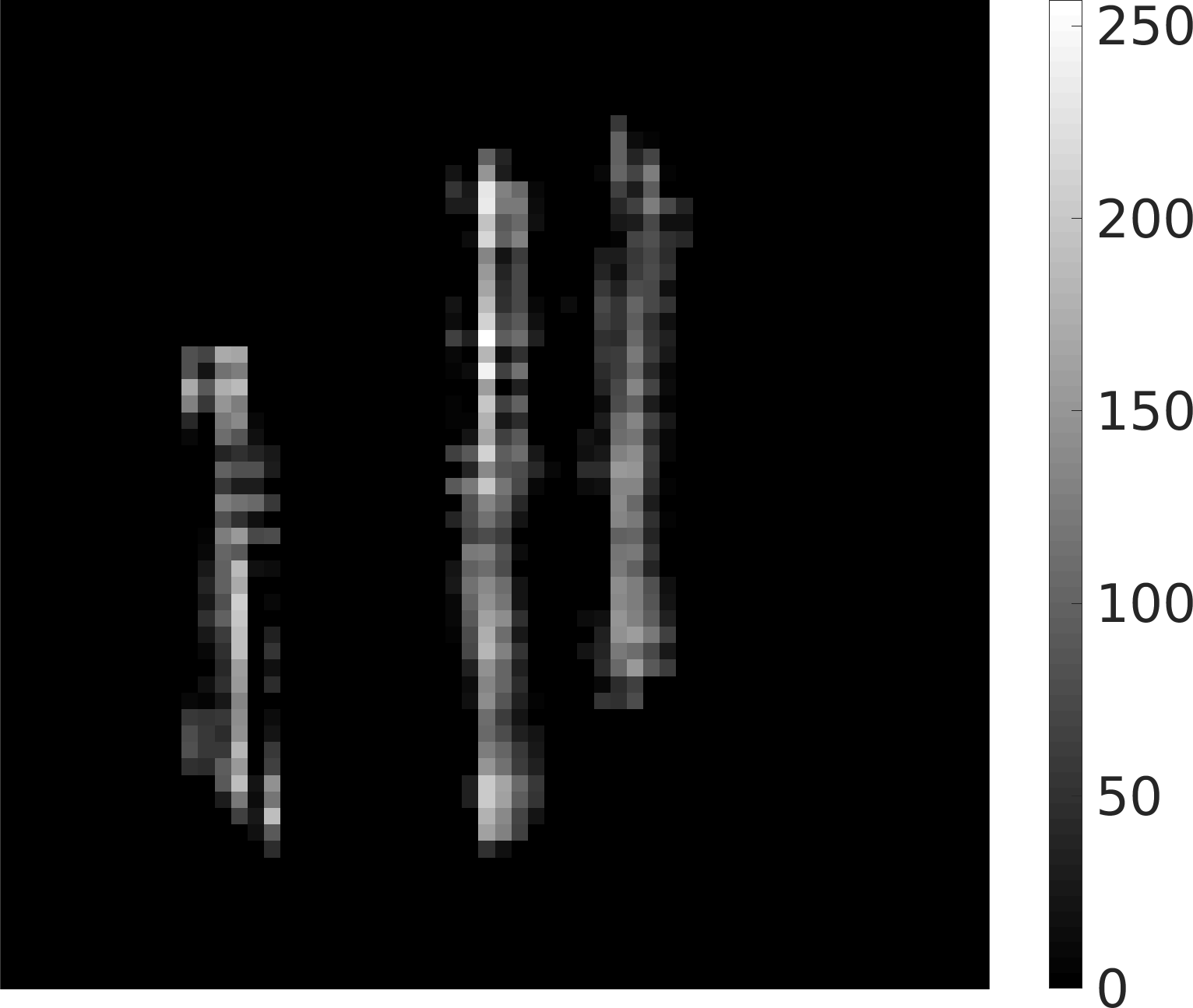}&
\includegraphics[width=0.16\textwidth]{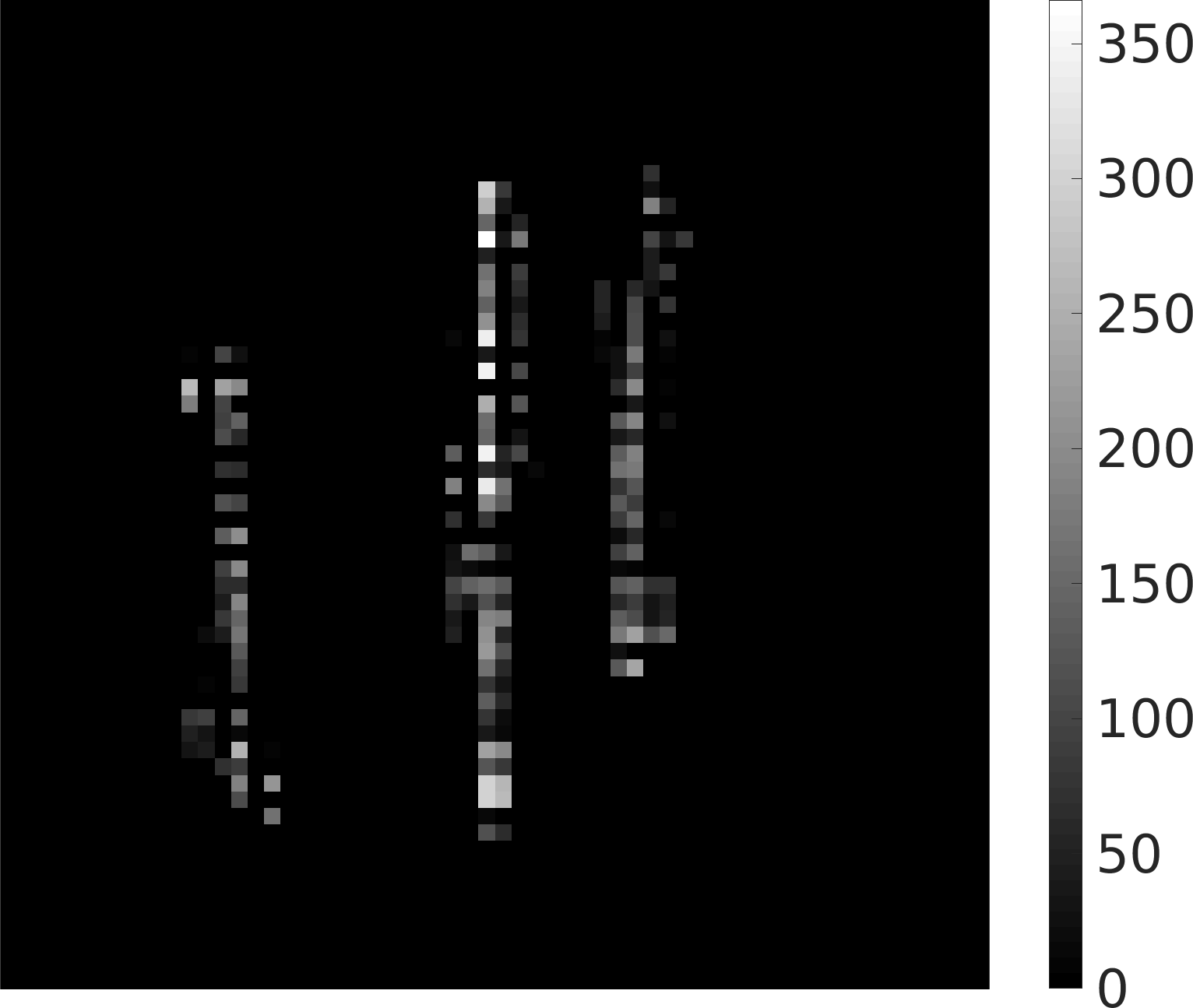}&
\includegraphics[width=0.16\textwidth]{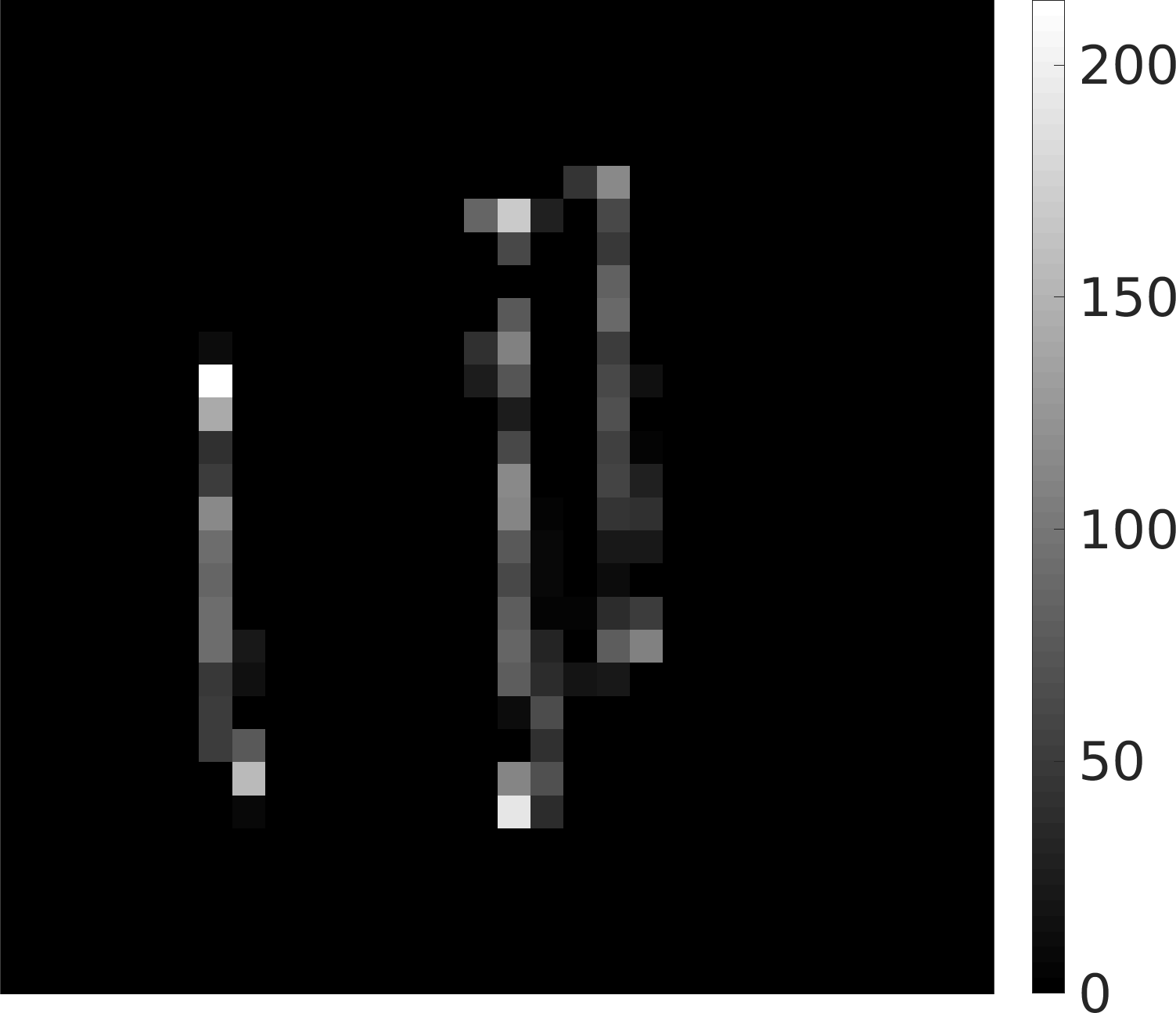}\\

\hline
4 &
\includegraphics[width=0.16\textwidth]{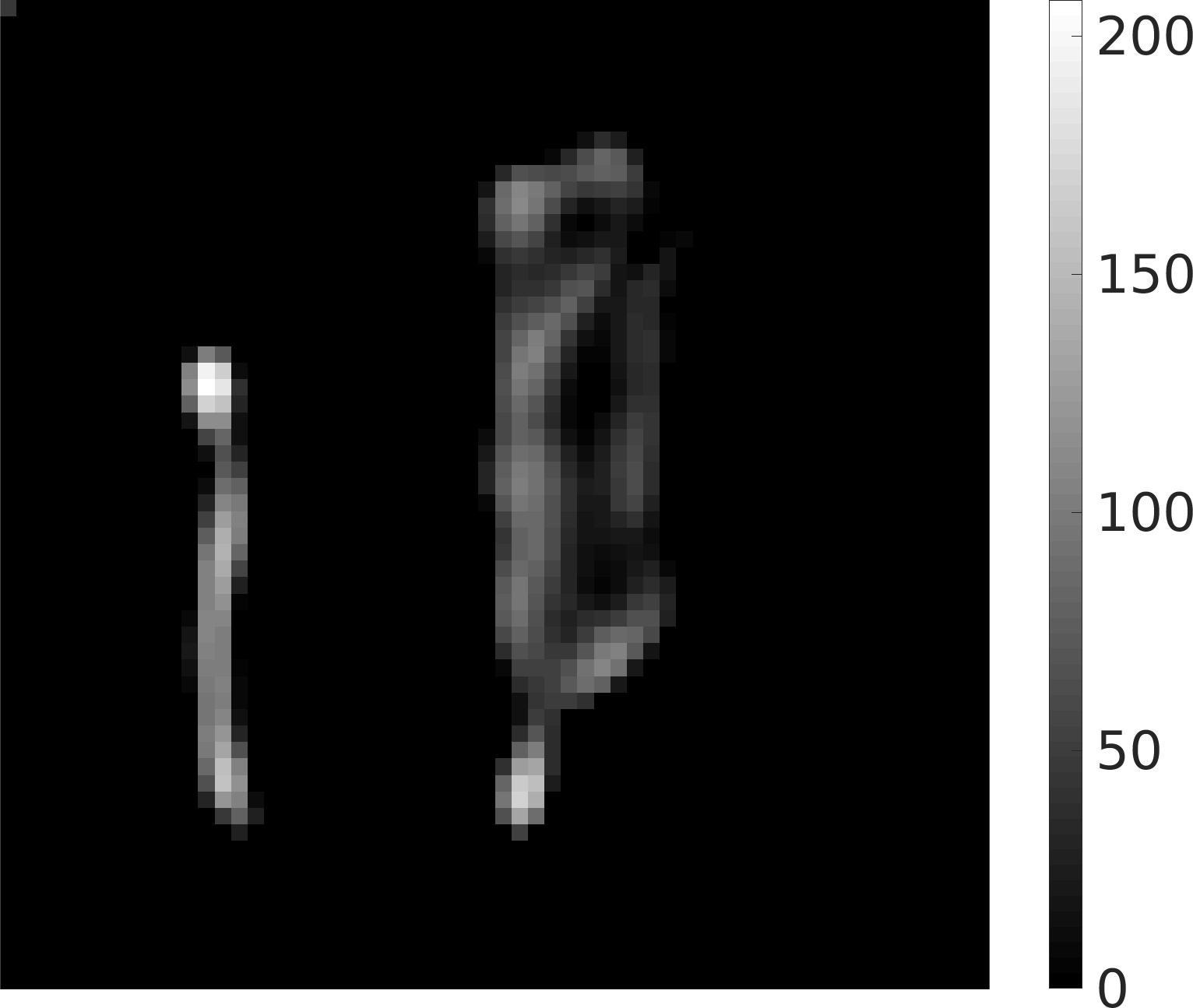}&
\includegraphics[width=0.16\textwidth]{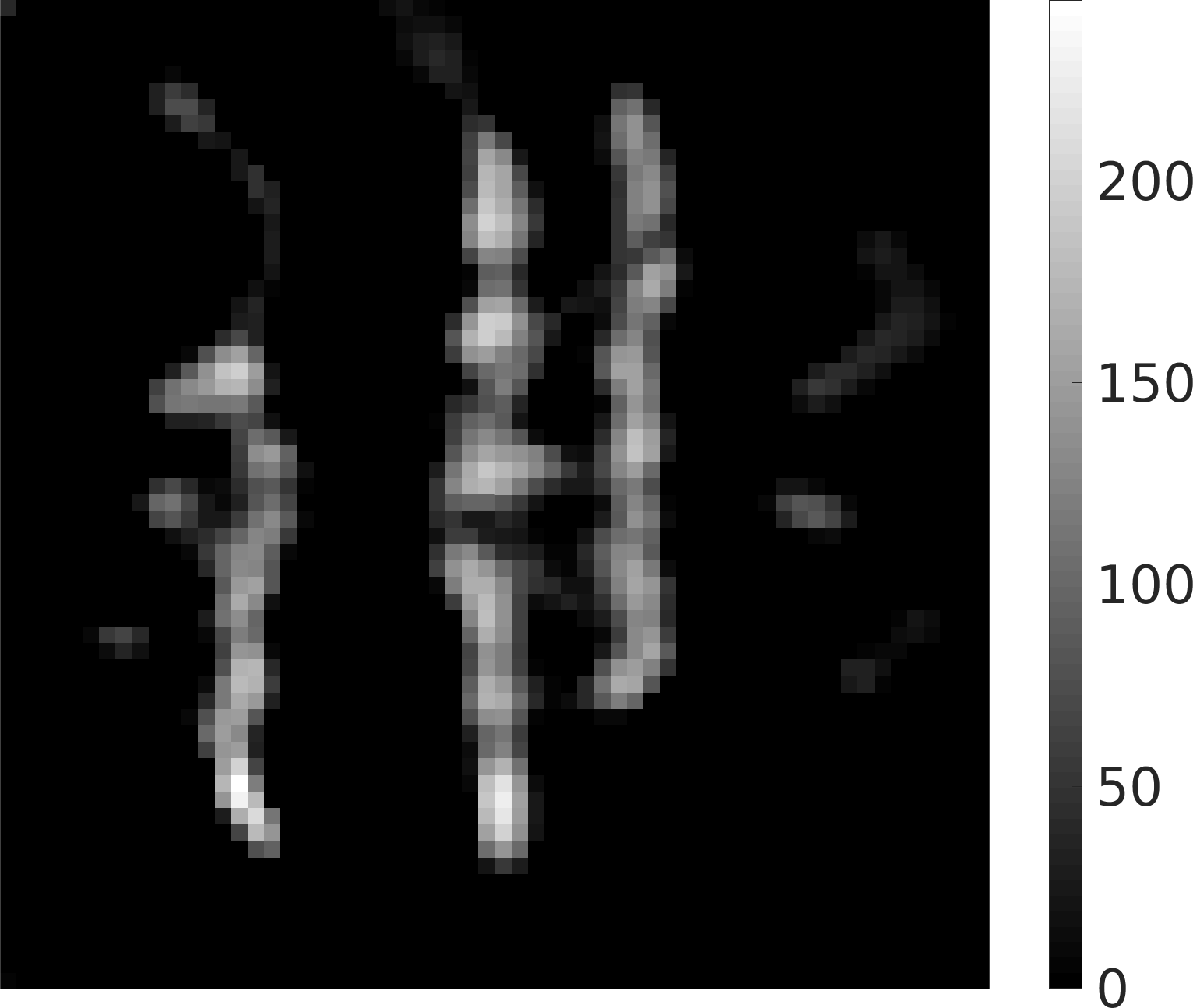}&
\includegraphics[width=0.16\textwidth]{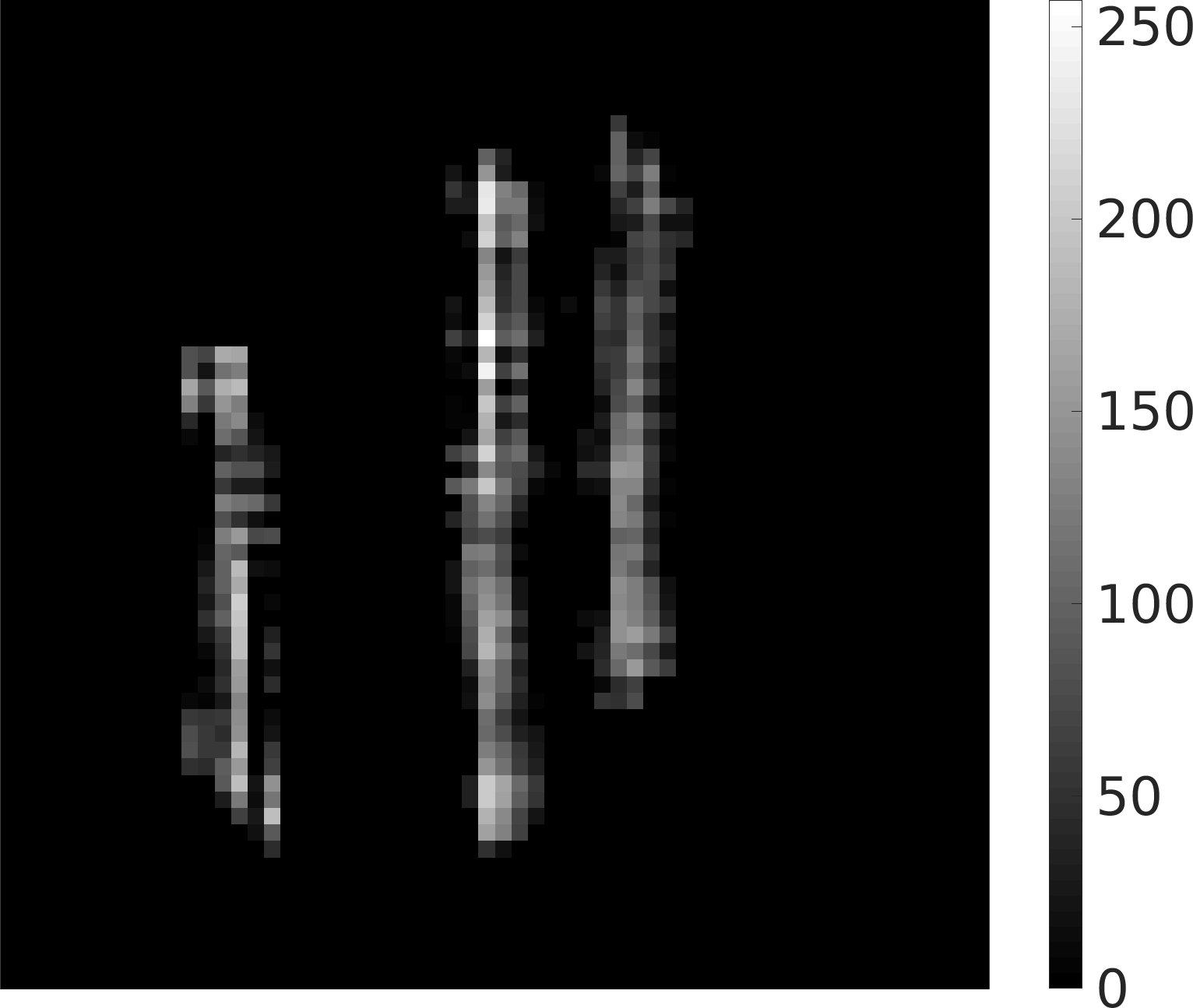}&
\includegraphics[width=0.16\textwidth]{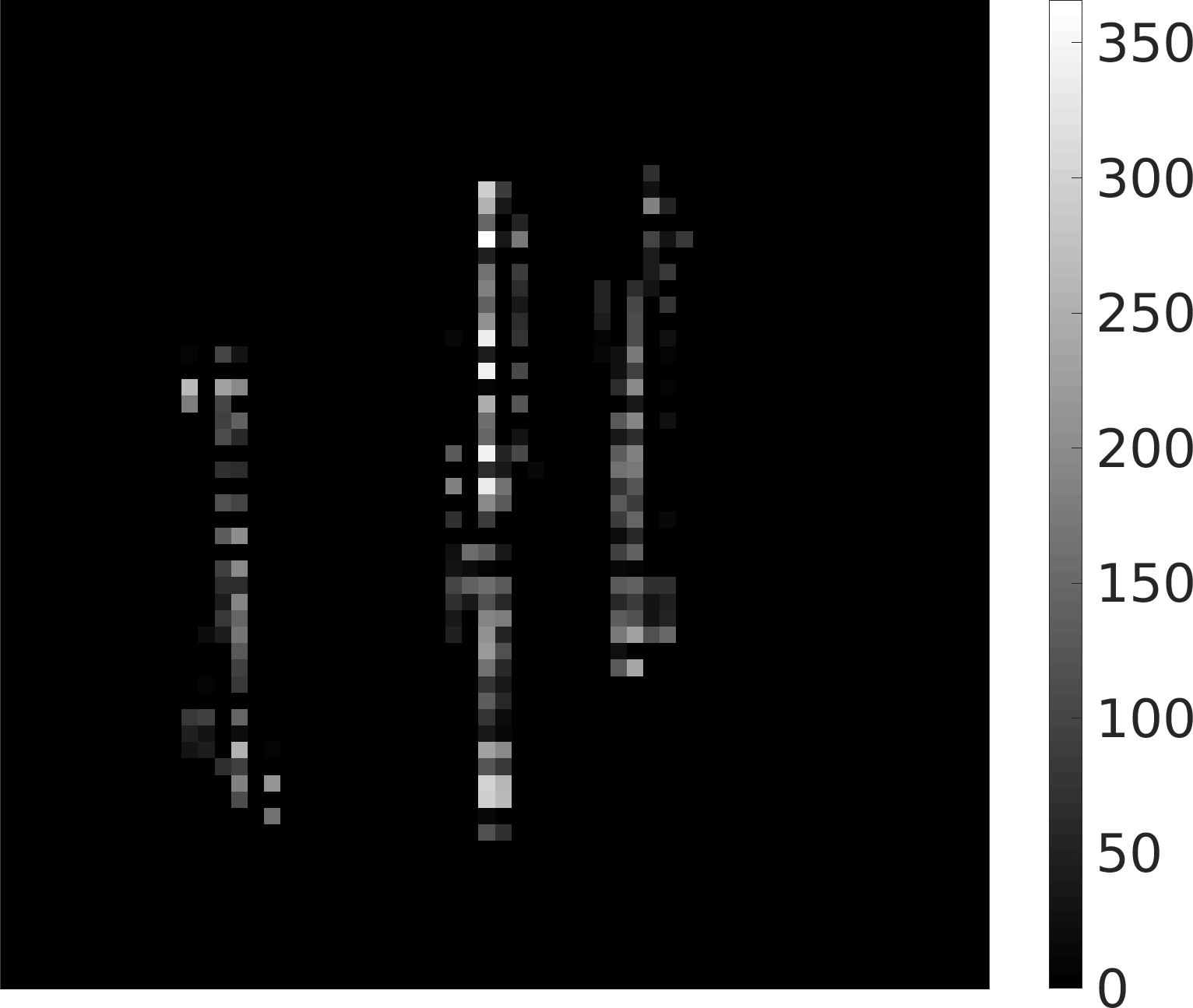}&
\includegraphics[width=0.16\textwidth]{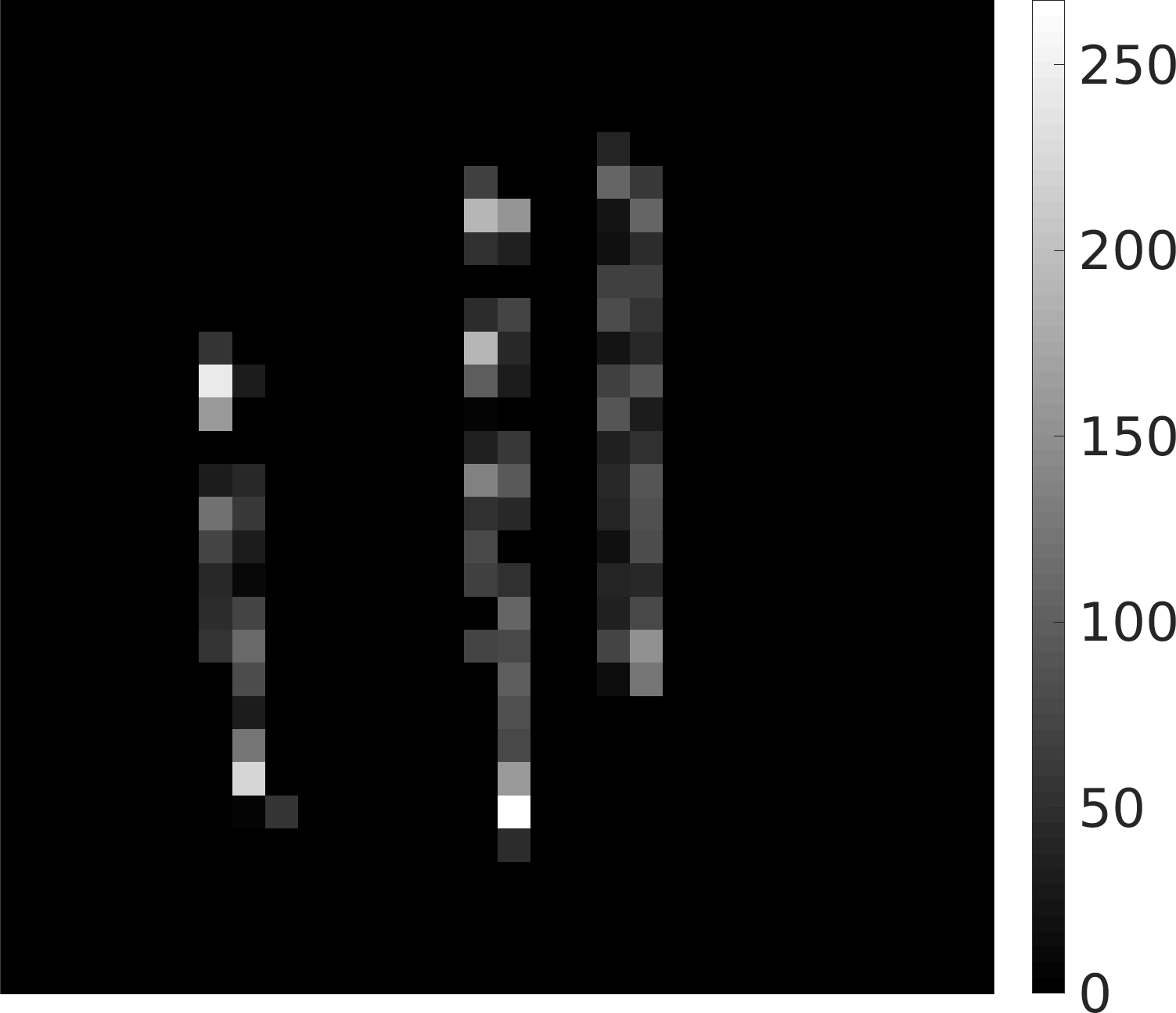}\\

\hline
5 &
\includegraphics[width=0.16\textwidth]{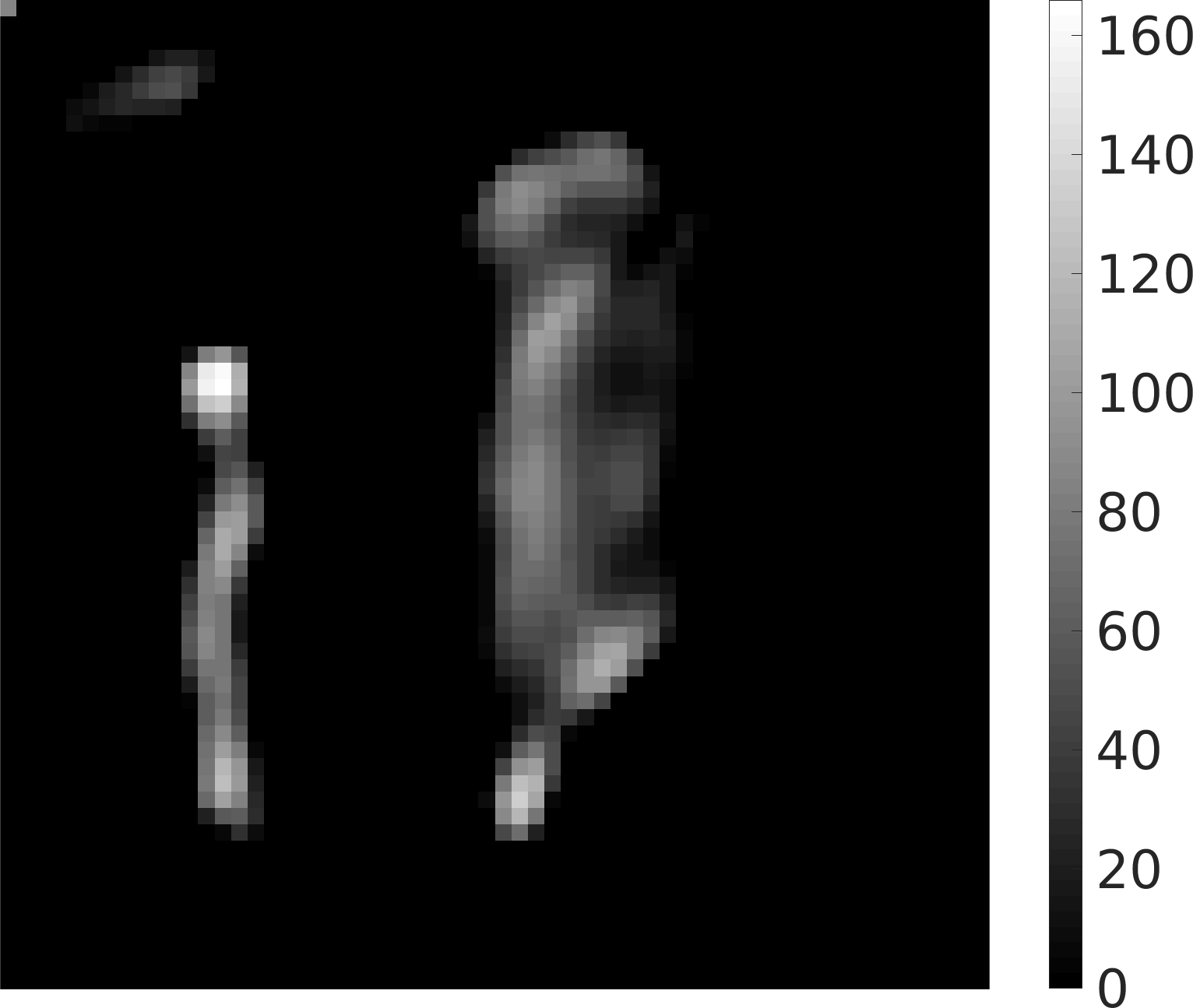}&
\includegraphics[width=0.16\textwidth]{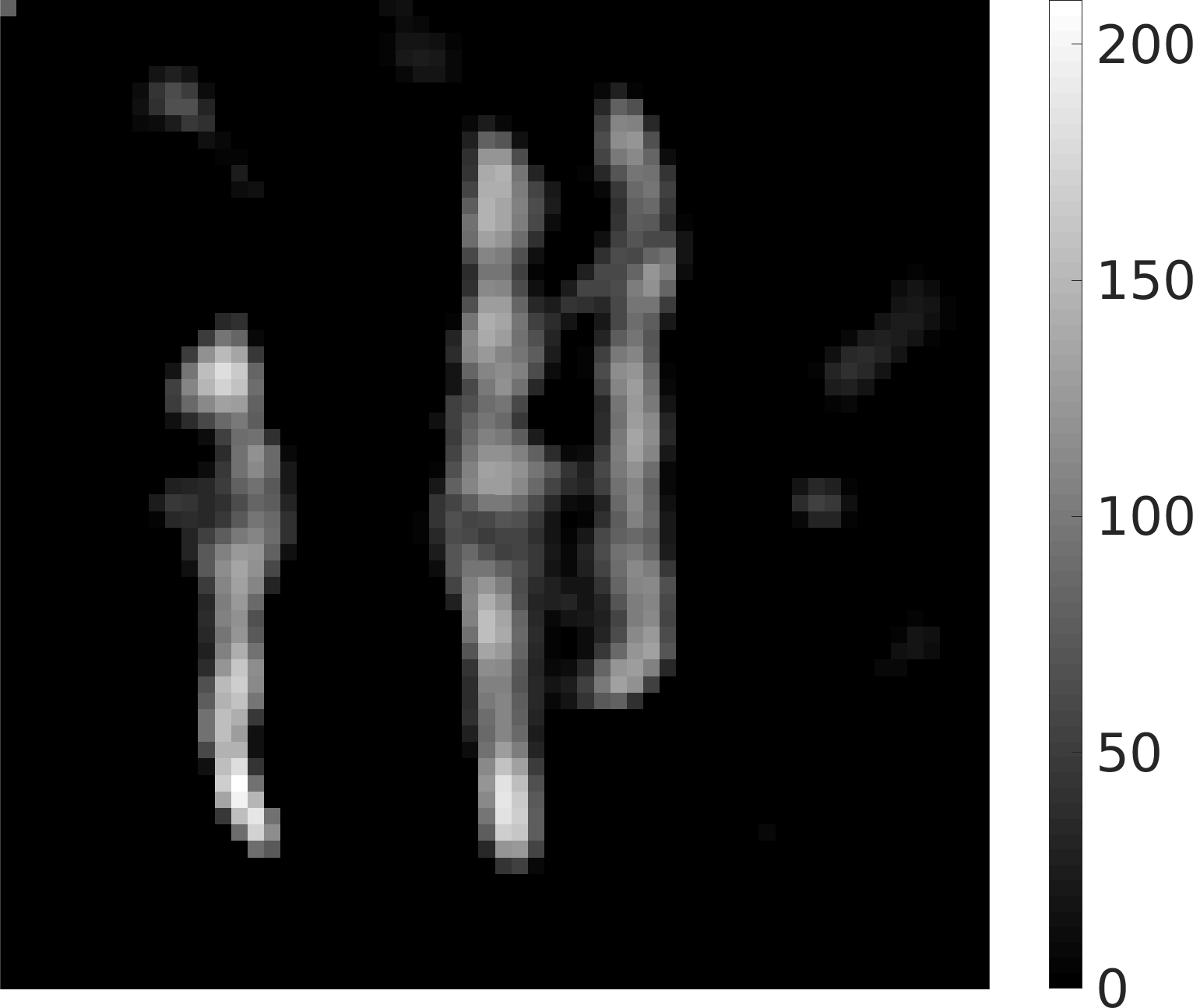}&
\includegraphics[width=0.16\textwidth]{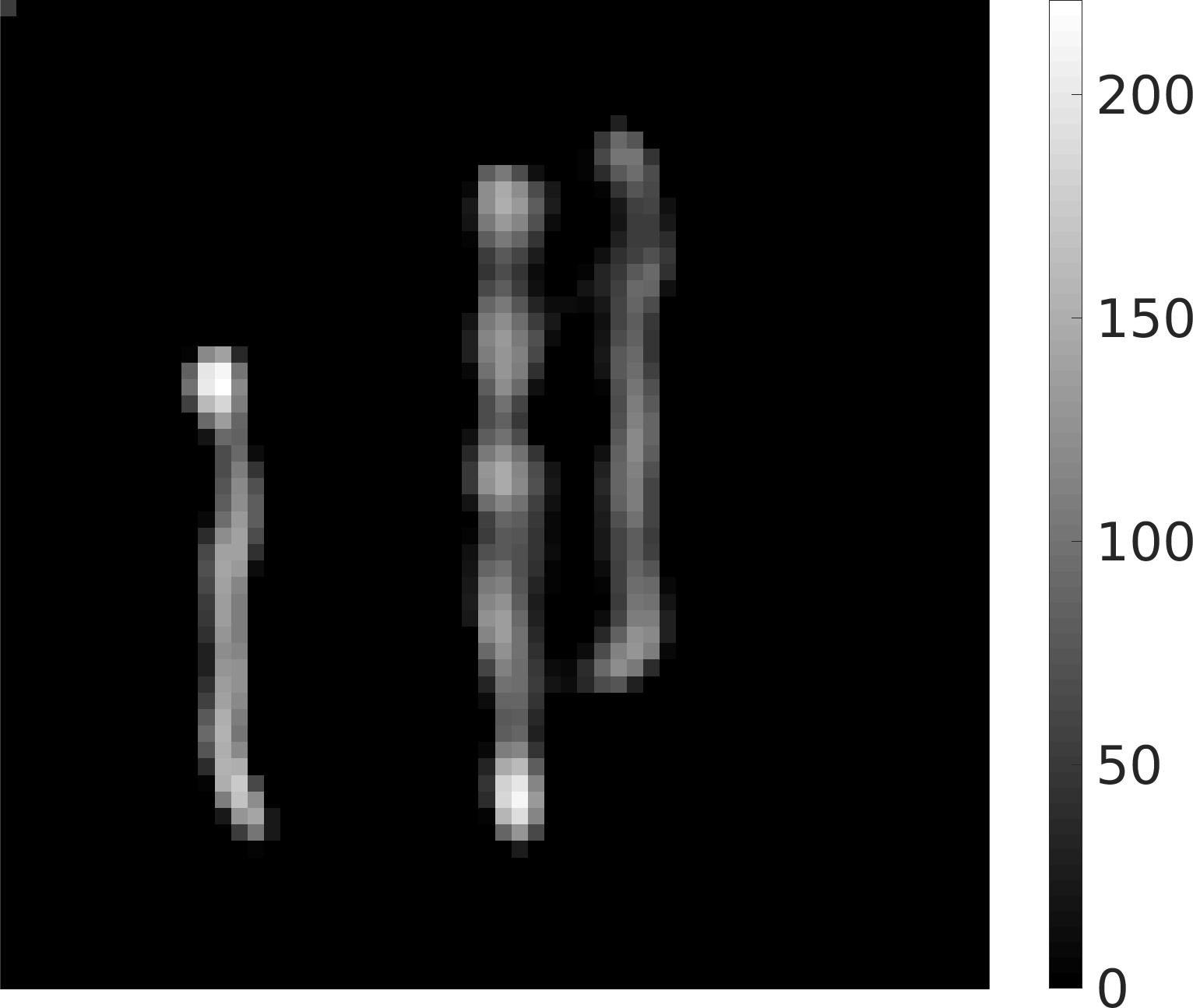}&
\includegraphics[width=0.16\textwidth]{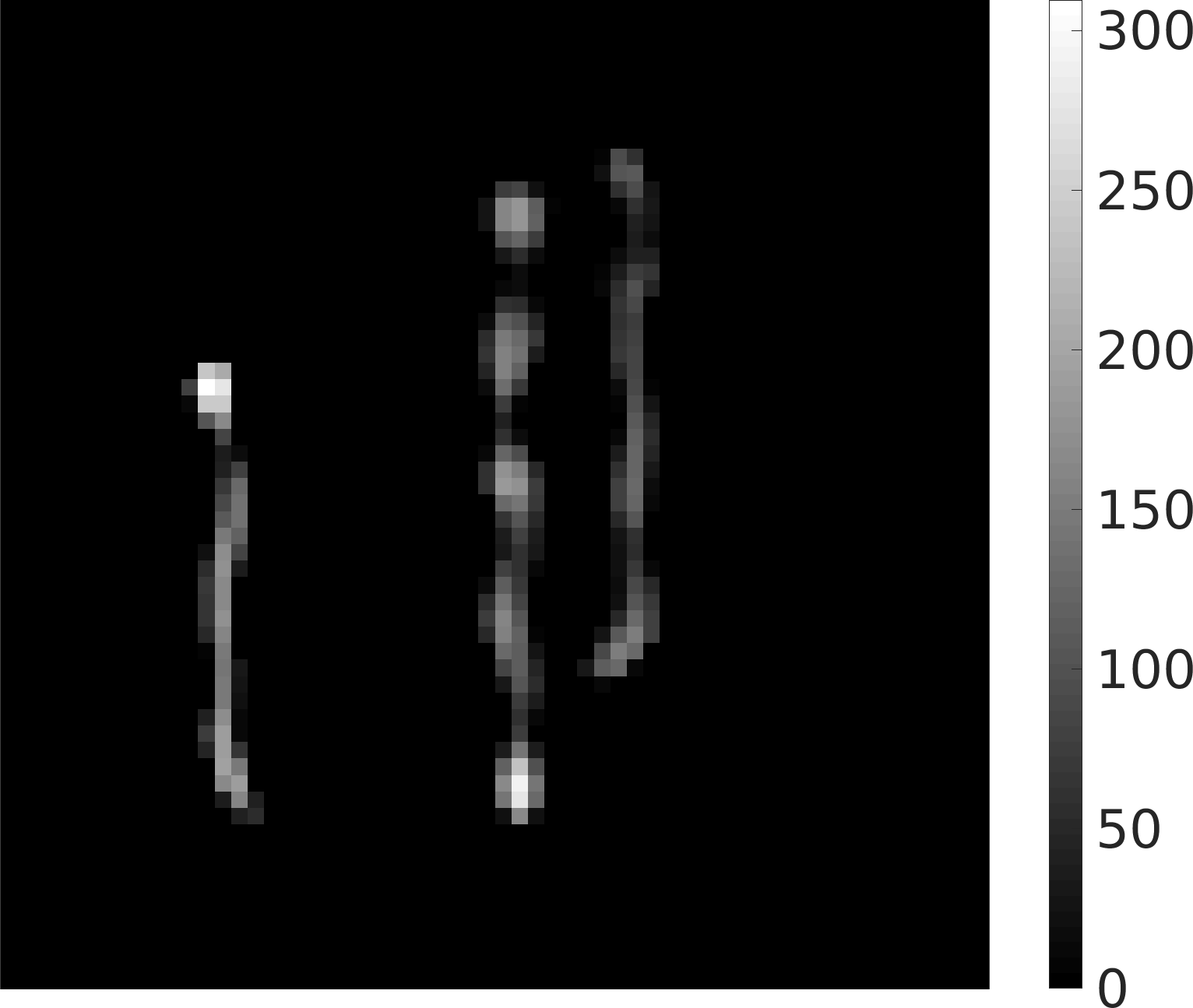}&
\includegraphics[width=0.16\textwidth]{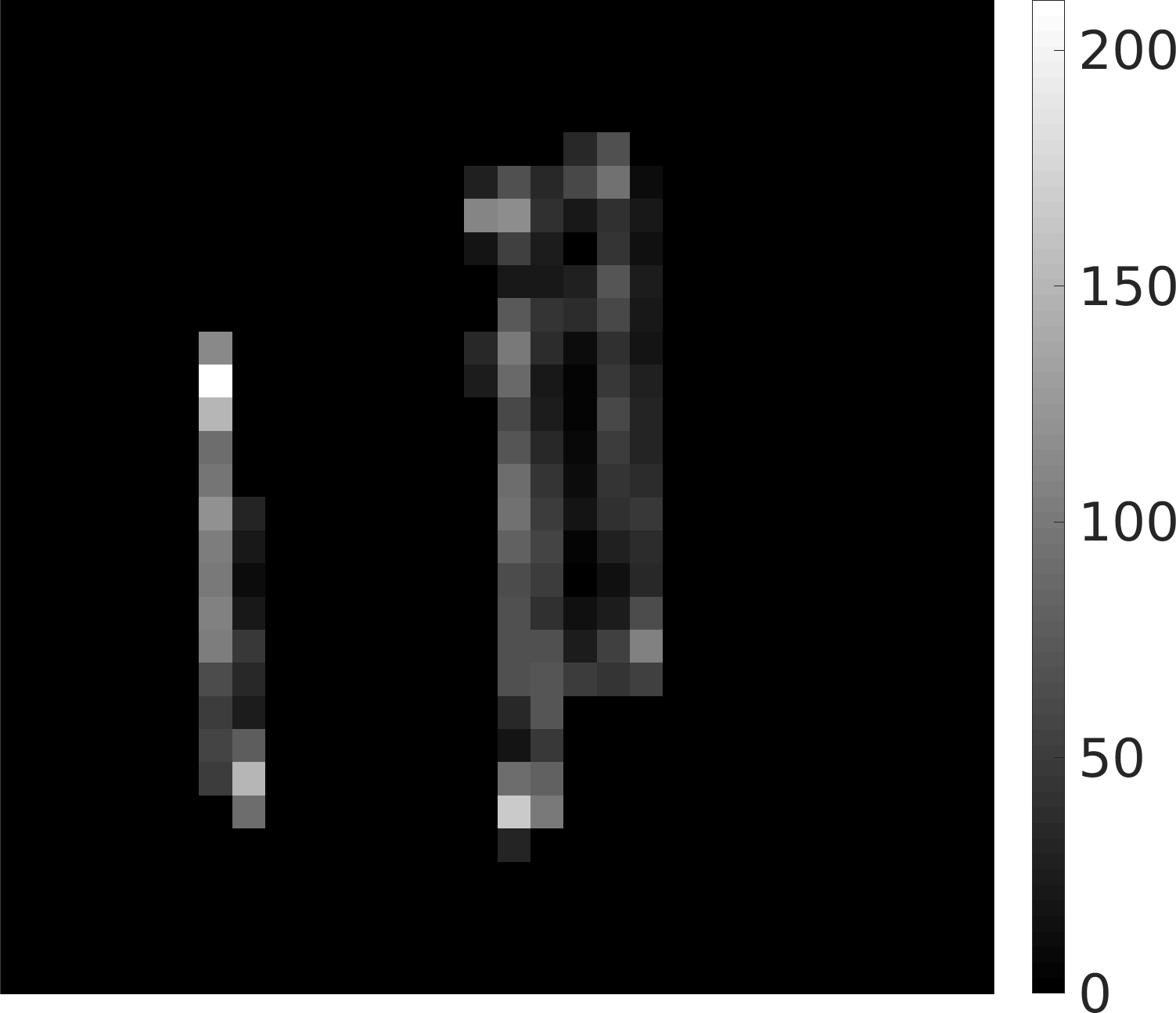}\\

\hline
6 &
\includegraphics[width=0.16\textwidth]{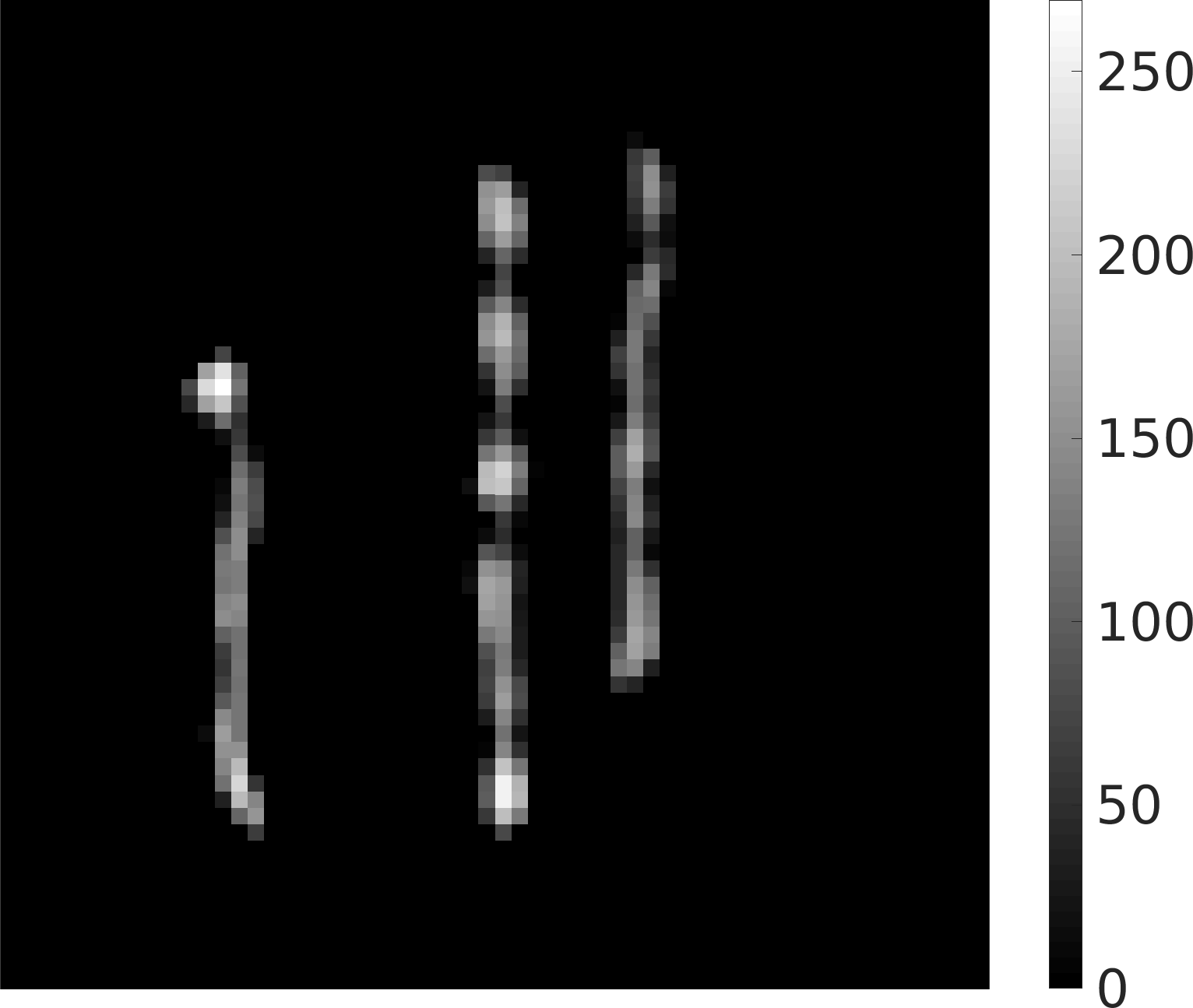}&
\includegraphics[width=0.16\textwidth]{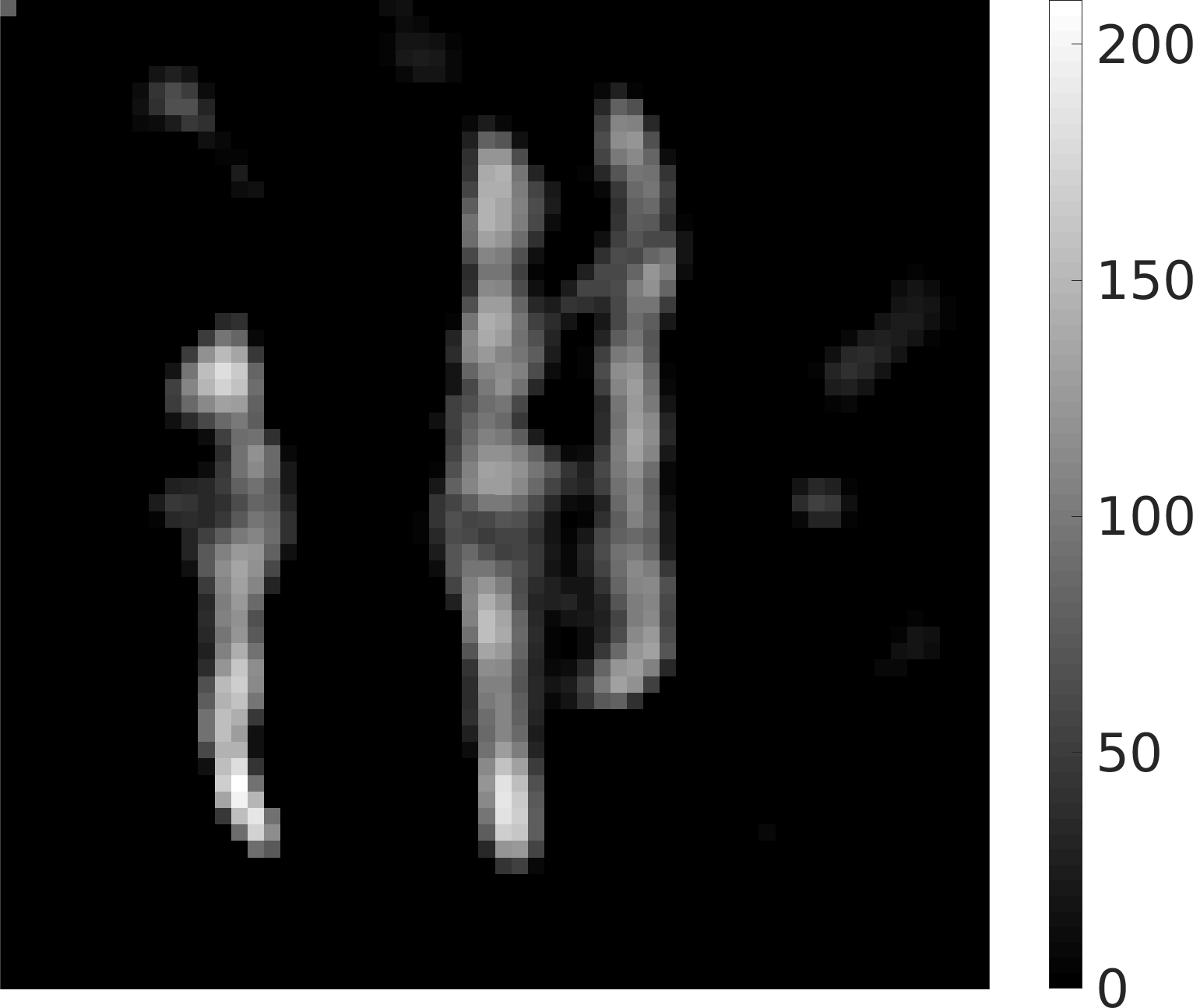}&
\includegraphics[width=0.16\textwidth]{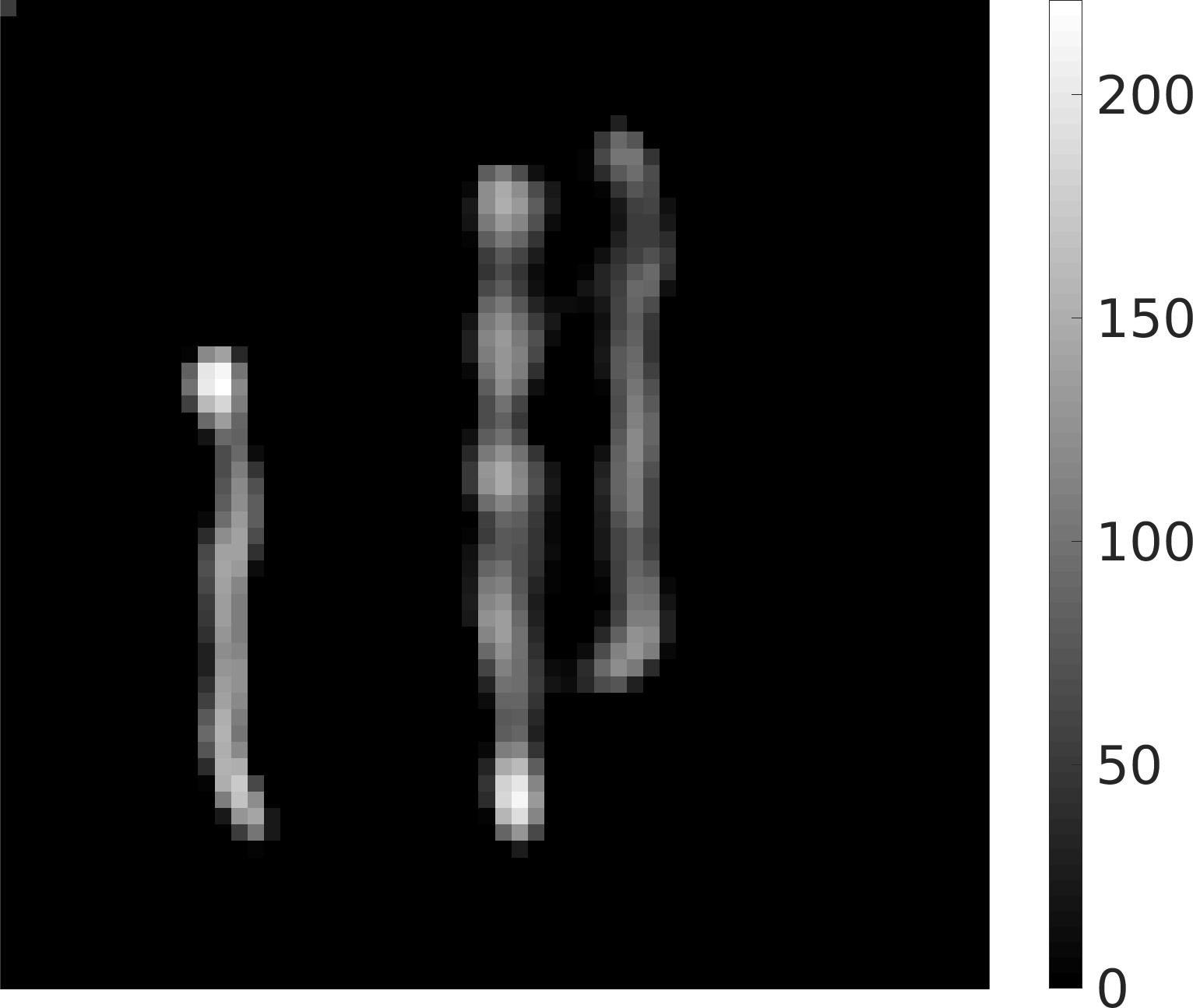}&
\includegraphics[width=0.16\textwidth]{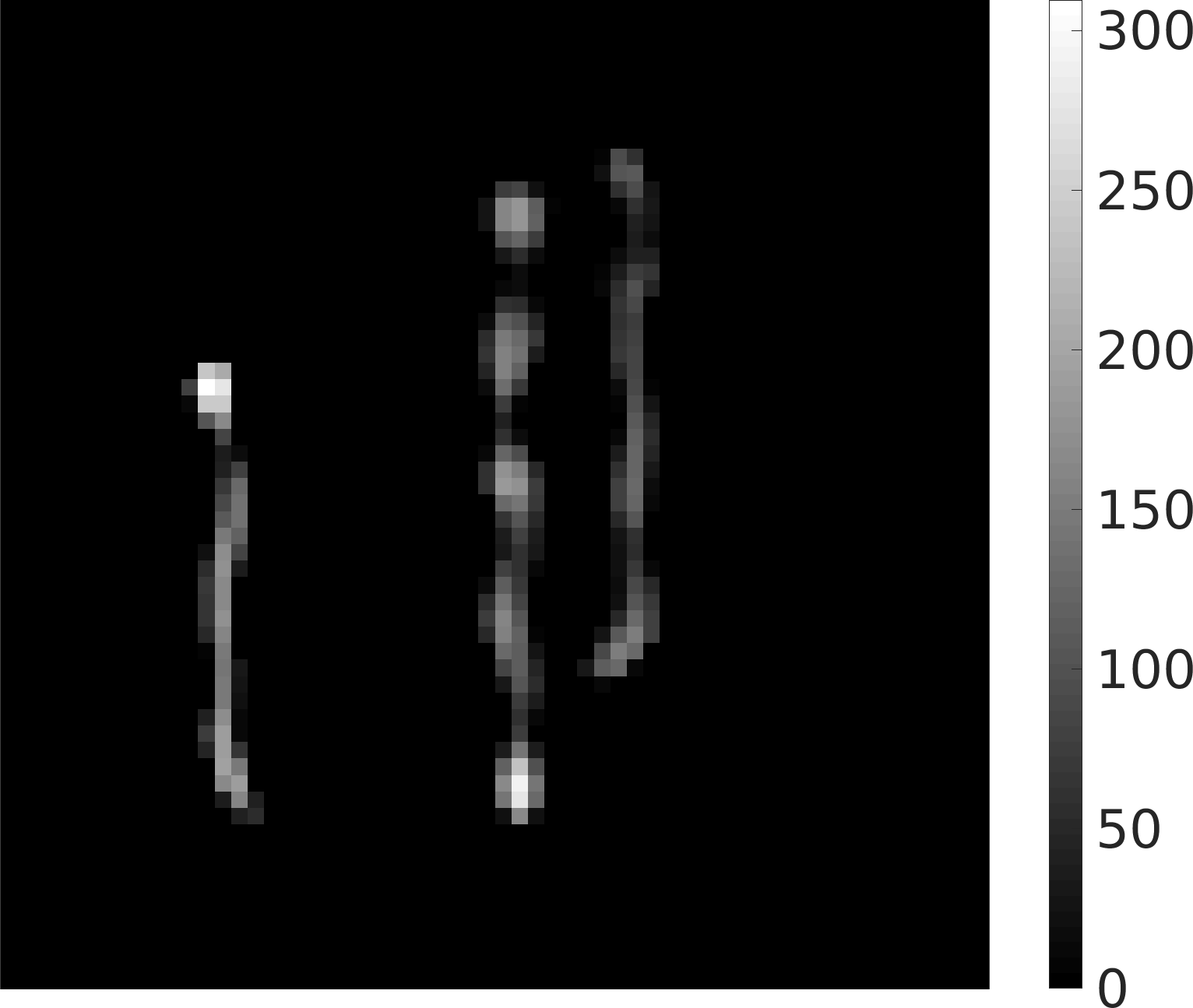}&
\includegraphics[width=0.16\textwidth]{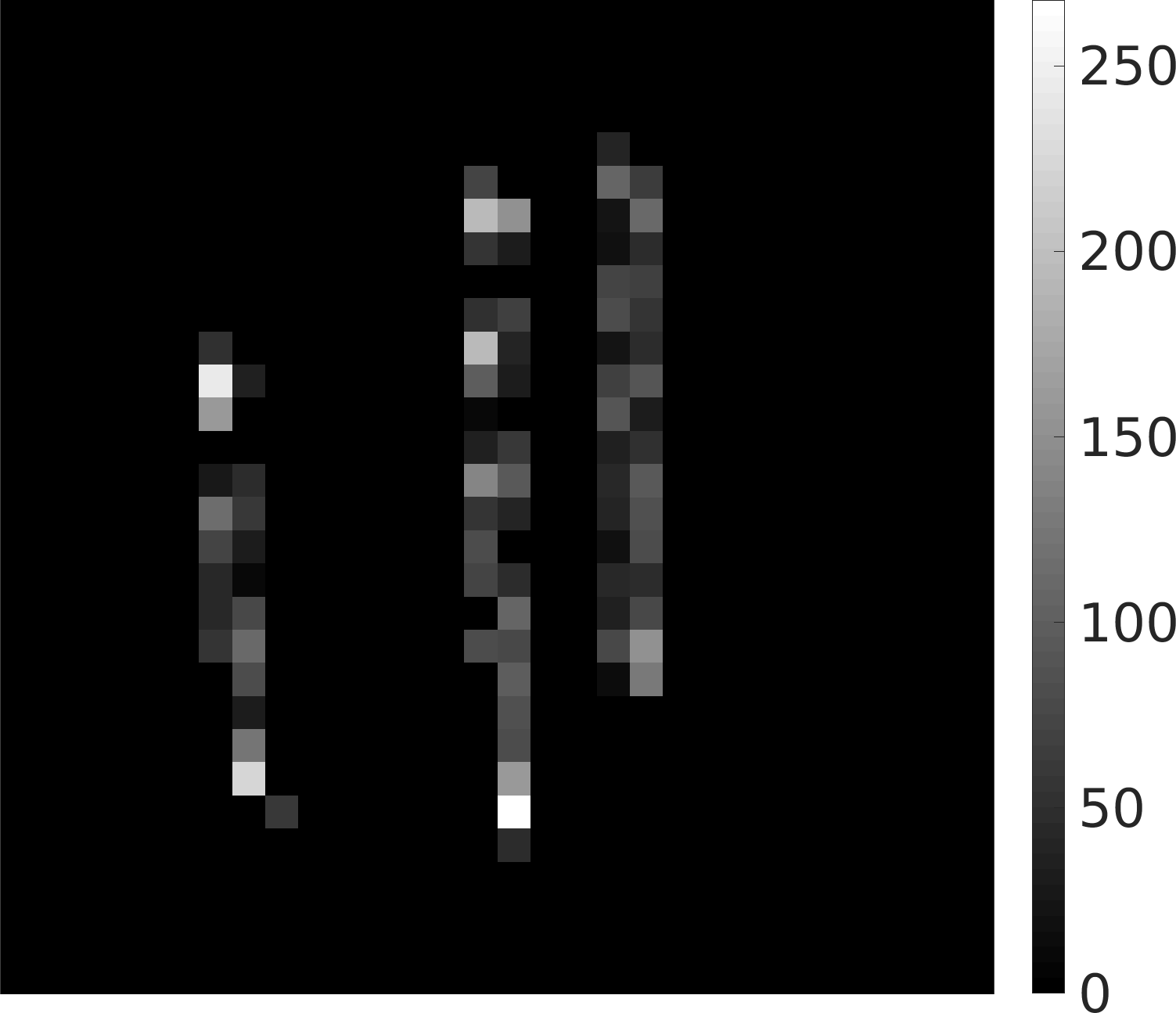}\\

\hline
7 &
\includegraphics[width=0.16\textwidth]{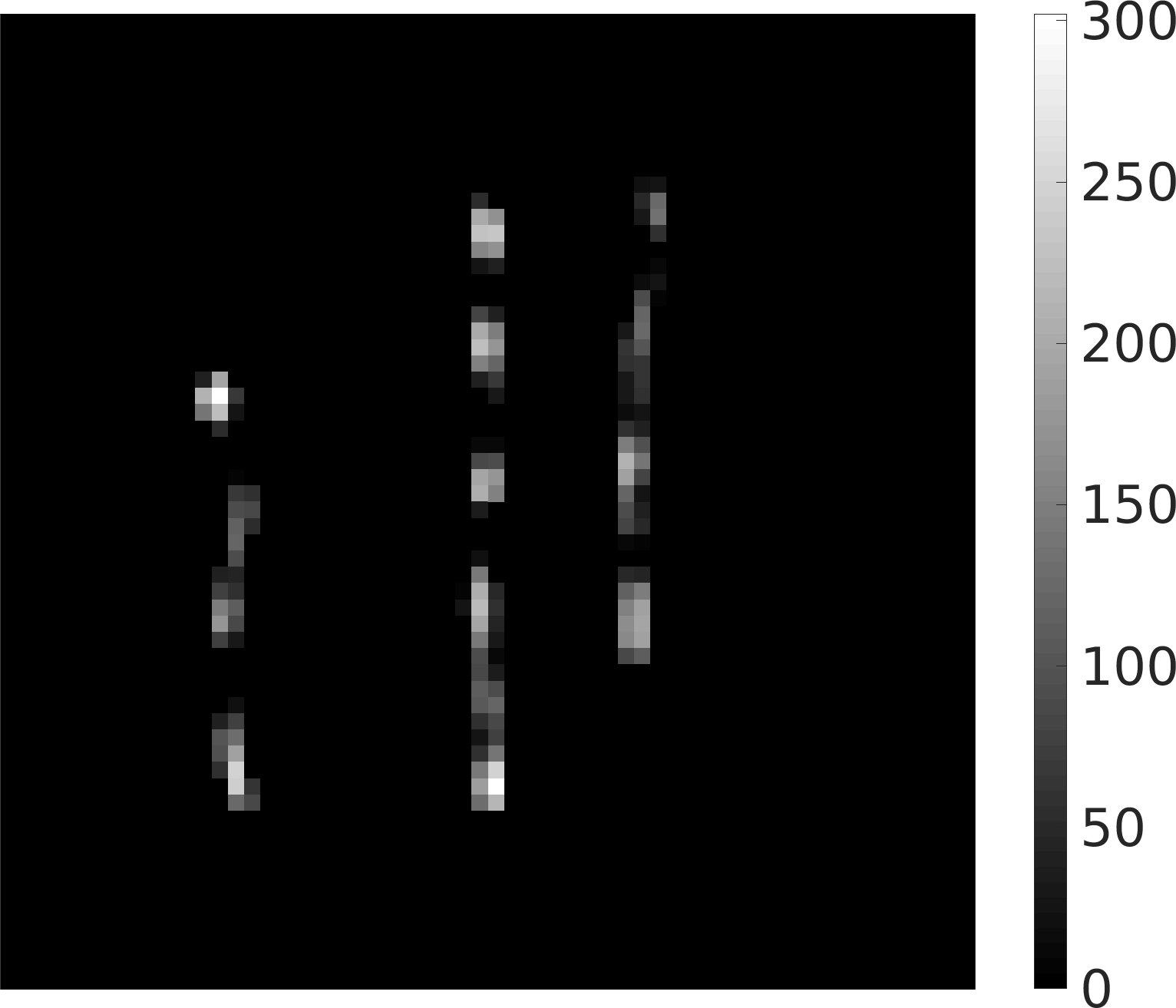}&
\includegraphics[width=0.16\textwidth]{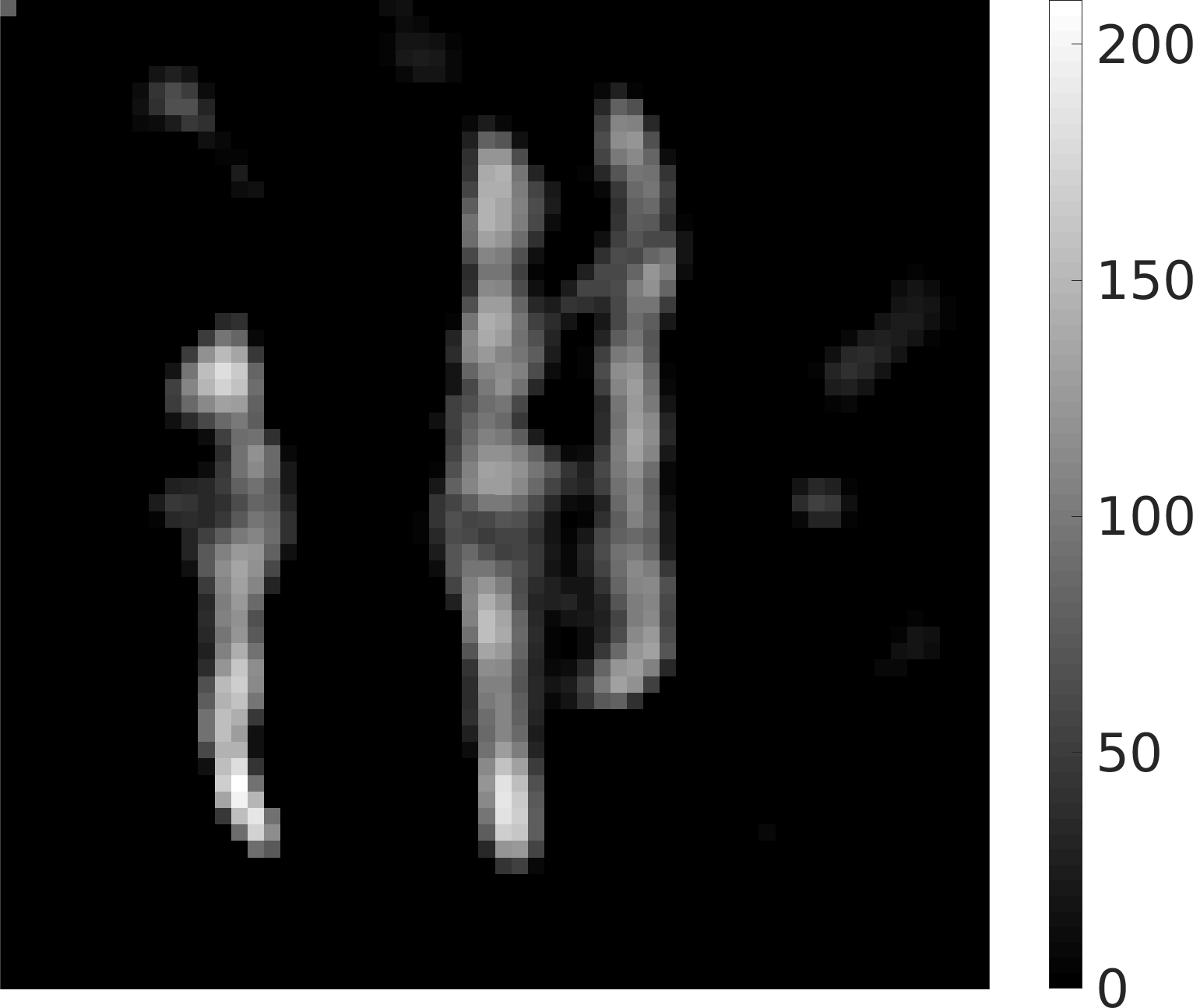}&
\includegraphics[width=0.16\textwidth]{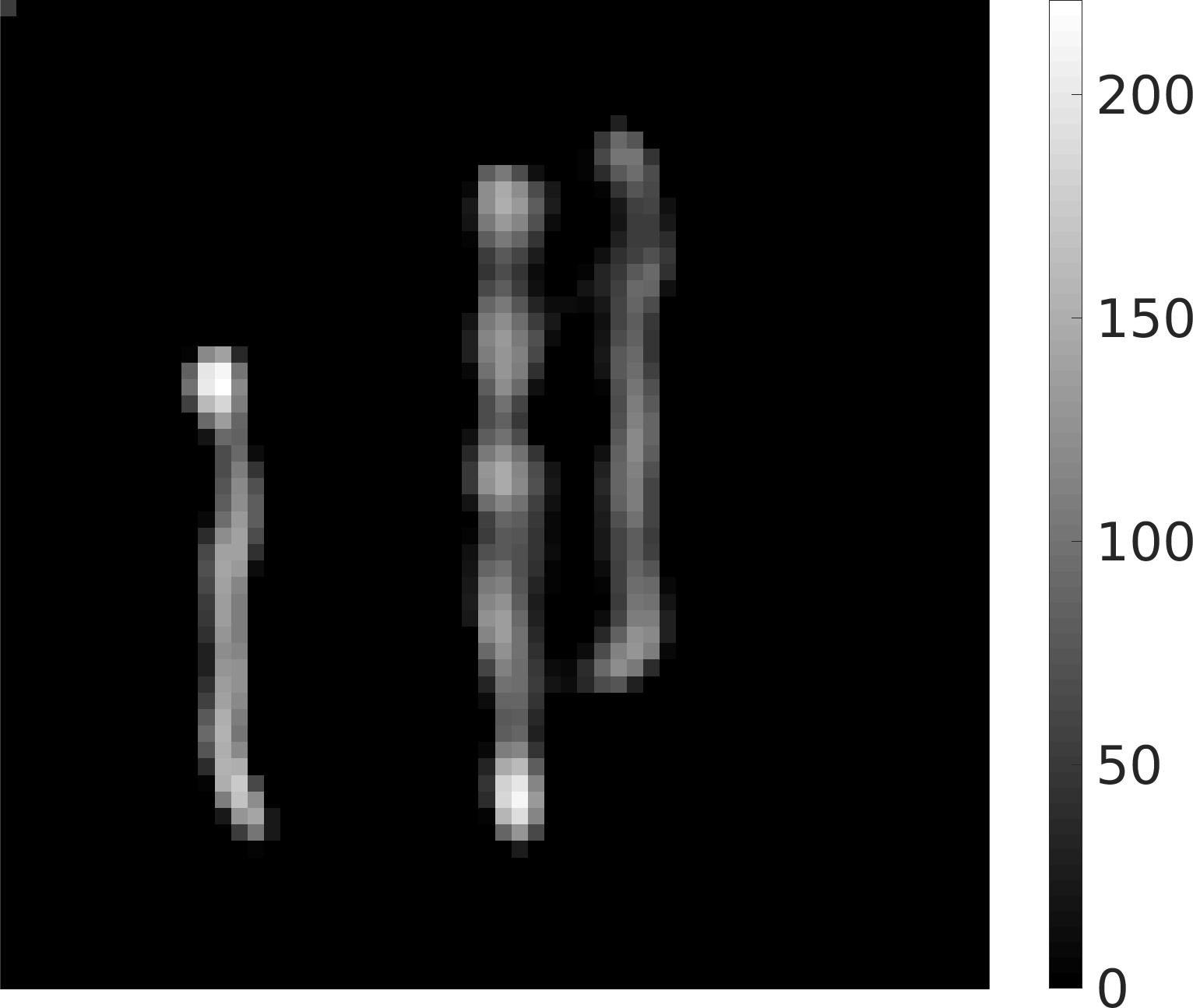}&
\includegraphics[width=0.16\textwidth]{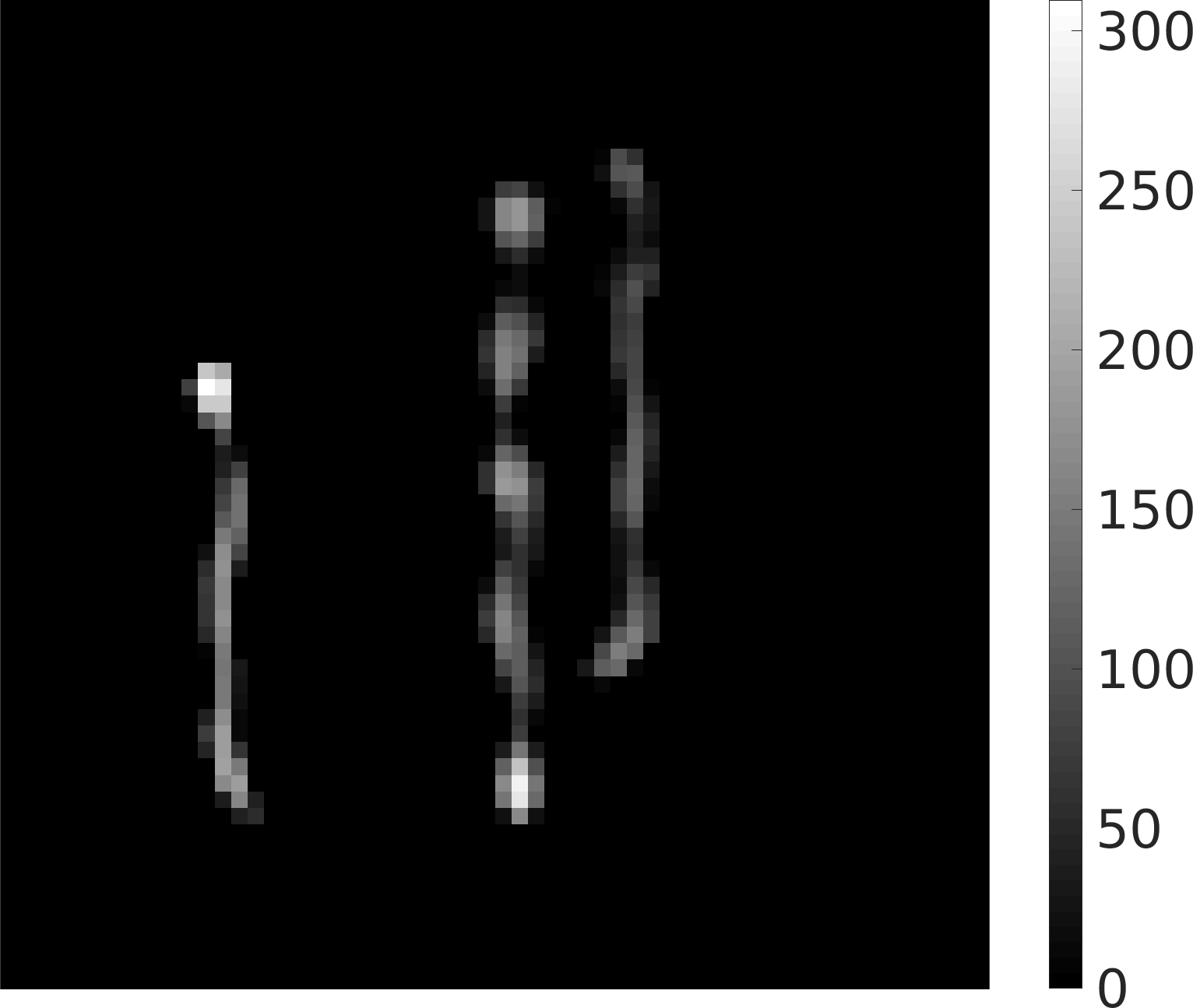}&
\includegraphics[width=0.16\textwidth]{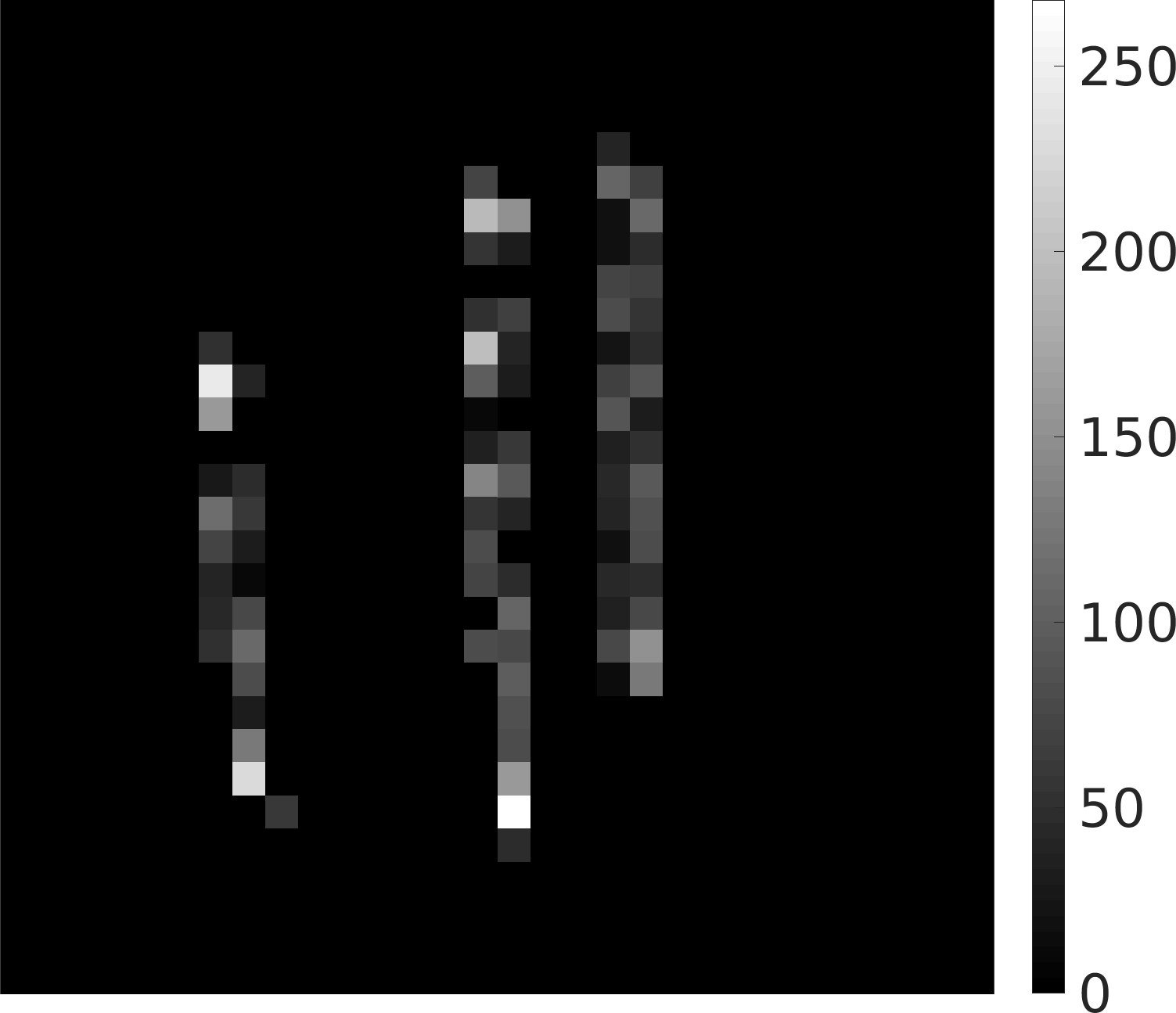}\\

\hline
8 &
\includegraphics[width=0.16\textwidth]{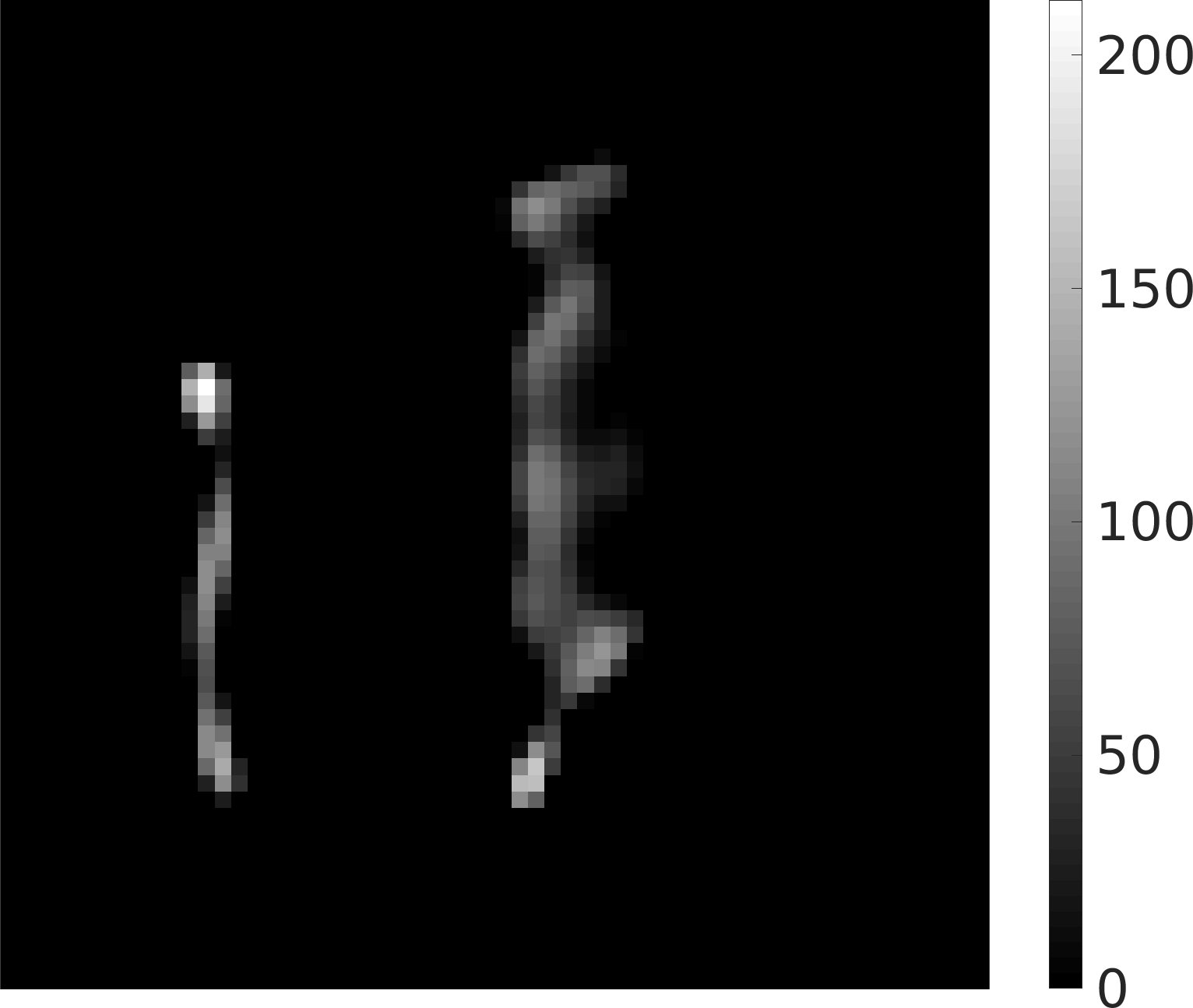}&
\includegraphics[width=0.16\textwidth]{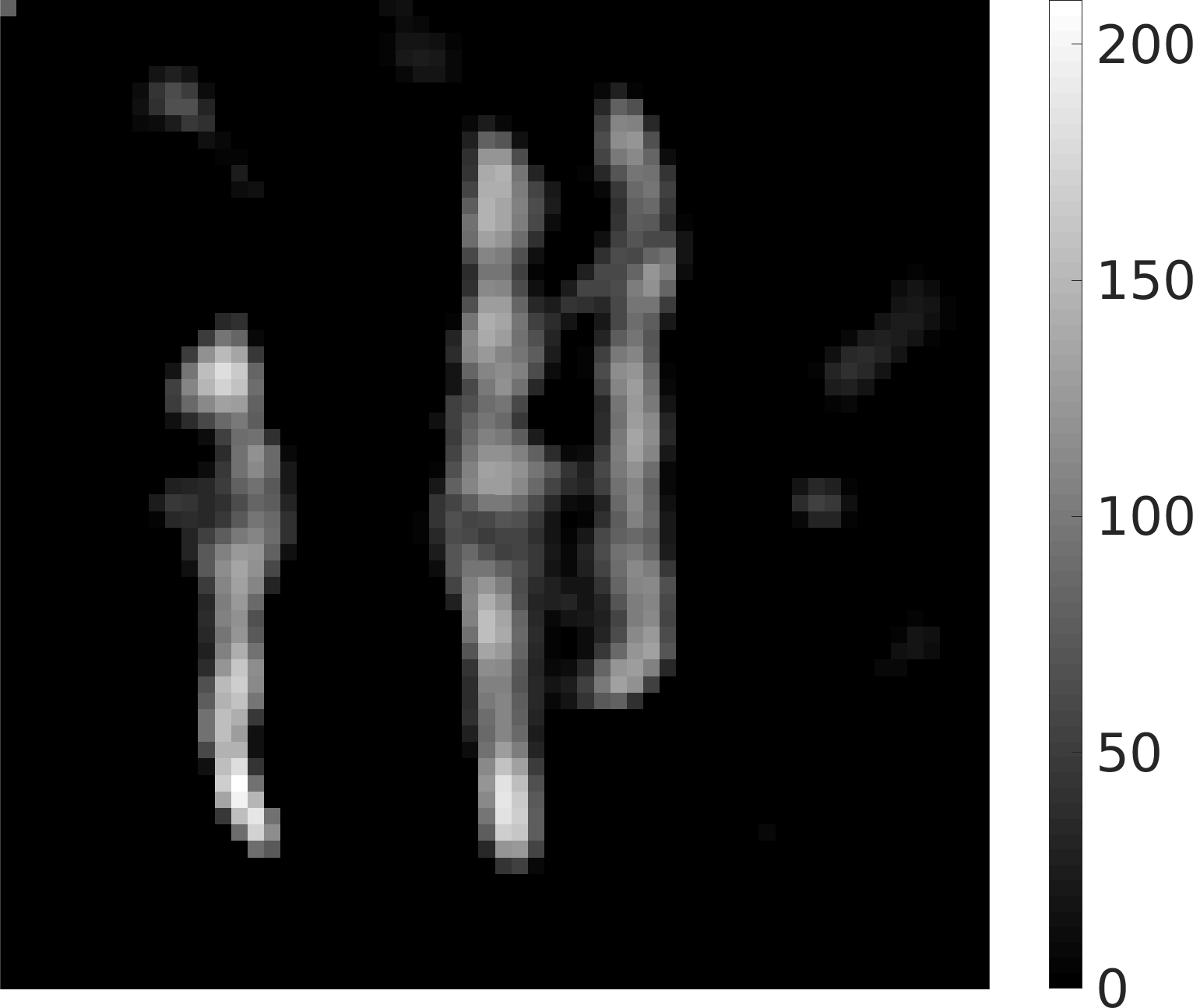}&
\includegraphics[width=0.16\textwidth]{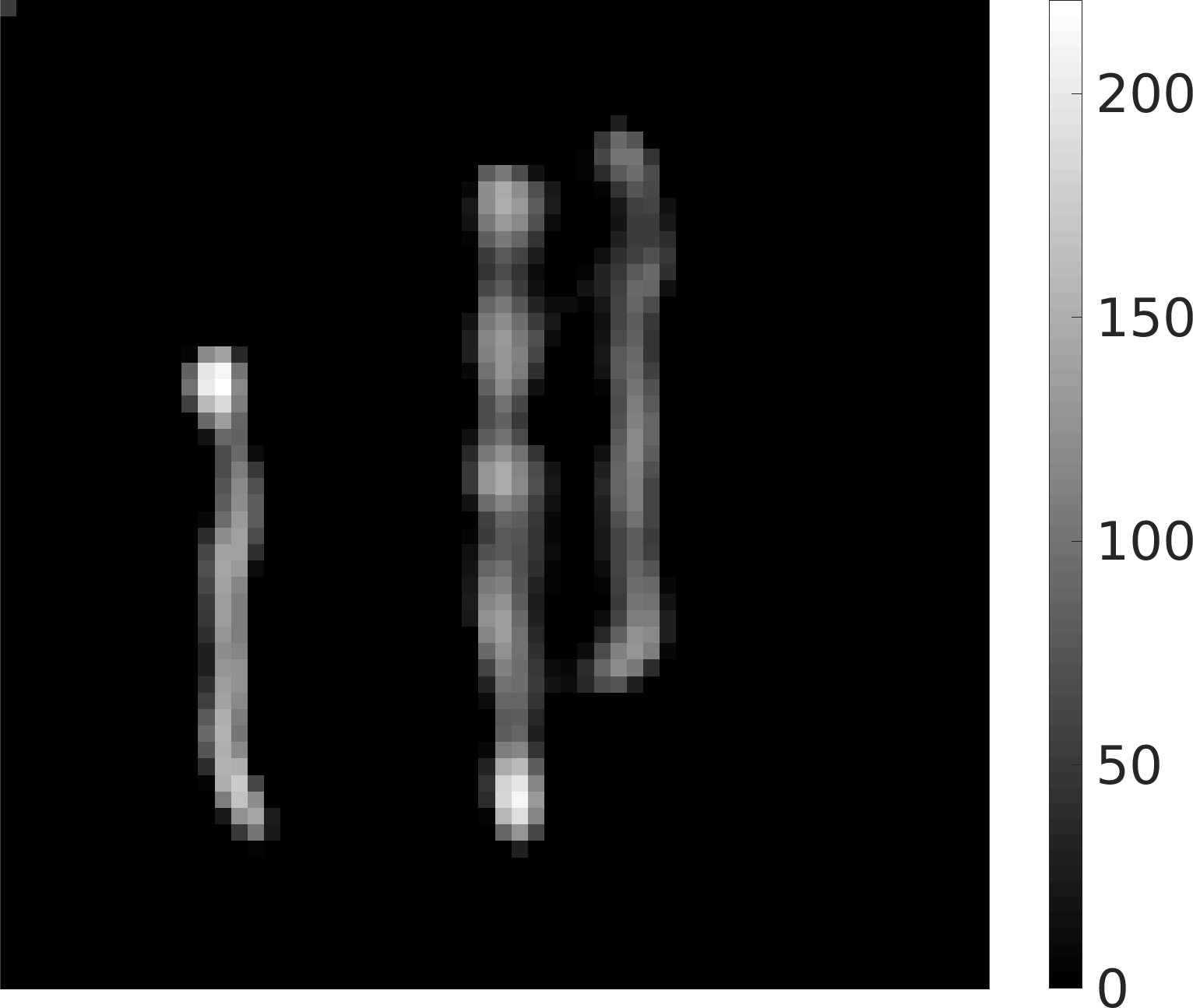}&
\includegraphics[width=0.16\textwidth]{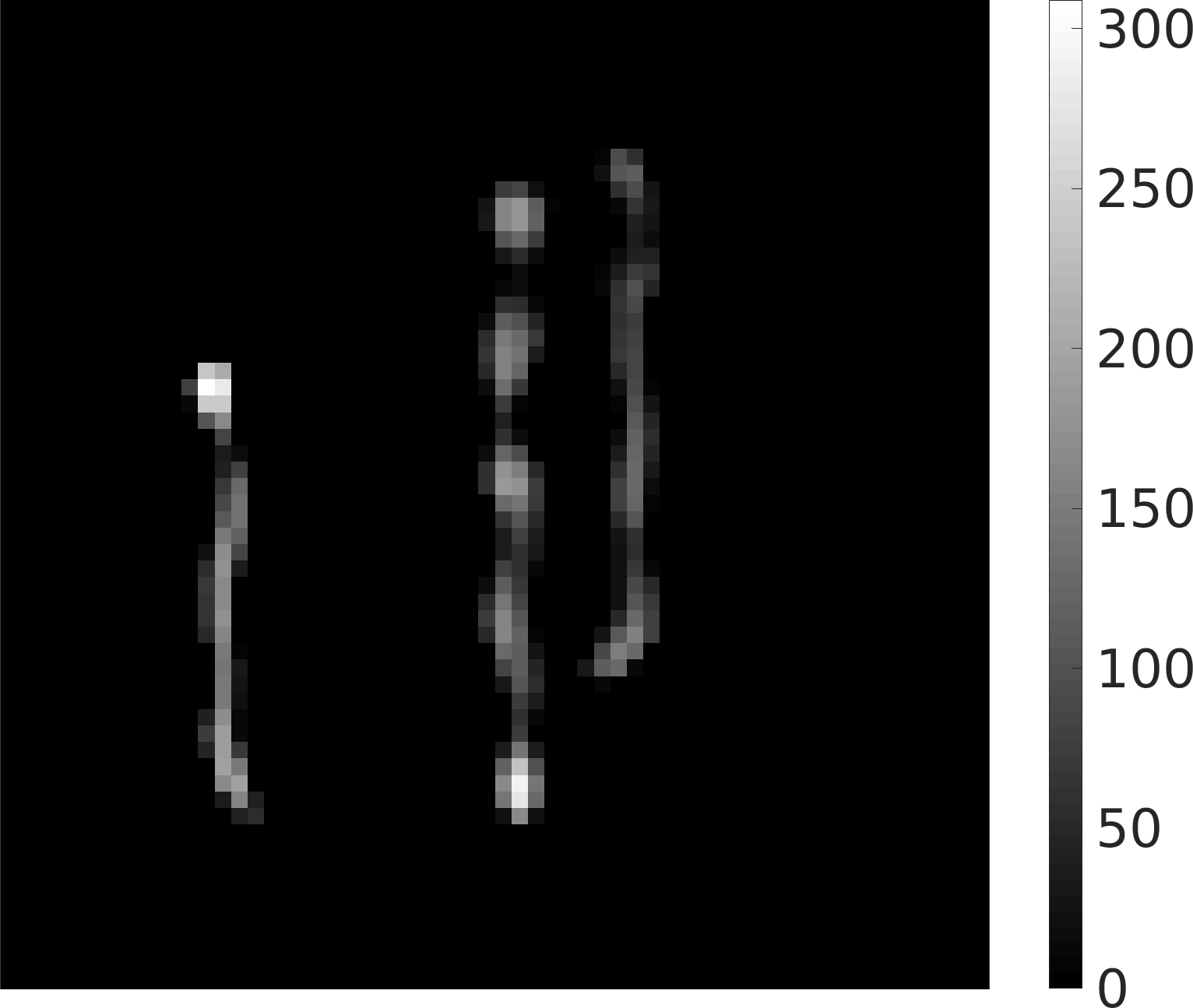}&
\includegraphics[width=0.16\textwidth]{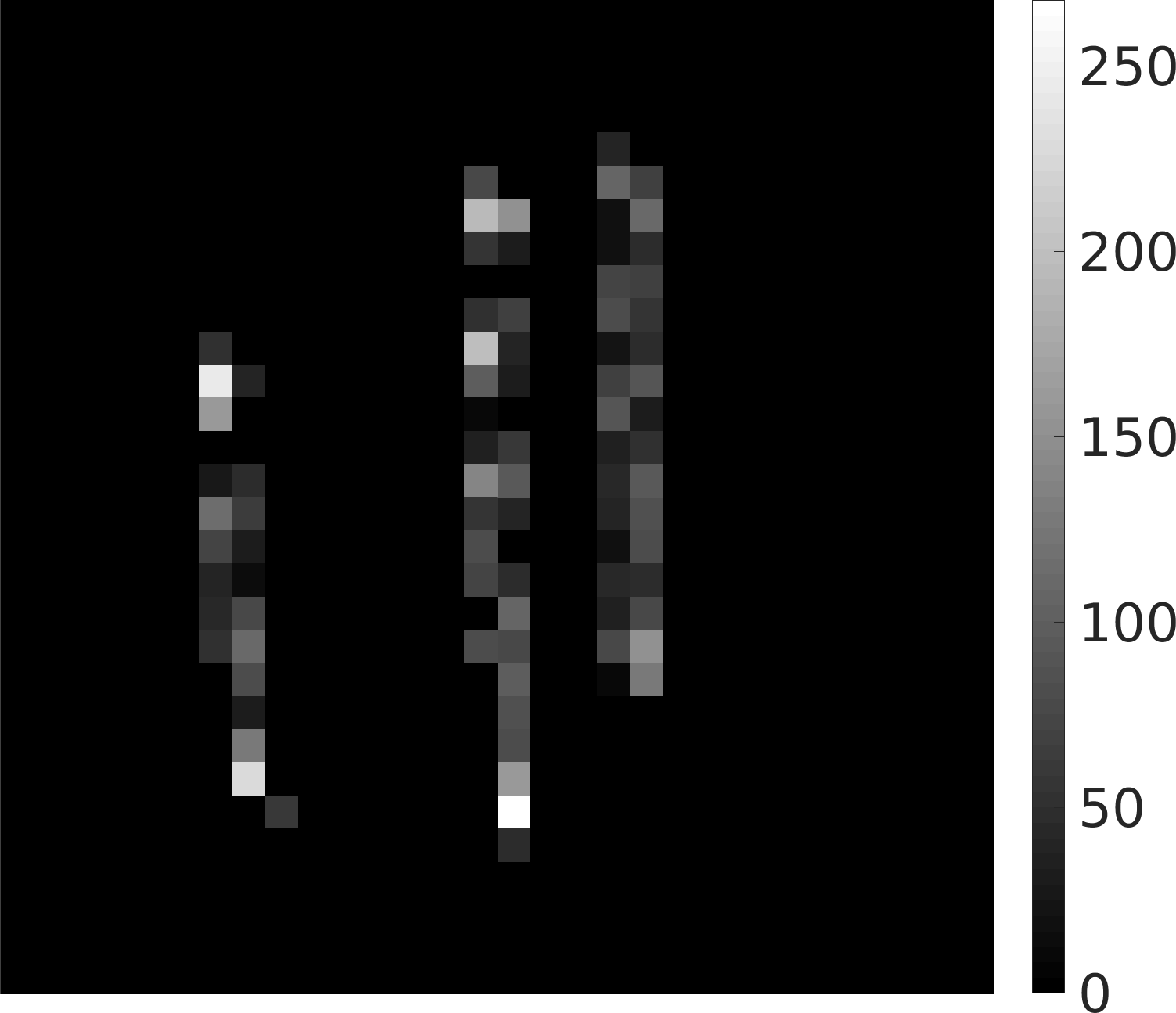}\\

%
\end{tabular}
    \caption{Concentration reconstructions for the phantom consisting of 3 glass capillaries (see \cite[Fig. 6]{KluthSzwargulskiKnopp2019}) are presented. All reconstructions are in mmol/l. The corresponding regularization parameters can be found in Table \ref{tab:MPI_reg_parameters}. Reconstructions are sorted by increasing absolute deviation from the expected tracer volume from top to bottom. }
    \label{fig:MPI_reconstructions}
\end{figure}

\begin{table}
\begin{tabular}{l|c|c}
    \multicolumn{3}{c}{\textbf{$c$-rec., $S=S_\mathrm{mod}$}} \\
    \hline
   \textbf{row no.}  & $\alpha$ &$\lambda$ \\
   \hline
   1 & $5.82 \times 10^{-11}$ & $6.25 \times 10^{-2}$ \\
      2 & $3.73 \times 10^{-9}$ & $3.12 \times 10^{-2}$ \\
       3 & $2.33 \times 10^{-10}$ & $6.25 \times 10^{-2}$ \\
             4 & $9.31 \times 10^{-10}$ & $3.12 \times 10^{-2}$ \\
                   5 & $3.73 \times 10^{-9}$ & $1.56 \times 10^{-2}$ \\
                    6 & $1.46 \times 10^{-11}$ & $6.25 \times 10^{-2}$ \\
                                        7 & $2.27 \times 10^{-13}$ & $1.25 \times 10^{-1}$ \\
                                            8 & $9.31 \times 10^{-10}$ & $6.25 \times 10^{-2}$ \\
                   
    \hline
\end{tabular}
\hspace{0.5cm}
\begin{tabular}{l|c|c}
    \multicolumn{3}{c}{\textbf{$c$-rec., $S=S_\mathrm{calib}$}} \\
    \hline
   \textbf{row no.}  & $\alpha$ &$\lambda$ \\
   \hline
   1 & $2.33 \times 10^{-10}$ & $6.25 \times 10^{-2}$ \\
      2 & $5.82 \times 10^{-11}$ & $6.25 \times 10^{-2}$ \\
      3 & $9.31 \times 10^{-10}$ & $1.25 \times 10^{-1}$ \\
          4 & $1.46 \times 10^{-11}$ & $6.25 \times 10^{-2}$ \\   
                    5 & $3.73 \times 10^{-9}$ & $3.12 \times 10^{-2}$ \\  
                              6 & $3.64 \times 10^{-12}$ & $6.25 \times 10^{-2}$ \\   
                                        7 & $9.09 \times 10^{-13}$ & $6.25 \times 10^{-2}$ \\   
                                                  8 & $2.27\times 10^{-13}$ & $6.25 \times 10^{-2}$ \\   
    \hline
\end{tabular}
\\ \vspace{0.5cm} \\
    \centering
    \begin{tabular}{l|c|c|c|c}
    \multicolumn{5}{c}{\textbf{$(c,S)$-rec.}} \\
    \hline
   \textbf{row no.}  & $\gamma$ & $\mu$ & $\alpha$ &$\lambda$ \\
    \hline
  1     & $1.0 \times 10^{-5}$ & $1.0$ & $3.64 \times 10^{-12}$ & $7.8\times 10^{-3}$ \\
    2    & $1.0 \times 10^{-4}$ & $1.0$ & $3.64 \times 10^{-12}$ & $7.8\times 10^{-3}$ \\
    3    & $1.0 \times 10^{-3}$ & $1.0$ & $3.64 \times 10^{-12}$ & $7.8\times 10^{-3}$ \\
        4    & $1.0 \times 10^{-2}$ & $1.0$ & $3.64 \times 10^{-12}$ & $7.8\times 10^{-3}$ \\
                5    & $1.0 \times 10^{-5}$ & $1.0\times 10^{-5}$ & $5.82 \times 10^{-11}$ & $7.8\times 10^{-3}$ \\
                6    & $1.0 \times 10^{-4}$ & $1.0\times 10^{-5}$ & $5.82 \times 10^{-11}$ & $7.8\times 10^{-3}$ \\
                                7    & $1.0 \times 10^{-3}$ & $1.0\times 10^{-5}$ & $5.82 \times 10^{-11}$ & $7.8\times 10^{-3}$ \\
                                                                8    & $1.0 \times 10^{-4}$ & $1.0\times 10^{-4}$ & $5.82 \times 10^{-11}$ & $7.8\times 10^{-3}$ \\
    \end{tabular}
    \caption{Regularization parameters used to obtain the image reconstructions in Figure \ref{fig:MPI_reconstructions}}
    \label{tab:MPI_reg_parameters}
\end{table}

\section{Discussions and concluding remarks}
In this work, we considered a hybrid approach to obtain high resolution reconstructions in imaging applications by explicitly taking into account parameters of the imaging operators in the reconstruction process.
The present approach combines incomplete infinite- or high-dimensional model information (type A) with high-quality but finite-/lower-dimensional information (type B). Motivated by the application of interest, MPI, we analyzed a general Hilbert space setup for bilinear operators, i.e., a linear imaging operator as well as a linear dependence on the parameters of the imaging operator. 
Furthermore we derived a Kaczmarz-type algorithm to obtain a solution for the joint reconstruction problem and tested it in an academic problem as well as in the imaging application of MPI.

The theoretical findings in terms of stability, convergence and convergence rates are in line with the findings in \cite{Bleyer:2013cw} where the authors considered bilinear operators fulfilling stronger assumptions and a special case of the functional of the present work. In contrast to the work \cite{Bleyer:2013cw} we exploited the general results for nonlinear operator equations in \cite{Hofmann:2007us} by addressing the product space setup of the joint reconstruction problem.

An algorithmic solution to the joint reconstruction problem is derived using an alternating minimization approach as analogously formulated for a special case of the functional \cite{Bleyer_2015}. Motivated by the application of interest, Kaczmarz-type algorithms are exploited to minimize the respective functional for the image and the parameter reconstruction. The extension taking into account the link between low and high resolution system matrices is straight forward and it showed to be successful as illustrated by the academic test example.         
Furthermore, the numerical results for MPI illustrate the potential of the present approach to exploit multiple information sources to comprise best of both worlds in hybrid methods.

The obtained results of this work build the basis for several directions of research in different disciplines. 
In the context of MPI the present work motivates different future research questions. In particular, an experimental study where ground truth of the phantom is available is desirable which also enables the investigation of suitable image quality measures.

From a theoretical as well as an application point of view the extension to a nonlinear dependence on the parameters of the imaging operator is highly desirable. 
The interpretation of the linear operator $P$ linking the low and high resolution is then not as intuitive as in the present setup anymore. 
It can then rather be seen as a joint model calibration and image reconstruction problem. 
An intuitive further direction of research is the treatment of dynamically changing model parameters in time-dependent image reconstruction. In MPI, for example, the tracer material changes its magnetization behavior if the nanoparticles are immobilized \cite{Kluth2018a,M_ddel_2018}, e.g., if the nanoparticles are blocked while labeling certain types of tissue \cite{Wu2019}. 

\section*{Acknowledgements}
The authors would like to thank T. Knopp  from University Medical Center Hamburg-Eppendorf for sharing the MPI data used in this work.
T. Kluth acknowledges funding by the Deutsche Forschungsgemeinschaft
(DFG, German Research Foundation) - project number 281474342/GRK2224/1 ``Pi$^3$ : Parameter Identification
- Analysis, Algorithms, Applications''. 
C. Bathke acknowledges funding by the project ``MPI$^2$'' funded by the Federal Ministry of Education and Research (BMBF, project no.
05M16LBA). P. Maass acknowledges the support by the Federal Ministry of Education and Research (BMBF project no. 05M20LBC, HYDAMO). M. Jiang acknowledges the funding by  the National Science Foundation of China (11961141007, 61520106004), and the friendly hospitality of Prof. Peter Maass during his sabbatical visit to University of Bremen.

\bibliographystyle{abbrv}
\bibliography{literature}
 \appendix

\section{Proofs of Section \ref{sec:basic_properties}}
\label{app:appendix1}
We now provide a proof of Lemma \ref{lemma:projectionPenalty}, which is a classical result from linear operator theory.
\begin{proof}
The definition of the functional $T$ and the properties of $P$ yield that $T$ is proper. Due to the linearity of $P$ and the triangle inequality, it holds for any $a\in (0,1)$ and $x,y \in Y$
	\begin{align*}
		f((1-a)x+ay) &= \| P((1-a)x+ay)- s_{\mathrm{calib}}\| \\
		&= \| P ((1-a)x) - (1-a) s_{\mathrm{calib}}+ P(ay)- a s_{\mathrm{calib}}\| \\
		&\leq  (1-a)\|P (x) - s_{\mathrm{calib}}\|+ a \|P(y)-  s_{\mathrm{calib}}\| \\
		&= (1-a)f(x) +a f(y)
	\end{align*}
	with $f(s):= \| P(s)-s_{\mathrm{calib}}\|$. Since $f$ is convex, also $f^2\equiv T$ is convex. The weak lower semi-continuity of the functional $T$ follows from the weak lower semi-continuity of the Hilbert space norm and the continuity of $P$.
\end{proof}
We now include a proof of Lemma \ref{lemma:FrechetDerivative1}, which also follows by standard arguments.
	\begin{proof}
	Using the linearity of $B$ in both arguments, we obtain
	\begin{align*}
		\lVert &B(c+x,s+y) - B(c,s) - B(c,y)-B(x,s) \rVert  = \lVert B(x,y) \rVert \leq C \lVert x \rVert \lVert y \rVert \leq \frac{C}{2} \lVert (x,y) \rVert^2 
	\end{align*}
It thus follows
	\begin{align*}
		\lim_{\lVert (x,y) \rVert_{X\times Y} \to 0}
		\frac{\lVert B(c+x,s+y)-B(c,s)-B^\prime(c,s)(x,y)\rVert_Z}{\lVert (x,y) \rVert_{X\times Y}} &= 0
	\end{align*}	
	for  $B^\prime(c,s)(x,y) = B(x,s) + B(c,y)$.
		\end{proof}

\section{Proofs of Section \ref{sec:setupA2}}
\label{app:appendix2}

Proof of Theorem \ref{thm:existence_etc_2}.
\begin{proof}
We consider the tuple $(c,s)\in H_1\coloneqq X \times Y$ where $H_1$ also is a Hilbert space equipped with the canonical inner product
$	\langle x_1,x_2 \rangle_{H_1} \coloneqq \langle c_1,c_2 \rangle_X + \langle s_1,s_2 \rangle_Y$.
The discrepancy term in $\J$ contains two terms such that we similarly consider the space $H_2 \coloneqq Z \times Y$ and the data tuples $y_{\eta(\delta,\epsilon)} = (u_\delta,\sqrt{\gamma} s_{\mathrm{mod},\epsilon})$ and $y^\ast = (u^\ast, \sqrt{\gamma} s^\ast)$ with noise estimate 
\begin{equation}
	\lVert y_{\eta(\delta,\epsilon)} - y^\ast \rVert^2_{H_2} = \lVert u_\delta - u^\ast \rVert^2_Z + \lVert s_{\mathrm{mod},\epsilon} - s^\ast \rVert^2_Y \leq 
	\delta^2 + \gamma \epsilon^2 \eqqcolon \eta^2(\delta,\epsilon).
\end{equation}
Considering the operator $F: H_1 \rightarrow H_2$, $(c,s) \mapsto (B(c,s), \sqrt{\gamma} s)$, which is bilinear and weakly continuous in the first component and linear and weakly continuous in the second one.
It thus holds
\begin{equation}
	\lVert F(c,s) - y_{\eta(\delta,\epsilon)} \rVert^2_{H_1}
	= \lVert B(c,s) - u_\delta \rVert^2_Z + \gamma \lVert s - s_{\mathrm{mod},\epsilon}\rVert^2_Y
\end{equation}
which allows to rewrite the functional $\J$ as
\begin{align} \label{eq:scherzer}
	J^\eta_\alpha (c,s)
		&:= \frac{1}{2} \| F(c,s)-y_{\eta(\delta,\epsilon)} \|_Z^2 + \alpha {\mathcal{R}_2}(c,s) \\
		\text{ and let }\tilde{\mathcal{R}}(c,s) &:=
	 \penc(c) + \frac{\nu_2}{2} \|P(s)-s_\mathrm{calib} \|^2_Y + \nu_1 \pens(s).
\end{align}
As the assertions (i) and (ii) would follow immediately from \cite[Theorems 3.1, 3.2]{Hofmann:2007us}, it is sufficient to show the following \cite[Assumption 2.1]{Hofmann:2007us}:
	\begin{itemize}
	\item[(a)] $H_1$ and $H_2$ are Banach spaces with associated topologies $\tau_{H_1}$ and $\tau_{H_2}$ weaker than the norm topologies.
	\item[(b)] $\|\cdot\|_{H_2}$ is sequentially lower semi-continuous with respect to $\tau_{H_2}$.
	\item[(c)] $F\colon \D(F) \subset H_1 \rightarrow H_2$ is continuous with respect to the topologies $\tau_{H_1}$ and $\tau_{H_2}$.
	\item[(d)] $\tilde{\mathcal{R}}\colon H_1 \rightarrow [0,\infty]$ is proper, convex and $\tau_{H_1}$ lower semi-continuous.
	\item[(e)] $\D(F)$ is closed with respect to $\tau_{H_1}$ and $\D\coloneqq \D(F)\cap\D(\tilde{R})\neq\emptyset$.
	\item[(f)] For every $\alpha >0$ and $M>0$ the level sets
	\begin{equation}
		\mathcal{M}_\alpha(M) \coloneqq \{(c,s)\in \D \colon J^\eta_\alpha(c,s) \leq M \}
	\end{equation}
	are $\tau_{H_1}$ sequentially compact.	
	\end{itemize}
As $X,Y$ and $Z$ are Hilbert spaces with weak topologies, $H_1$ and $H_2$ are also Hilbert spaces with weak topologies such that (a) and (b) are fulfilled. 
From weak continuity of $B$ and weak continuity in the second component of $F$ it follows (c). 
The definition of $\mathcal{R}_c$ and $\mathcal{R}_s$ and Lemma \ref{lemma:projectionPenalty} yield (d).
Since $\D(F)=H_1$ and the penalty terms are proper, (e) is fulfilled as well.
(f) is equivalent to the weak lower semi-continuity of the whole functional, which results from (d), the weak continuity of $B$, and that the discrepancy term $D_2$ is defined by weakly continuous Hilbert space norms.

Assertion (iii) can be proved analogously to \cite[Theorems 3.5]{Hofmann:2007us} with some minor adaptations taking into account the properties of the parameter sequences $\alpha_j,\beta_j,\mu_j$ which yield $\alpha_j \tilde{\mathcal{R}}(c,s) \leq (\alpha_j,\beta_j,\mu_j)^t \mathcal{R}_2(c,s)$ for any $(c,s)\in H_1$. With $\eta_j:=\eta(\delta_j,\epsilon_j)$ from the definition of $(c^j,s^j)_j$ it follows
\begin{equation*}
    \frac12 \|F(c^j,s^j)-y_{\eta_j}\|_{H_2}^2 + (\alpha_j,\beta_j,\mu_j)^t \mathcal{R}_2(c^j,s^j) \leq \frac12 \eta_j^2 + (\alpha_j,\beta_j,\mu_j)^t \mathcal{R}_2(c^\ast,s^\ast) 
\end{equation*}
which shows $\lim_{j\to \infty} \| F(c^j,s^j) - y^\ast \|_{H_2}=0$ and yields 
\begin{equation*}
    \tilde{\mathcal{R}}(c^j,s^j) \leq \frac{\eta_j^2}{2\alpha_j} + \frac{(\alpha_j,\beta_j,\mu_j)^t\mathcal{R}_2(c^\ast,s^\ast)}{\alpha_j}
\end{equation*}
implying $\limsup_{j\to \infty}{\tilde{\mathcal{R}}(c^j,s^j)}\leq \limsup_{j\to \infty}{(1,\beta_j/\alpha_j,\mu_j/\alpha_j)^t\mathcal{R}_2(c^j,s^j)} \leq \tilde{\mathcal{R}}(c^\ast,s^\ast) $.
We thus obtain
\begin{align*}
    \limsup_{j\to \infty}{\left(\frac12 \|F(c^j,s^j)-y_{\eta_j}\|_{H_2}^2 + \alpha_0 \tilde{\mathcal{R}}(c^j,s^j)\right)}& \\
    \leq \limsup_{j\to \infty}{\left( \frac12 \|F(c^j,s^j)-y_{\eta_j}\|_{H_2}^2  + \alpha_j \tilde{\mathcal{R}}(c^j,s^j)\right)} &+ \limsup_{j\to \infty}{\left( ( \alpha_0 -\alpha_j  )\tilde{\mathcal{R}}(c^j,s^j)  \right) } \leq \alpha_0 \tilde{\mathcal{R}}(c^\ast,s^\ast) < \infty 
\end{align*}
Existence of a weakly convergent subsequence of $(c^j,s^j)_j$ and that the limit of each weakly convergent subsequence is an $\tilde{\mathcal{R}}$ minimizing solution can be derived following the remaining steps in the proof of \cite[Theorems 3.5]{Hofmann:2007us}.

Weak convergence of the subsequence in the $s$ component to $s^\ast$ follows immediately from the definition of the $\tilde{\mathcal{R}}$-minimizing solution with respect to $F$ which concludes the proof.
\end{proof}

\end{document}